\documentclass[11pt,reqno]{amsart}    

\usepackage[margin=37mm]{geometry}
\usepackage{soul}

\usepackage{amssymb,amsmath,amsfonts,amsthm, amscd, mathrsfs, helvet, mathtools}
\usepackage{bm, stmaryrd}
\usepackage{setspace}
\usepackage{enumerate}
\usepackage{graphicx,verbatim}
\usepackage[usenames, dvipsnames]{color}
\usepackage{color}
\usepackage[mathcal]{euscript}
\usepackage{frcursive}
\usepackage[T1]{fontenc}
\usepackage{hyperref}

\usepackage{tikz}
\usetikzlibrary{calc}
\usetikzlibrary{arrows,matrix,snakes}
\usetikzlibrary{decorations.markings,shapes.geometric,shapes.misc}    

\theoremstyle{plain}
\newtheorem{theorem}{Theorem}[section]
\newtheorem*{theorem*}{Theorem}
\newtheorem{proposition}[theorem]{Proposition}
\newtheorem*{proposition*}{Proposition}

\newtheorem{lemma}[theorem]{Lemma}
\newtheorem{corollary}[theorem]{Corollary}
\newtheorem{conjecture}[theorem]{Conjecture}

\theoremstyle{remark}

\newtheorem{remarkx}[theorem]{Remark}

\newenvironment{remark}
  {\pushQED{\qed}\remarkx}
  {\popQED\endremarkx}

\definecolor{brightRed}{rgb}{1,0,0}

\definecolor{brightBlue}{rgb}{0,0,1}

\definecolor{myGreen}{rgb}{0,0.7,0}

\newcommand{\tensor}[1]{{\mathfrak{#1}}}

\newcommand{\nord}[1]{:\! #1 \!:}

\DeclareMathOperator{\ad}{ad}

\DeclareMathOperator{\End}{End}

\DeclareMathOperator{\hgt}{ht}
\DeclareMathOperator{\rank}{rank}
\DeclareMathOperator{\Conn}{Conn}
\DeclareMathOperator{\Proj}{Proj}
\DeclareMathOperator{\op}{op}
\DeclareMathOperator{\Op}{Op}
\DeclareMathOperator{\MOp}{MOp}

\DeclareSymbolFont{widetriangleaccent}{OMX}{yhex}{m}{n}
\DeclareMathAccent{\widetriangle}{\mathord}{widetriangleaccent}{"E6}
\newcommand{\longhookrightarrow}{\lhook\joinrel\relbar\joinrel\rightarrow}
\newcommand{\longtwoheadrightarrow}{\relbar\joinrel\twoheadrightarrow}

\def\cent{\mathsf k}
\def\cocent{\mathsf d}
\def\a{\mathfrak{a}}
\def\b{\mathfrak{b}}
\def\c{\mathfrak{c}}

\def\g{\mathfrak{g}}
\def\lg{{{}^L\!\g}}
\def\lgp{{}^L\!\g'}
\def\ln{{}^L\n}
\def\lhn{{}^L \hat\n}
\def\lhb{{}^L \hat\b}
\def\lhg{{}^L \hat\g}
\def\lhN{{}^L \! \hat N}
\def\lN{{}^L \! N}
\def\lb{{}^L\b}
\def\lh{{}^L\h}
\def\lhB{{}^L \! \hat B}
\def\lB{{}^L \! B}
\def\lH{{}^L \! H}
\def\lG{{}^L \! G}
\newcommand{\on}{}

\newcommand{\be}{\begin{equation}}
\newcommand{\ee}{\end{equation}}
\newcommand{\nn}{\nonumber}

\def\h{\mathfrak{h}}

\def\n{\mathfrak{n}}
\def\p{\mathfrak{p}}

\def\ZZ{\mathbb{Z}}

\def\ha{\mbox{\small $\frac{1}{2}$}}

\newcommand{\SimTo}{%
\xrightarrow{\raisebox{-.3em}{\tiny $\sim$}}%  \xrightarrow{\raisebox{-0.3 em}{\smash{\ensuremath{\sim}}}}%
}

\newcommand{\bb}[1]{[\kern-.1em[ #1 ]\kern-.1em]}%{\llbracket #1 \rrbracket}
\newcommand{\lau}[1]{(\kern-.2em( #1 )\kern-.2em)}%{(( #1 ))}

\newcommand{\lbf}[1]{\langle\kern-.2em\langle #1 \rangle\kern-.2em\rangle}%{(\kern-.2em( #1 )\kern-.2em)}%{(( #1 ))}
\newcommand{\biglbf}[1]{\big\langle \kern-.25em \big\langle #1 \big\rangle \kern-.25em \big\rangle}%{(( #1 ))}

\newcommand{\rbf}[1]{(\kern-.2em( #1 )\kern-.2em)}%{(\kern-.2em( #1 )\kern-.2em)}%{(( #1 ))}
\newcommand{\bigrbf}[1]{\big( \kern-.25em \big( #1 \big) \kern-.25em \big)}%{(( #1 ))}

\newcommand{\wt}[1]{\widetilde #1}

\def\A{\mathcal{A}}

\def\C{\mathcal{C}}
\def\CC{\mathbb{C}}
\def\CP{\mathbb{P}^1}
\def\RR{\mathbb{R}}
\def\H{\mathcal{H}}

\def\F{\mathcal{F}}
\def\K{\mathcal{K}}

\def\P{\mathcal{P}}

\def\M{\mathcal{M}}

\def\sl{\mathfrak{sl}}

\newcommand{\Cx}{\mathbb C^\times}

\newcommand{\ru}{r}
\newcommand{\8}{{\infty}}    
\newcommand{\onto}{\twoheadrightarrow}
\newcommand{\cc}{\mathscr C}
\newcommand{\into}{\hookrightarrow}
\newcommand{\bl}{{\bullet}}    
\newcommand{\mc}{\mathcal}
\let\nc\newcommand
\newcommand{\del}{\partial}    
\newcommand{\hc}{h^\vee}
\newcommand{\reg}{\mathrm{reg}}

\let\ox\otimes
\let\wx\wedge

\nc\ket{\rangle\,}

\def\1{\tensor{1}}
\def\2{\tensor{2}}
\def\3{\tensor{3}}
\def\4{\tensor{4}}

\numberwithin{equation}{section}

\author{Sylvain Lacroix}
\author{Beno\^{\i}t Vicedo}
\author{Charles Young}
\address{Univ Lyon, Ens de Lyon, Univ Claude Bernard, CNRS, Laboratoire de Physique, F-69342 Lyon, France}
\address{Department of Mathematics, University of York, York YO10 5DD, U.K.}
\address{School of Physics, Astronomy and Mathematics, University of Hertfordshire, College Lane, Hatfield AL10 9AB, UK.}

\email{sylvain.lacroix@ens-lyon.fr}
\email{benoit.vicedo@gmail.com}  
\email{c.a.s.young@gmail.com}

\begin{document}

\title[Affine Gaudin models and hypergeometric functions on affine opers]{Affine Gaudin models and hypergeometric\\
functions on affine opers}

\input{epsf}

\begin{abstract}
We conjecture that quantum Gaudin models in affine types admit families of local higher Hamiltonians, labelled by the (countably infinite set of) exponents, whose eigenvalues are given by functions on a space of meromorphic opers associated with the Langlands dual Lie algebra. This is in direct analogy with the situation in finite types. However, in stark contrast to finite types, we prove that in affine types such functions take the form of hypergeometric integrals, over cycles of a twisted homology defined by the levels of the modules at the marked points. That result prompts the further conjecture that the Hamiltonians themselves are naturally expressed as such integrals.

We go on to describe the space of meromorphic affine opers on an arbitrary Riemann surface. We prove that it fibres over the space of meromorphic connections on the canonical line bundle $\Omega$. Each fibre is isomorphic to the direct product of the space of sections of the square of $\Omega$ with the direct product, over the exponents $j$ not equal to 1, of the twisted cohomology of the $j^{\rm th}$ tensor power of $\Omega$.
\end{abstract}

\maketitle
\setcounter{tocdepth}{1}
\tableofcontents

\section{Introduction and overview} \label{sec: intro}
Let $\g$ be any symmetrizable Kac-Moody Lie algebra. Pick a collection $z_1,\dots,z_N$ of distinct points in the complex plane. The quadratic Hamiltonians of the quantum \emph{Gaudin model} for these data are the elements  
\begin{equation} \label{quad Ham intro}
\H_i \coloneqq \sum_{\substack{j=1\\j\neq i}}^N \frac{\Xi_{ij}}{z_i-z_j} 
          ,\qquad i=1,\dots,N,
\end{equation}
of the (suitably completed) tensor product $U(\g)^{\ox N}$, where the notation $\Xi_{ij}$ means $\Xi$ acting in tensor factors $i$ and $j$. 
Here $\Xi = \sum_{\alpha} \Xi_{(\alpha)}$ is the (possibly infinite) sum over all root spaces of $\g$ of the canonical elements $\Xi_{(\alpha)}\in \g_\alpha \ox \g_{-\alpha}$ defined by the standard bilinear form on $\g$ \cite[Chapter 2]{KacBook}. The action of $\Xi$ is well-defined on tensor products of highest-weight $\g$-modules. Let $L_{\lambda}$ denote the irreducible $\g$-module of highest weight $\lambda\in \h^* = \g_0^*$, and pick a collection $\lambda_{1},\dots,\lambda_N$ of weights. Then in particular $\H_i$ are well-defined as linear maps in $\End(\bigotimes_{i=1}^N L_{\lambda_i})$. These maps commute amongst themselves. The \emph{Bethe ansatz} is a technique for finding their joint eigenvectors and eigenvalues. One constructs a vector $\psi$ called the \emph{weight function} or \emph{Schechtman-Varchenko vector}, which depends on variables called \emph{Bethe roots}. Provided these variables obey certain \emph{Bethe ansatz equations}, then  $\psi$ is a joint eigenvector of the $\H_i$, with certain explicit eigenvalues. Let us stress that this statement is known to hold for arbitrary symmetrizable Kac-Moody algebras $\g$. Indeed, it follows from results in \cite{SV,RV}, as we recall in an appendix.

In the special case where $\g$ is of finite type, much more is known. Namely, in that case the quadratic Gaudin Hamiltonians $\H_i$ belong to a commutative subalgebra $\mc B \subset U(\g)^{\otimes N}$ called the \emph{Gaudin} \cite{Fopers} or \emph{Bethe} \cite{MTV1} subalgebra. The Schechtman-Varchenko vector is a joint eigenvector for this commutative algebra $\mc B$ \cite{FFR}, and the eigenvalues are encoded as functions on a space of \emph{opers} (see below for the definition). In fact there is even a stronger result that the image of $\mc B$ in $\End(\bigotimes_{i=1}^N L_{\lambda_i})$ can be identified with the algebra of functions on a certain space of monodromy-free opers whose singularities are at the marked points $z_i$ and whose residues at these singularities are given by the highest weights $\lambda_i$ -- see \cite{MTVschubert} in type A and \cite{RybnikovProof} in all finite types. 

Now suppose $\g$ is of untwisted affine type. The Gaudin model for such $\g$ was first studied in \cite{FFsolitons} where, by drawing from the analogy with the well understood finite type case \cite{FFR, Fre95, Fopersontheprojectiveline, Fopers}, two natural questions arose:
\begin{enumerate}
\item Are there higher Gaudin Hamiltonians? \emph{i.e.} are the quadratic Hamiltonians above part of some larger commutative subalgebra of (a suitable completion of) $U(\g)^{\ox N}$, such that $\psi$ is still a common eigenvector?
\item If yes, then what parameterizes the eigenvalues of these higher Hamiltonians?
\end{enumerate}
Regarding the first question, a general procedure for constructing the non-local higher affine Gaudin Hamiltonians was outlined in \cite{FFsolitons} by analogy with the construction of non-local quantum KdV and quantum Boussinesq Hamiltonians in \cite{BLZ1, BHK}. Moreover, in the two point case with $\g = \widehat{\mathfrak{sl}}_2$, the local higher Hamiltonians were also conjectured in \cite{FFsolitons} to coincide with the integrals of motion of the coset Virasoro algebra. The second question was also addressed in \cite{FFsolitons} where it was conjectured that the joint eigenvalues of the local and non-local higher Hamiltonians in various affine Gaudin models associated with $\g$ are parametrised by opers on $\CP$ for the Langlands dual Lie algebra $\null^L \g$, with particular singularities at the marked points of the corresponding affine Gaudin model.

In this paper we elucidate the conjectures of \cite{FFsolitons} concerning the \emph{local} higher affine Gaudin Hamiltonians and their eigenvalues. Specifically, by starting from the class of meromorphic affine (Miura) opers on $\CP$ which were conjectured in \cite{FFsolitons} to describe the spectrum of affine Gaudin Hamiltonians (affine (Miura) opers have been defined previously in \cite{MR1896178, Fopersontheprojectiveline}), the main result of the paper is that:
\begin{enumerate}[(i)]
\item There is a notion of quasi-canonical form for these affine opers which is the direct generalisation of the canonical form in finite type, and yet
\item The functions on the space of affine opers turn out to be of a very different character than in the finite case. Namely, they are given by hypergeometric integrals, over cycles of a certain twisted homology defined by the levels of the modules at the marked points.
\end{enumerate}
We conjecture that these hypergeometric integrals give the eigenvalues of local higher affine Gaudin Hamiltonians. This observation in turn allows us to make a conjecture about the form of the local higher affine Gaudin Hamiltonians themselves. 

\bigskip

To explain these statements, let us begin by recalling the situation in finite types.
Consider first $\g$ of finite type of rank $\ell$. The spectrum of the Gaudin algebra for $\g$ is described by $\lg$-opers, where $\lg$ is the Langlands dual of $\g$, \emph{i.e.} the Kac-Moody algebra with transposed Cartan matrix. Let $\lg = \ln_- \oplus \lh \oplus \ln_+$ be a Cartan decomposition and $\bar p_{-1} \coloneqq \sum_{i=1}^\ell \check f_i\in \ln_-$ the corresponding principal nilpotent element ($\check f_i$ are Chevalley generators). A \emph{Miura $\lg$-oper} is a connection of the form
\begin{subequations}
\be d + \left(\bar p_{-1} + u(z)\right) dz \ee 
where $u(z)$ is a meromorphic function valued in $\lh = \h^*$. Let $\alpha_i\in \h^*$, $i=1,\dots,\ell$ be the simple roots of $\g$; they are also the simple coroots of $\lg$. For the Gaudin model with regular singularities as described above, $u(z)$ generically takes the form
\be u(z) = - \sum_{i=1}^N \frac{\lambda_i}{z-z_i} + \sum_{j=1}^m \frac{\alpha_{c(j)}}{z-w_j},\label{u} \ee
\label{mop}\end{subequations}
where $w_1,\dots, w_m$ are the $m\in \ZZ_{\geq 0}$ Bethe roots, with ``colours'' $\{c(j)\}_{j=1}^m \subset \{1,\dots,\ell\}$.
An $\lg$-\emph{oper} is a gauge equivalence class of connections of the form 
\be d + (\bar p_{-1} + b(z)) dz, \nn\ee
where $b(z)$ is a meromorphic function valued in $\lb_+ = \lh \oplus \ln_+$, under the gauge action of the unipotent subgroup ${}^L\!N = \exp(\ln_+)$. So in particular each Miura $\lg$-oper defines an underlying $\lg$-oper, namely the equivalence class to which it belongs.  It is known that each $\lg$-oper has a \emph{unique} representative of the form 
\be d + \left( \bar p_{-1} + \sum_{k\in \bar E} \bar v_k(z)\bar p_k  \right) dz. \label{canform}\ee
Here the sum is over the (finite) set\footnote{In exactly one case, that of type $D_{2n}$, $\bar E$ is a multiset.} $\bar E$ of exponents of $\lg$. For each exponent $k\in \bar E$, $\bar p_k\in \ln_+$ is a certain nonzero element of grade $k$ in the principal gradation of $\lg$. Its coefficient $\bar v_k(z)$ is a meromorphic function valued in $\CC$.
Since this representative is unique, these functions $\{\bar v_k(z)\}_{k\in \bar E}$ are good coordinates on the space of $\lg$-opers. On the underlying $\lg$-oper of the Miura $\lg$-oper in \eqref{mop}, the functions $\{\bar v_k(z)\}_{k\in \bar E}$ will generically have poles at all the poles of $u(z)$. The Bethe equations are precisely the equations needed to ensure they in fact \emph{only} have poles at the marked points $\{z_i\}_{i=1}^N$ and \emph{not} at the Bethe roots $\{w_j\}_{j=1}^m$. Suppose the Bethe equations hold. Then the Schechtman-Varchenko vector $\psi$ obeys $S_k(z) \psi = \bar v_k(z)\psi $ for all $k\in \bar E$, where $\{S_k(z)\}_{k\in \bar E}$ are certain generating functions of the Gaudin algebra.

\bigskip

Let us now turn to affine types and try to follow the steps above as closely as possible. Suppose $\g$ is an untwisted affine Kac-Moody algebra with Cartan matrix of rank $\ell$. Let $\lg$ be its Langlands dual, with Cartan decomposition $\lg =\ln_- \oplus \lh \oplus \ln_+$. Following \cite{Fopersontheprojectiveline, FFsolitons}, we define a \emph{Miura $\lg$-oper} to be a connection of the form
\be d + \left( p_{-1} + u(z) \right) dz \label{amop} \ee
where $u(z)$ is again a meromorphic function valued in $\lh=\h^*$, and where now $p_{-1} \coloneqq \sum_{i=0}^{\ell} \check f_i\in \ln_-.$
The Cartan subalgebra $\lh=\h^*$ of $\lg$ is now of dimension $\ell+2$. As a basis, we may choose the simple roots $\alpha_i$, $i=0,1,\dots, \ell$, of $\g$ (which are the simple coroots of $\lg$) together with a choice of derivation element. It is natural to choose the derivation corresponding to the principal gradation of $\lg$. So let us pick a derivation element $\rho \in \lh$ such that $[\rho, \check e_i] = \check e_i$, $[\rho,\check f_i] = -\check f_i$ for each $i=0,1,\dots,\ell$. 

By analogy with the finite case, it was conjectured in \cite{FFsolitons} that for the Gaudin model with regular singularities at the marked points $\{z_i\}_{i=1}^N$ the relevant Miura $\lg$-opers are those with $u(z)$ just as in \eqref{u} except that now the ``colours'' of the Bethe roots $\{c(j)\}_{j=1}^m \subset \{0,1,\dots,\ell\}$ can include $0$. We can write $u(z)$ in our basis as
\be u(z) = \sum_{i=0}^\ell u_i(z) \alpha_i - \frac{\varphi(z)}{h^\vee}\rho \nn\ee
where $\{u_i(z)\}_{i=0}^\ell$ and $\varphi(z)$ are meromorphic functions valued in $\CC$. (It proves convenient to include the factor of one over the Coxeter number  $h^\vee$ of $\lg$.) The function $\varphi(z)$ depends only on the levels $k_i$ of the $\g$-modules $L_{\lambda_i}$, \emph{i.e.} the values of the central element of $\g$ on these modules: 
\begin{equation} \label{twist function intro}
\varphi(z) = \sum_{i=1}^N \frac{k_i}{z-z_i}.
\end{equation}

An $\lg$-\emph{oper} is defined in \cite{Fopersontheprojectiveline, FFsolitons} to be a gauge equivalence class of connections of the form
\be d + (p_{-1} + b(z)) dz, \nn\ee
where $b(z)$ is a meromorphic function valued in $\lb_+ = \lh \oplus \ln_+$, under the gauge action of the subgroup ${}^L\!N_+ = \exp(\ln_+)$. This subgroup is no longer unipotent, but it is still easy to make sense of gauge transformations grade-by-grade in the principal gradation. See \S\ref{sec: def oper} below. In this way we shall show that each $\lg$-oper has a representative of the form
\be d + \Bigg( p_{-1} - \frac{\varphi(z)}{h^\vee} \rho + \sum_{r\in E} v_r(z) p_r  \Bigg) dz. \label{acf}\ee
Here $E$ denotes the set of positive exponents of $\lg$, which is now an infinite set\footnote{once again, multiset, in the case of type $\null^1 D_{2n}$.}. For each $r\in \pm E$, $p_r$ is a certain nonzero element of grade $r$ in the principal gradation of $\lg$, and its coefficient $v_r(z)$ is a meromorphic function valued in $\CC$. 
In particular, the underlying $\lg$-oper of the Miura $\lg$-oper in \eqref{amop} has a representative of this form. 

\emph{However}, in stark contrast to the case of finite type algebras above, the representative \eqref{acf} is not unique, because there is a residual gauge freedom. This freedom is generated by gauge transformations of the form $\exp( \sum_{r\in E_{\geq 2}} g_r(z) p_r )$, where $g_r(z)$ are meromorphic functions valued in $\CC$ and $E_{\geq 2}$ is the set of positive exponents of $\lg$ excluding $1$. Such a transformation preserves the form of the connection \eqref{acf} and the function $\varphi(z)$ while sending, for each\footnote{There is a subtlety for $r=1$; see Corollary \ref{cor: v1} below.} $r\in E_{\geq 2}$,
\be v_r(z) \longmapsto v_r(z) - g'_r(z) + \frac{r \varphi(z)}{h^\vee} g_r(z). \label{vr}\ee 
Consequently, these $v_r(z)$ are not themselves well-defined functions on the space of $\lg$-opers, and one should not expect them to parameterize eigenvalues of Gaudin Hamiltonians. Rather, one should take appropriate integrals of them. Indeed, consider the multivalued (for generic $k_i$) function on $\CC \setminus\{z_1,\dots,z_N\}$ defined as
\begin{equation*}
\mc P(z) \coloneqq \prod_{i=1}^N (z-z_i)^{k_i}.
\end{equation*}
If we multiply $v_r(z)$ by $\mc P(z)^{-r/h^\vee}$ then its transformation property \eqref{vr} can equivalently be written as
\begin{equation*}
\mc P(z)^{-r/h^\vee} v_r(z) \longmapsto \mc P(z)^{-r/h^\vee} v_r(z) - \partial_z \big( \mc P(z)^{-r/h^\vee} g_r(z) \big).
\end{equation*}
We now see that in order to get gauge-invariant quantities we should consider integrals
\begin{equation} \label{ci}
I^\gamma_r \coloneqq \int_\gamma \mc P(z)^{-r/h^\vee} v_r(z) dz
\end{equation}
over any cycle $\gamma$ along which $\mc P^{-r/h^\vee}$ has a single-valued branch. The prototypical example of such a cycle is a Pochhammer contour, drawn below around two distinct points $z_i$ and $z_j$, $i, j = 1, \ldots, N$:
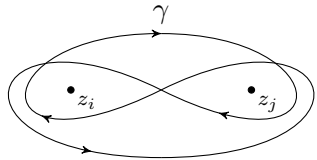
\begin{center}
\begin{tikzpicture}[scale=.6]
\filldraw (0,0) node [below right=-.5mm]{\scriptsize $z_i$} circle (2pt);
\filldraw (4,0) node [below right=-.5mm]{\scriptsize $z_j$} circle (2pt);
\draw[-stealth', postaction={decorate,decoration={markings,mark=at position .4 with {\arrow{stealth'}}}}] (2,0) .. controls (-2,-2) and (-2,1.25) .. (2,1.25) node[above]{$\gamma$};
\draw[postaction={decorate,decoration={markings,mark=at position .8 with {\arrow{stealth'}}}}] (2,1.25) .. controls (6,1.25) and (6,-2) .. (2,0);
\draw[postaction={decorate,decoration={markings,mark=at position .8 with {\arrow{stealth'}}}}] (2,0) .. controls (-2,2) and (-3,-1.5) .. (2,-1.5);
\draw (2,-1.5) .. controls (7,-1.5) and (6,2) .. (2,0);
\end{tikzpicture}
\end{center}
Another way of formulating the above, described in more detail in \S\ref{sec: twisted homology}, is to note that the transformation property \eqref{vr} says that the $1$-form $v_r(z) dz$ is really an element of some suitably defined twisted cohomology, and \eqref{ci} represents its integral over the class of a cycle $\gamma$ in the dual twisted homology. (For an introduction to twisted homology and local systems see \emph{e.g.} \cite{EFK}. Note, though, that the local system underlying the twisted homology described above is conceptually distinct from the ``usual'' local system associated to Gaudin models, namely the local system defined by the KZ connection.) 

Next one should ask about the role of the Bethe equations. Consider the underlying $\lg$-oper of the Miura $\lg$-oper \eqref{amop} with $u(z)$ as in \eqref{u}. We shall show that the Bethe equations are precisely the equations needed to ensure that there exists a choice of gauge in which the functions $v_r(z)$ only have poles at the marked points $\{z_i\}_{i=1}^n$ and not at the Bethe roots $\{w_j\}_{j=1}^m$. The Bethe equations thus ensure that the integrands $\mc P(z)^{-r/h^\vee} v_r(z) dz$ in \eqref{ci} have no residues at the Bethe roots. Thus, in particular, if the Bethe equations hold then the integrals \eqref{ci} do not depend on the position of the chosen contour $\gamma$ relative to these Bethe roots.

The form of the functions \eqref{ci} on the space of $\lg$-opers leads us to conjecture the existence of a collection $\{ \mathcal S_r(z) \}_{r \in E}$ of meromorphic functions valued in (a suitable completion of) $U(\g)^{\otimes N}$, whose properties are listed in Conjecture \ref{conj: higher Ham}. These ensure, in particular, that the corresponding integrals
\begin{equation} \label{op ci}
\hat Q^\gamma_r \coloneqq \int_\gamma \mc P(z)^{-r/h^\vee} \mathcal S_r(z) dz
\end{equation}
mutually commute in the quotient of (the completion of) $U(\g)^{\otimes N}$ in which the central elements act by the levels $k_i$. Moreover, we conjecture that the Schechtman-Varchenko vector $\psi$ is a simultaneous eigenvector of the $\hat Q^\gamma_r$ with eigenvalues given by \eqref{ci}. That is, $\hat Q^\gamma_r \psi = I^\gamma_r \psi$ for any choice of contour $\gamma$ as above and any $r \in E$.

A first non-trivial check of these conjectures is to show that in the case $\g' = \widehat{\mathfrak{sl}}_M$, for $M \geq 3$, there are commuting \emph{cubic} Hamiltonians fitting this pattern. In a forthcoming paper \cite{LVY2}, we explicitly construct such cubic Hamiltonians and prove that they commute; we also check that $\psi$ is an eigenvector with the expected eigenvalues as in \eqref{ci}, at least for $0$ and $1$ Bethe roots.

Our conjecture on the general form of the ``local'' higher affine Gaudin Hamiltonians in \eqref{op ci} is motivated by the recent construction of local integrals of motion in \emph{classical} affine Gaudin models. Specifically, it was shown in \cite{V17} that classical affine Gaudin models provide a unifying framework for describing a broad family of classical integrable field theories. One of the defining features of such theories is that the Poisson bracket of their Lax matrix is characterised by a certain rational function, called the \emph{twist function} $\varphi(z)$. We restrict attention in this article to those with twist function of the form \eqref{twist function intro}. It was subsequently shown in \cite{LMV17}, in the case when $\g$ is the untwisted affine Kac-Moody algebra associated with a semisimple Lie algebra of classical type, how to associate an infinite set $\{ Q^x_r \}_{r \in E}$ of local integrals of motion in such a theory to each zero $x$ of the twist function. These local charges were obtained by generalising the original procedure of \cite{Evans:1999mj} for classical principal chiral models  on compact Lie groups of classical type, which had later also been extended to various other classical integrable field theories in \cite{Evans:2000hx,Evans:2000qx,Evans:2005zd}, see also \cite{Evans:2001sz}.
As we argue in \S\ref{sec: classical lim}, the integral over the contour $\gamma$ in \eqref{op ci} localises in the classical limit to critical points of the function $\P(z)$, in other words to zeroes of the twist function $\varphi(z)$. In this sense, the operators \eqref{op ci} provide natural quantisations of the local integrals of motion $Q^x_r$ in the classical affine Gaudin model.

Let us finally note that the appearance of hypergeometric integrals, as in \eqref{ci}, is very suggestive in relation to recent work on the massive ODE/IM correspondence for the Fateev model \cite{Lukyanov:2013wra, Bazhanov:2013cua}.

\bigskip

The paper is organised as follows.

In \S\ref{sec: affine algebra}, to set the notation we recall the definition of an affine Kac-Moody algebra $\g$ and its Langlands dual $\lg$, focusing on the latter for the purpose of this paper. In particular, we recall the definition and main properties of its principal subalgebra.

In \S\ref{sec: opers} we introduce, following \cite{Fopersontheprojectiveline, FFsolitons}, the space of meromorphic $\lg$-opers on $\CP$, working in a fixed global coordinate on $\CC \subset \CP$. The main result of this section is Theorem \ref{thm: quasi-canonical form} which describes the quasi-canonical form of an $\lg$-oper $[\nabla]$. This allows us to describe gauge invariant functions on the space of $\lg$-opers as hypergeometric integrals of the form \eqref{ci} in Corollary \ref{cor: opint}.

In \S\ref{sec: Miura opers} we introduce the relevant class of Miura $\lg$-opers, following \cite{FFsolitons}, with simple poles at the marked points $z_i$, $i = 1, \ldots, N$ with residues $\lambda_i \in \h^\ast$, and additional simple poles at the Bethe roots $w_j$, $j = 1, \ldots, m$. The $\lg$-oper $[\nabla]$ underlying such a Miura $\lg$-oper $\nabla$ is shown to be regular at each of the Bethe roots $w_j$ if and only if the Bethe equations hold. Moreover, we show that the eigenvalues of the quadratic Gaudin Hamiltonians \eqref{quad Ham intro} on the tensor product $\bigotimes_{i=1}^N L_{\lambda_i}$ appear as the residues at the $z_i$ in the coefficient of $p_1$ in any quasi-canonical form of the $\lg$-oper $[\nabla]$.

Based on the description of functions on the space of $\lg$-opers from Corollary \ref{cor: opint}, in \S\ref{sec: main conj} we formulate our main conjecture about the form of the higher Gaudin Hamiltonians of an affine Gaudin model associated with the affine Kac-Moody algebra $\g$. See Conjecture \ref{conj: higher Ham}.

In \S\ref{sec: coord} we give a coordinate-independent definition of meromorphic $\lg$-opers on an arbitrary Riemann surface $\Sigma$. In particular, we compare and contrast the description of the space of $\lg$-opers in the cases when $\lg$ is of finite and affine type.

Specialising the discussion of \S\ref{sec: coord} to the case $\Sigma = \CP$, \S\ref{sec: twisted homology} is devoted to a coordinate-independent description of the functions on the space of $\lg$-opers from Corollary \ref{cor: opint}.

In \S\ref{sec: discussion} we discuss various connections between the present work and the literature. In particular, we compare our main Theorem \ref{thm: quasi-canonical form} with the procedure of Drinfel'd and Sokolov \cite{DS} for constructing classical integrals of motion of generalised (m)KdV. We also mention connections with the (massive) ODE/IM correspondence. We provide motivation for Conjecture \ref{conj: higher Ham} by relating the form of the classical limit of the higher Gaudin Hamiltonians with the existing hierarchy of classical integrals of motion in classical affine Gaudin models.

Finally, in appendix \ref{sec: hyp arr} we briefly review the work of Schechtmann and Varchenko \cite{SV} on the diagonalisation of the quadratic Gaudin Hamiltonians for an arbitrary Kac-Moody algebra $\g$ by the Bethe ansatz.

\subsubsection*{Acknowledgements}
CY is grateful to E. Mukhin for interesting discussions. 
The authors thank M. Magro for interesting discussions. 
This work is partially supported by the French Agence Nationale de la Recherche (ANR) under grant ANR-15-CE31-0006 DefIS.

\section{The affine algebra $\lg$} \label{sec: affine algebra}

\subsection{Cartan data and defining relations} \label{sec: Cartan data}
Let $\g \coloneqq \g(A)$ be an untwisted affine Kac-Moody algebra with indecomposable Cartan matrix $A \coloneqq (A_{ij})_{i,j=0}^\ell$, and let $\lg \coloneqq  \g({}^t\!A)$ be its Langlands dual, namely the affine Kac-Moody algebra associated with the transposed Cartan matrix. 
We have the Cartan decomposition 
\be \g = \n_- \oplus \h \oplus \n_+\nn\ee 
where $\h$ is a complex vector space of dimension $\dim \h = \ell + 2$. The sets of simple roots $\{ \alpha_i \}_{i=0}^\ell$ and simple coroots $\{ \check \alpha_i \}_{i=0}^\ell$ of $\g$ are by definition linearly independent subsets of $\h^\ast$ and $\h$, respectively, such that $A_{ij} = \langle\alpha_j, \check\alpha_i\rangle$ for $i, j \in I \coloneqq \{ 0,\ldots, \ell \}$. Here $\langle\cdot,\cdot\rangle: \h^\ast \times \h \to \CC$ denotes the canonical pairing.
In the Cartan decomposition of $\lg$, 
\be \lg = \ln_- \oplus \lh \oplus \ln_+,\nn\ee 
we may identify $\lh = \h^*$. Then $\{ \alpha_i \}_{i=0}^\ell$ is a set of simple \emph{coroots} of $\lg$, and $\{ \check\alpha_i \}_{i=0}^\ell$ a set of simple \emph{roots} of $\lg$. 
In terms of the Chevalley generators $\check e_i$, $i \in I$, of $\ln_+$ and $\check f_i$, $i \in I$, of $\ln_-$, the defining relations of $\lg$ are given by
\begin{subequations} \label{KM relations}
\begin{alignat}{2}
\label{KM rel a} [x, \check e_i] &= \langle x,\check \alpha_i\rangle  \check e_i, &\qquad
[x, \check f_i] &= - \langle x,\check\alpha_i\rangle \check f_i, \\
\label{KM rel b} [x, x'] &= 0, &\qquad
[\check e_i, \check f_j] &= \alpha_i \delta_{i,j},
\end{alignat}
for any $x, x' \in \lh$, together with the Serre relations
\begin{equation} \label{KM rel c}
(\text{ad}\, \check e_i)^{1- A_{ji}} \check e_j = 0, \qquad (\text{ad}\, \check f_i)^{1- A_{ji}} \check f_j = 0.
\end{equation}
\end{subequations}

\begin{remark}
We shall be mostly concerned with the Lie algebra $\lg$ rather than $\g$. Nevertheless, since we have in mind applications to the Gaudin model for $\g$, we prefer to keep the notation adapted to $\g$, at the cost of the somewhat non-standard appearance of these relations \eqref{KM relations} and others below.
\end{remark}

Let $a_i$ (resp. $\check a_i$), $i\in I$, be the unique positive relatively prime integers such that $A\, \null^t \!(a_0, \ldots, a_\ell) = 0$ (resp. $\null^t\!A\,\null^t\!(\check a_0, \ldots, \check a_\ell) =0$).
Define
\be h\coloneqq \sum_{i=0}^\ell  a_i\qquad \hc \coloneqq \sum_{i=0}^\ell \check a_i.\nn\ee 
Then $h$ is the Coxeter number of $\g$ (and the dual Coxeter number of $\lg$) while $\hc$ is the Coxeter number of $\lg$ (and the dual Coxeter number of $\g$). Define also
\be \delta \coloneqq \sum_{i=0}^\ell a_i \alpha_i,\qquad \cent \coloneqq \sum_{i=0}^\ell \check a_i \check \alpha_i. \nn\ee
Then $\delta$ spans the centre of $\lg$ while $\cent$ spans the centre of $\g$. Denote by $\g' = [\g,\g]$ and $\lgp\coloneqq [\lg,\lg]$ the derived subalgebras of $\g$ and $\lg$, respectively.

We shall suppose that $\check a_0 = 1$ and $a_0 = 1$.
\begin{remark} \label{rem: not A2l}
One has $\check a_0 = 1$ and $a_0 = 1$ for all affine Kac-Moody algebras except for type $\null^2\!A_{2k}$. In type $\null^2\!A_{2k}$ one can choose to take either $\check a_0 = 1$ and $a_0 = 2$ or vice versa $\check a_0 = 2$ and $a_0 = 1$.
The Cartan matrices in these two descriptions are transposes of one another so that in this case $\lg$ and $\g$ are both twisted, of type $\null^2\!A_{2k}$. Since we have in mind applications to the Gaudin model for an untwisted affine Kac-Moody algebra $\g$, we shall not consider this case. 
\end{remark}

Recall that given any $d\in \h$ such that $\langle \delta, d\rangle \neq 0$, $\{\check\alpha_i\}_{i=0}^\ell \cup \{d\}$ forms a basis of $\h$; and similarly, given any $\Lambda\in \lh$ such that $\langle \Lambda, \cent\rangle \neq 0$, $\{\alpha_i\}_{i=0}^\ell \cup \{ \Lambda\}$ provides a basis for $\lh$. We call such elements $d$ and $\Lambda$ \emph{derivation elements} of $\h$ and $\lh$, respectively.

Let $\cocent\in \h$ be a derivation element of $\g$ such that
\begin{equation*}
\langle \alpha_i , \cocent  \rangle = \delta_{i,0},\qquad i\in I.
\end{equation*}
Such a $\cocent$ is unique up to the addition of a multiple of $\cent$. Having made such a choice we define a non-degenerate symmetric bilinear form $(\cdot | \cdot) : \h \times \h \to \CC$ on $\h$ by
\begin{equation} \label{bilinear form def}
(\check\alpha_i | x) =  a_i \check a_i^{-1} \langle \alpha_i, x \rangle, \qquad
(\cocent | \cocent) = 0
\end{equation}
for any $i \in I$ and $x \in \h$. It extends uniquely to an invariant symmetric bilinear form on the whole of $\g$ \cite[Proposition 2.2]{KacBook}, which we also denote $(\cdot | \cdot) : \g \times \g \to \CC$. It also induces a linear isomorphism $\nu: \h\SimTo \lh$, and hence we have a non-degenerate symmetric bilinear form $(\nu^{-1}(\cdot) |\nu^{-1}(\cdot) ): \lh \times \lh \to \CC$ on $\lh$, which henceforth we shall also denote by $(\cdot|\cdot)$. The latter then extends uniquely to an invariant symmetric bilinear form  $(\cdot | \cdot) : \lg \times \lg \to \CC$ on the whole of $\lg$.

There exists a unique set $\{\Lambda_i\}_{i=0}^\ell \subset \lh$ of derivation elements of $\lh$, the fundamental coweights of $\lg$ (and the fundamental weights of $\g$) relative to our choice of $\cocent$, such that 
\begin{equation}
\langle \Lambda_i, \cocent \rangle = 0\qquad\text{and}\qquad
\langle \Lambda_i, \check\alpha_j \rangle = \delta_{i,j},\qquad i,j\in I.\label{def: Lambda}
\end{equation}
Likewise, there exists a unique set $\{ \check\Lambda_i \}_{i=0}^\ell \subset \h$ of derivation elements of $\h$, the fundamental coweights of $\g$ (and the fundamental weights of $\lg$) such that
\begin{equation}
\langle \Lambda_0, \check\Lambda_i \rangle = 0\qquad\text{and}\qquad
\langle \alpha_i, \check\Lambda_j \rangle = \delta_{i,j},\qquad i,j\in I.\label{def: co Lambda}
\end{equation}
In particular, we have $\check\Lambda_0 = \cocent$.

\subsection{Principal gradation} \label{sec: principal grad}

Let $\check Q \coloneqq \bigoplus_{i=0}^\ell \ZZ \check \alpha_i$ be the root lattice of $\lg$. We have the root space decomposition
\be\lg = \bigoplus_{\check\alpha\in \check Q} \lg_{\check\alpha}, \nn\ee
where $\lg_{\check\alpha} \coloneqq \{ x \in \lg \,|\, [h, x] = \langle h,\check\alpha\rangle x \;\text{for all} \; h \in \lh\}$. In particular, for the origin of the root lattice $0 \in \check Q$ we have $\lg_0 = \lh$.
The \emph{height} of a root $\check\alpha = \sum_{i=0}^\ell r_i \check\alpha_i \in \check Q$ is $\hgt(\check \alpha) \coloneqq \sum_{i=0}^\ell r_i$. The \emph{principal gradation} of $\lg$ is the $\ZZ$-gradation defined by 
\begin{equation*}
\lg = \bigoplus_{n \in \ZZ} \lg_n,\qquad \lg_n \coloneqq \bigoplus_{\substack{\check\alpha\in\check Q\\ \hgt (\check\alpha) = n}} \lg_{\check \alpha}.
\end{equation*}
Equivalently, the principal gradation is the $\ZZ$-gradation defined by 
\be \deg(\check e_i) = 1,\qquad \deg(\check f_i) = -1, \qquad i\in I,\nn\ee 
and $\deg(\lh) = 0$. 
In particular $\lg_{0} = \lh$, so that the notation $\lg_0$, where the subscript $0$ could stand for either $0 \in \check Q$ or $0 \in \ZZ$, is unambiguous.

Let $\rho \in \lh$ be the unique derivation element of $\lh$ such that
\begin{equation*}
\langle \rho, \check \alpha_i \rangle = 1, \qquad (\rho | \rho) = 0,
\end{equation*}
for every $i \in I$. By the first property we have $\langle \rho, \cent \rangle = h^\vee$.
The $\ad$-eigenspaces of $\rho$ are the subspaces $\lg_n$, $n\in \ZZ$. Indeed, we have
\begin{equation*}
[\rho, \check e_i] = \check e_i,\qquad [\rho,\check f_i] = -\check f_i, \qquad i\in I.
\end{equation*}

\subsection{Principal subalgebra and exponents}\label{sec: princ sub}
\def\lbg{\mathcal L}
Define $p_{-1}$, the \emph{cyclic element} of $\lg$, as
\begin{equation}
p_{-1} \coloneqq \sum_{i=0}^\ell \check f_i.\label{def: pm1}
\end{equation}
It belongs to the $(-1)^{\rm st}$-grade of the derived subalgebra $\lgp=[\lg,\lg]$.
There is a realization of $\lgp$ as the central extension of a certain twisted loop algebra $\lbg$, in such a way that the power of the formal loop variable $t$ measures the grade in the principal gradation. (Equivalently, the derivation element $\rho\in \lg$ is realized as $t \del_t$.) By studying this realization, one establishes some important facts about the adjoint action of $p_{-1}$. Here we shall merely recall these facts; for more details see \cite[Chapter 14]{KacBook}.
Let 
\be \pi : \lgp \longrightarrow \lbg \cong \lgp/\CC\delta\nn\ee 
be the canonical projection. The twisted loop algebra $\lbg$ is the direct sum of the image and the kernel of the adjoint action of $\pi(p_{-1})$:
\begin{subequations}\label{Ld}
\be\lbg = \ker(\ad_{\pi(p_{-1})}) \oplus \textup{im}(\ad_{\pi(p_{-1})}),\ee
and this decomposition respects the principal gradation, \emph{i.e.} for each $n\in \ZZ$,
\be \lbg_n = \ker(\ad_{\pi(p_{-1})})_n \oplus \textup{im}(\ad_{\pi(p_{-1})})_n, \ee
\end{subequations}
where $\lbg_n \coloneqq \pi(\lgp_n)$.

The graded subspaces $\textup{im}(\ad_{\pi(p_{-1})})_n$ are all of dimension $\ell$ and moreover
\begin{equation*}
\ad_{\pi(p_{-1})} : \textup{im}(\ad_{\pi(p_{-1})})_n \overset{\sim}\longrightarrow \textup{im}(\ad_{\pi(p_{-1})})_{n-1}
\end{equation*}
is a linear isomorphism for each $n$.
The graded subspaces $\ker(\ad_{\pi(p_{-1})})_n$ have dimensions encoded by the exponents. Indeed, the multiset of \emph{exponents} of $\lg$ is by definition the multiset consisting of each integer $n$ with multiplicity $\dim(\ker(\ad_{\pi(p_{-1})})_n)$. One has  $\dim(\ker(\ad_{\pi(p_{-1})})_n)=  \dim(\ker(\ad_{\pi(p_{-1})})_{-n})$ and $\dim(\ker(\ad_{\pi(p_{-1})})_0)=0$. So the multiset of exponents is of the form $\pm E$, where we denote by $E$ the multiset of strictly positive exponents. The kernel  $\ker(\ad_{\pi(p_{-1})})$ forms an abelian Lie subalgebra of the twisted loop algebra $\lbg$, called the \emph{principal subalgebra}.

We need the ``lift'' to $\lg$ of the decomposition \eqref{Ld}. For each $n\in \ZZ_{\neq 0}$, we have $\lg_n = \lgp_n$, the map $\pi|_{\lgp_n} : \lgp_n \xrightarrow\sim \lbg_n$ is a linear isomorphism, and one defines 
\be \a_n \coloneqq (\pi|_{\lgp_n})^{-1}\left( \ker(\ad_{\pi(p_{-1})})_n\right),\qquad 
 \c_n \coloneqq (\pi|_{\lgp_n})^{-1}\left( \textup{im}(\ad_{\pi(p_{-1})})_n\right).\nn\ee
Meanwhile the subspaces $\a_0$ and $\c_0$ of $\lg_0=\lh$ are defined as
\be 
\a_0 \coloneqq \CC\delta \oplus \CC\rho,\qquad 
   \c_0 \coloneqq \ad_{p_{-1}}( \c_{1} )
 .\nn\ee
Then for each $n\in \ZZ$ we have the direct sum decomposition
\begin{equation} \label{lg n decompA}
\lg_n = \a_n \oplus \c_n.
\end{equation}
Let $\a = \bigoplus_{n\in \ZZ} \a_n$ and $\c = \bigoplus_{n\in \ZZ} \c_n$, so $\lg = \a \oplus \c$. 
One has $\dim(\c_n)=\ell$ for each $n\in \ZZ$ and the linear map 
\begin{equation*}
\ad_{p_{-1}} : \c_{n} \overset{\sim}\longrightarrow \c_{n-1}
\end{equation*} 
is an isomorphism for every $n\in \ZZ$. 

The subspace $\a$ is a Lie subalgebra, the \emph{principal subalgebra of $\lg$}. It is the central extension, by a one-dimensional centre $\CC\delta$, of the principal subalgebra $\textup{im}(\ad_{\pi(p_{-1})})$ of $\lbg$, equipped with a derivation element $\rho$. Indeed, we may pick a basis $\{p_n\}_{n\in \pm E} \cup \{ \delta, \rho\}$ of $\a$ where for each exponent $n\in \pm E$, $p_n\in \a_n$. 
This basis can be so chosen that the non-trivial Lie algebra relations of $\a$ are given by
\begin{equation} \label{Lie alg a com rel}
[p_m, p_n] = m \delta_{m+n,0} \, \delta, \qquad [\rho, p_n] = n\, p_n, \qquad m, n \in \pm E.
\end{equation}
The restriction to $\a$ of the bilinear form $(\cdot | \cdot)$ on $\lg$ is non-degenerate, with the non-trivial pairings given by
\begin{equation*}
(\delta | \rho) = (\rho | \delta) = h^\vee, \qquad (p_m | p_n) = h^\vee \delta_{m+n,0}, \qquad m, n \in \pm E.
\end{equation*}

\begin{remark}\label{rem: pirem}$ $
\begin{enumerate}[(a)]
\item $\pm 1$ are always exponents with multiplicity 1. We keep $p_{-1}$ as in \eqref{def: pm1} and set 
\be p_1 = \sum_{i=0}^\ell a_i \check e_i.\nn\ee
\item The pattern of exponents is periodic with period $rh^\vee$, where ${}^r\!X_{\mathsf N}$ is the type of $\lg$ in Kac's notation. For a table of the patterns of exponents in all types see e.g. \cite[Chapter 14]{KacBook} or \cite[\S 5]{DS}.
\item The exponents of $\g$ and $\lg$ are the same \cite[Corollary 14.3]{KacBook}, which is important for Conjecture \ref{conj: higher Ham} below. (Consequently, if  ${}^sY_{\mathsf M}$ is the type of  $\g$ then the pattern of exponents is also periodic with period $s h$, which need not equal $rh^\vee$. For us $s=1$ since $\g$ is untwisted.)
\item In all types except ${}^1\!D_{2k}$, the multiset $E$ of positive exponents is actually a set, \emph{i.e.} $\dim(\a_n)\in \{0,1\}$ for all $n\in \ZZ_{\neq 0}$. In such cases, for each $j\in E$ the basis element $p_j\in \a_j$ is unique up to rescaling. Exceptionally, in type $\null^1 \!D_{2k}$ one has $\dim(\a_{2k-1 + (4k-2) n})=2$ for every $n\in\ZZ$. For each $n \geq 0$ one must therefore pick two basis vectors, each one labelled by one of the two distinct copies of $2k-1 + (4k-2) n$ in $E$.\label{rem: p freedom} (The basis vectors for $n\leq 0$ are then fixed by the form of the bilinear form above.)
\item
The action of the $\CC$-linear map $\ad_{p_{-1}} : \lg \to \lg$ on the subspaces $\lg_n$, $n \in \ZZ$, of the principal gradation of $\lg$ can be summarised in the following diagram
\begin{equation*}
\begin{tikzpicture}[bij/.style={above,sloped,inner sep=0.6pt}]
\matrix (m) [matrix of math nodes, row sep=.8em, column sep=2.5em,text height=1.5ex, text depth=0.25ex]    
{
& & & \CC \rho & \a_{-1} & \a_{-2} & \cdots\\
\cdots & \c_2 & \c_1 & \c_0 & \c_{-1} & \c_{-2} & \cdots\\
\cdots & \a_2 & \a_1 & \CC \delta & & & \\
};
\path[->] (m-2-1) edge node[bij]{$\sim$} (m-2-2);
\path[->] (m-2-2) edge node[bij]{$\sim$} (m-2-3);
\path[->] (m-2-3) edge node[bij]{$\sim$} (m-2-4);
\path[->] (m-3-3) edge (m-3-4);
\path[->] (m-2-4) edge node[bij]{$\sim$} (m-2-5);
\path[->] (m-1-4) edge (m-1-5);l
\path[->] (m-2-5) edge node[bij]{$\sim$} (m-2-6);
\path[->] (m-2-6) edge node[bij]{$\sim$} (m-2-7);
\node at ($1/2*(m-1-4)+ 1/2*(m-2-4)$) {$\oplus$};
\node at ($1/2*(m-2-2)+ 1/2*(m-3-2)$) {$\oplus$};
\node at ($1/2*(m-2-3)+ 1/2*(m-3-3)$) {$\oplus$};
\node at ($1/2*(m-2-4)+ 1/2*(m-3-4)$) {$\oplus$};
\node at ($1/2*(m-1-5)+ 1/2*(m-2-5)$) {$\oplus$};
\node at ($1/2*(m-1-6)+ 1/2*(m-2-6)$) {$\oplus$};
\end{tikzpicture}
\end{equation*}
where each column corresponds to a subspace $\lg_n$ decomposed as in \eqref{lg n decompA}.
\item \label{rem: Bn vs im p-1}
Recall the decomposition \eqref{Ld} of the subquotient $\lbg$. In fact $\c_n = \textup{im}(\ad_{p_{-1}})_n$ and $\a_n = \ker(\ad_{p_{-1}})_n$ for every $n\in \ZZ$, with precisely the following exceptions: $\c_0 \neq ( \textup{im} \ad_{p_{-1}} )_0$ and $\c_{-1} \neq ( \textup{im} \ad_{p_{-1}} )_{-1}$ since $\delta = [p_1, p_{-1}]$ and $p_{-1} = [p_{-1}, \rho]$ both belong to the image of $\ad_{p_{-1}} : \lg \to \lg$; and $\a_1 \neq ( \ker \ad_{p_{-1}} )_1$ since $\ad_{p_{-1}} : \lg_1 \hookrightarrow \lg_0$ is injective. \qedhere
\end{enumerate}
\end{remark}

\section{$\lg$-opers and quasi-canonical form} \label{sec: opers}

\subsection{Inverse limits} \label{sec: inverse limits}
Recall the subalgebras $\lh$ and $\ln_+$ of $\lg$ from \S\ref{sec: Cartan data}. We introduce also the Borel subalgebra $\lb_+ \coloneqq \lh \oplus \ln_+ \subset \lg$. These can be described in terms of the principal gradation of $\lg$ as $\ln_+ = \bigoplus_{n > 0} \lg_n$ and $\lb_+ = \bigoplus_{n \geq 0} \lg_n$.
Moreover, there is a natural descending $\ZZ_{> 0}$-filtration on $\ln_+$ (and $\lb_+$) by Lie ideals
\begin{equation*}
\ln_k = \bigoplus_{n \geq k} \lg_n, \qquad k \in \ZZ_{> 0}.
\end{equation*}
Since $\ln_k \subset \lgp$ for each $k \in \ZZ_{> 0}$, these ideals also define a descending $\ZZ_{>0}$-filtration on the derived subalgebra $\lb'_+ \coloneqq \lb_+ \cap \lgp$.

Let $\M$ be the field of meromorphic functions on $\CP \coloneqq \CC \cup \{ \infty \}$. For any Lie subalgebra $\p \subset \lg$ we introduce the Lie algebra $\p(\M) \coloneqq \p \otimes \M$ of $\p$-valued meromorphic functions on $\CP$.

The Lie algebras $\ln_k(\M)$, $k \in \ZZ_{> 0}$ endow $\ln_+(\M)$ with a descending $\ZZ_{>0}$-filtration by ideals such that the quotient Lie algebras $\ln_+(\M) / \ln_k(\M)$, $k \in \ZZ_{>0}$ are nilpotent. Consider the Lie algebra defined as the inverse limit
\begin{equation*}
\lhn_+(\M) \coloneqq \varprojlim \ln_+(\M) / \ln_k(\M).
\end{equation*}
By definition, its elements are infinite sums $\sum_{n > 0} y_n$, with $y_n \in \lg_n (\M)$, which truncate to finite sums when working in the quotient $\ln_+(\M) / \ln_k(\M)$ for any $k \in \ZZ_{> 0}$. 

\begin{remark}
It should be stressed that for a given element $\sum_{n > 0} y_n$ of $\lhn_+(\M)$, the orders of the poles of the $\lg_n$-valued meromorphic functions $y_n$ are allowed to increase without bound as $n$ increases. Thus $\lhn_+(\M)$ is strictly larger than $\lhn_+ \ox \M$, where $\lhn_+ \coloneqq \varprojlim \ln_+ /\ln_k$ is the completion of $\ln_+$.
\end{remark}

We also have the inverse limits
\begin{alignat*}{3} \lhb_+(\M) &\coloneqq &\lh(\M) &\oplus \lhn_+(\M) &&= \varprojlim \lb(\M) / \ln_k(\M),\\
\lhg(\M) &\coloneqq \ln_-(\M) \oplus {}&\lh(\M) &\oplus \lhn_+(\M) &&= \varprojlim \lg(\M) / \ln_k(\M).
\end{alignat*}
The latter is an inverse limit of vector spaces only, since the $\ln_k(\M)$ are not Lie ideals in $\lg(\M)$. Nonetheless, $\lhg(\M)$ is a Lie algebra, with $\lg(\M)$ as a subalgebra.\footnote{Given any two elements $x= \sum_{n>-N} x_n$, $y=\sum_{n>-M} y_n$ of the vector space $\lhg(\M)$, with each $x_n,y_n\in \lg_n(\M)$, their Lie bracket $[x,y] = \sum_{k>-N-M} \sum_{\substack{n>-N,m>-M\\ n+m = k}} [x_n,y_m]$ is a well-defined element of $\lhg(\M)$ since the inner sum is finite for each $k$. This bracket obeys the Jacobi identity and agrees with the usual bracket on $\lg(\M)\subset \lhg(\M)$.}

\subsection{The group $\lhN_+(\M)$} \label{sec: group lhN}

For every $k \in \ZZ_{> 0}$, the Baker-Campbell-Hausdorff formula then endows the vector space $\ln_+(\M)/ \ln_k(\M)$ with the structure of a group. Specifically, we denote this group by
\begin{equation*}
\exp \big( \ln_+(\M) / \ln_k(\M) \big) \coloneqq \{ \exp(m) \,|\, m \in \ln_+(\M) / \ln_k(\M) \},
\end{equation*}
whose elements are denoted formally as exponentials of elements in $\ln_+(\M) / \ln_k(\M)$. The group operation is then defined as
\begin{equation} \label{BCH formula}
\exp(x) \exp(y) \coloneqq \exp(x \bullet y) = \exp(x+y+ \ha [x,y] + \ldots)
\end{equation}
for all $x, y \in \ln_+(\M)/ \ln_k(\M)$. Here $x \bullet y$ is given by the Baker-Campbell-Hausdorff formula, whose first few term are shown in the exponent on the right hand side of \eqref{BCH formula}. The sum is finite because $\ln_+(\M) / \ln_k(\M)$ is nilpotent.

Now the formal exponential map
$\exp : \ln_+(\M) / \ln_k(\M) \SimTo \exp\big( \ln_+(\M) / \ln_k(\M) \big)$
is a bijection by definition,
and there are canonical group homomorphisms $\pi^m_k$ making the following diagram commutative:
\begin{equation*}
\begin{tikzpicture}[bij/.style={below,sloped,inner sep=2pt}]
\matrix (m) [matrix of math nodes, row sep=2.7em, column sep=2.5em,text height=1.5ex, text depth=0.25ex]    
{
\exp \big( \ln_+(\M) / \ln_m(\M) \big) & \exp \big( \ln_+(\M) / \ln_k(\M) \big)\\
\ln_+(\M) / \ln_m(\M) & \ln_+(\M) / \ln_k(\M)\\
};
\path[->>] (m-1-1) edge node[above]{$\pi^m_k$} (m-1-2);
\path[<-] (m-1-2) edge node[bij]{$\sim$} node[left=2.5mm]{\raisebox{-2.5mm}{$\exp$}} (m-2-2);
\path[<-] (m-1-1) edge node[bij]{$\sim$} node[left=2.5mm]{\raisebox{-2.5mm}{$\exp$}} (m-2-1);
\path[->>] (m-2-1) edge (m-2-2);
\end{tikzpicture}
\end{equation*}
for all $m \geq k > 0$. We define a group $\lhN_+(\M)$ as the corresponding inverse limit
\begin{equation} \label{pro-unipotent}
\lhN_+(\M) \coloneqq \varprojlim \exp \big( \ln_+(\M) / \ln_k(\M) \big).
\end{equation}
The above commutative diagram defines an exponential map $\exp : \lhn_+(\M) \to \lhN_+(\M)$.

\subsection{Definition of an $\lg$-oper} \label{sec: def oper}
Now, and until \S\ref{sec: coord} below, we shall pick and fix a global coordinate $z$ on $\CC \subset \CP$. Thus, for any $f\in \lhb_+(\M)$ its holomorphic de Rham differential is $d f = dz \del_z f$.

Define $\op_{\lg}(\CP)$ to be the affine space of connections of the form \cite{Fopersontheprojectiveline}
\be \nabla = d + p_{-1} dz + b dz, \qquad b\in \lhb_+(\M). \label{affine space}\ee

\begin{remark} 
This is an affine space over $\lhb_+(\M)$. For the moment, in calling it a space of connections we mean merely that it admits an action of the group $\lhN_+(\M)$ by gauge transformations, as we shall now describe. In \S\ref{sec: coord} we will discuss its behaviour under coordinate transformations.
\end{remark}

Define the adjoint action of the group $\lhN_+(\M)$ on the vector space $\lhg(\M)$ as follows. Let $g = \exp(m) \in \lhN_+(\M)$ with $m = \sum_{n > 0} m_n \in \lhn_+(\M)$. 
\begin{subequations} \label{gauge transf}
For any $u \in \lhg(\M)$, which we write as $u=\sum_{n \geq M} u_n$ for some $M \in \ZZ$, we define the adjoint action of $g$ on $u$ as
\begin{equation} \label{gauge transf c}
g u g^{-1} \coloneqq \sum_{k \geq 0} \frac{1}{k!} \ad_m^k u = \sum_{n \geq M} u_n + \sum_{n \geq M} \sum_{r > 0} [m_r, u_n] + \frac{1}{2} \sum_{n \geq M} \sum_{r,s > 0} \big[ m_s, [m_r, u_n] \big] + \ldots
\end{equation}
where the dots represent terms involving an increasing number of $m_n$'s with $n \geq 1$. Since $\deg m_n = n$ in the principal gradation of $\lg$, it follows that for each $k \in \ZZ_{> 0}$ there are only finitely many terms of degree less than $k$ in the expression on the right hand side. Therefore the sum on the right hand side of \eqref{gauge transf c} is a well-defined element of $\lhg(\M)$. 

\begin{lemma} \label{lem: ga}The definition \eqref{gauge transf c} defines an action of the group $\lhN_+(\M)$ on $\lhg(\M)$.\end{lemma}
\begin{proof}
By the Baker-Campbell-Hausdorff formula we have
\begin{equation*}
\bigg( \sum_{k \geq 0} \frac{1}{k!} \ad_m^k \bigg) \bigg( \sum_{\ell \geq 0} \frac{1}{\ell!} \ad_n^\ell \bigg) u = \sum_{k \geq 0} \frac{1}{k!} \ad_{m \bullet n}^k u,
\end{equation*}
for any $m, n \in \lhn_+(\M)$ and $u \in \lhg(\M)$.
\end{proof}

Now we define also
\begin{equation} \label{gauge transf a}
(dg) g^{-1} \coloneqq \sum_{k \geq 1} \frac{1}{k!} \ad_m^{k-1} dm
= \sum_{n > 0} d m_n + \frac{1}{2} \sum_{n,r > 0} [m_r, d m_n] + \ldots,
\end{equation}
\end{subequations}
which is a well-defined sum in $\lhn_+(\M)dz$. 
\begin{lemma} \label{lem: dga} For any  $g,h\in \lhN_+(\M)$, we have
\be d(gh) (gh)^{-1} = g \left((dh) h^{-1}\right) g^{-1} + (dg) g^{-1}. \nn\ee
\end{lemma} 
\begin{proof} By direct calculation from the definitions \eqref{gauge transf} one verifies that 
\be d( gyg^{-1}) = \left[ dg g^{-1}, gyg^{-1}\right] + g (dy) g^{-1},\nn\ee
for any $y \in \lhg(\M)$ and any $g\in \lhN_+(\M)$.
By Lemma \ref{lem: ga}, we have $(gh) y (gh)^{-1} = g (h y h^{-1}) g^{-1}$, and on applying $d$ to both sides we obtain
\be \left[- d(gh) (gh)^{-1} + g \left((dh) h^{-1}\right) g^{-1} + (dg) g^{-1}, x \right] = 0\nn\ee
where $x=gyg^{-1}$ is arbitrary. Since the centre of $\lhn_+(\M)$ is trivial, the result follows. 
\end{proof}

Observe that $g p_{-1} g^{-1} -p_{-1} \in \lhb_+(\M)$, so that if $u \in \lhg(\M)$ is of the form $u=p_{-1} + b$ with $ b\in \lhb_+(\M)$ then so is $gug^{-1}$. Hence, from Lemmas \ref{lem: ga} and \ref{lem: dga}, we have the following.

\begin{proposition}
We have an action of $\lhN_+(\M)$ on $\op_{\lg}(\CP)$ defined by
\begin{align*}
\lhN_+(\M) \times \op_{\lg}(\CP) &\longrightarrow \op_{\lg}(\CP), \\
(g, d + p_{-1} dz + b dz) &\longmapsto d + g p_{-1} g^{-1} dz - (dg) g^{-1} + g b g^{-1}dz,
\end{align*}
which we refer to as the action by \emph{gauge transformations}. If $\nabla \in \op_{\lg}(\CP)$ then we denote by $\nabla^g \in \op_{\lg}(\CP)$ its gauge transformation by an element $g \in \lhN_+(\M)$. \qed\end{proposition}

Our main object of interest, the space of \emph{$\lg$-opers}, can now be defined, following \cite{Fopersontheprojectiveline}, as the quotient of the affine space \eqref{affine space} by this gauge action
\begin{equation*}
\Op_{\lg}(\CP) \coloneqq \op_{\lg}(\CP) \big/ \lhN_+(\M).
\end{equation*}
In fact, we shall be interested in certain affine subspaces of $\Op_{\lg}(\CP)$ defined as follows.

\subsection{Twist function $\varphi$} \label{sec: twist function}
Fix a choice of meromorphic function $\varphi$ on $\CP$, called the \emph{twist function}. We call a derivation element $\Lambda$ of $\lh$ \emph{normalised} if $\langle \Lambda, \cent \rangle = 1$. 
Define $\op_{\lg}(\CP)^\varphi$ to be the affine subspace of $\op_{\lg}(\CP)$ consisting of connections of the form
\be d + p_{-1} dz - \Lambda \varphi dz + b' dz, \qquad b' \in \lhb'_+(\M). \nn\ee

\begin{lemma} \label{lem: Lambda indep}
The affine subspace $\op_{\lg}(\CP)^\varphi$ is independent of the choice of normalised derivation element $\Lambda$, and it is stable under $\lhN_+(\M)$-valued gauge transformations.
\begin{proof}
Let $\Lambda$ and $\Lambda'$ be two choices of normalised derivation element of $\lh$. Then we have $\langle \Lambda, \cent \rangle - \langle \Lambda', \cent \rangle = 0$ so that $\Lambda - \Lambda'$ is in the span of the simple roots $\alpha_i$, $i \in I$ and hence $\Lambda - \Lambda' \in \lh \cap \lb'_+$. It follows that $\op_{\lg}(\CP)^\varphi$ is independent of $\Lambda$. 

Since we have the direct sum of vector spaces $\lg = \lgp \oplus \CC \Lambda$, and so in particular $\lb_+ = \lb'_+ \oplus \CC \Lambda$, it follows from the definition of the action of $\lhN_+(\M)$ on $\op_{\lg}(\CP)$ by gauge transformations that $\op_{\lg}(\CP)^\varphi$ is stable.
\end{proof}
\end{lemma}

Given a choice of twist function $\varphi$, we may now define the corresponding affine subspace of $\lg$-opers as
\begin{equation} \label{lg opers twist}
\Op_{\lg}(\CP)^\varphi \coloneqq \op_{\lg}(\CP)^\varphi \big/ \lhN_+(\M).
\end{equation}
If $\nabla \in \op_{\lg}(\CP)^\varphi$ then we shall denote its class in $\Op_{\lg}(\CP)^\varphi$ by $[\nabla]$.

We introduce also the \emph{twisted de Rham differential} corresponding to the twist function $\varphi$. For every $f\in \lhb_+(\M)$, 
\begin{align} \label{twisted de Rham def}
d^\varphi f &\coloneqq df - {h^\vee}^{-1}\varphi (\ad_\rho f)  dz\\
&\,=dz \left( \del_z f - {h^\vee}^{-1} \varphi(\ad_\rho f)  \right).\nn
\end{align}

\subsection{Quasi-canonical form of an $\lg$-oper} \label{sec: quasi-can form}
Recall, from \S\ref{sec: princ sub}, the definition of the principal subalgebra $\a$ of $\lg$, and its basis $\{ p_j \}_{j \in \pm E} \cup \{ \delta, \rho \}$ where $E$ is the multiset of positive exponents of $\lg$.

Let $\hat\a(\M)$ denote the completion of the algebra $\a(\M)$ of $\a$-valued mermorphic functions on $\CP$:
\begin{equation*}
\hat \a(\M) \coloneqq \varprojlim \a(\M)/(\a \cap \n_k)(\M).
\end{equation*}
For each $n \in \ZZ_{\geq 0}$, let $\hat\a_{\geq n}(\M) \coloneqq \varprojlim \a_{\geq n}(\M)/(\a_{\geq n} \cap \n_k)(\M)$, where $\a_{\geq n}\coloneqq \bigoplus_{j=n}^\8 \a_j$. These are Lie subalgebras of $\hat\a(\M)$.

\begin{theorem} \label{thm: quasi-canonical form}
Every class $[\nabla] \in \Op_{\lg}(\CP)^\varphi$ has a representative $\nabla \in \op_{\lg}(\CP)^\varphi$ of the form
\begin{equation*}
\nabla = d + p_{-1} dz - {h^\vee}^{-1} \rho \, \varphi dz + a dz, \qquad a\in \hat\a_{\geq 1}(\M).
\end{equation*}
We say that such a representative is in \emph{quasi-canonical form}. For any $g \in \lhN_+(\M)$, $\nabla^g$ is still in quasi-canonical form if and only if $g = \exp(f) \in \exp( \hat\a_{\geq 2}(\M) )$, in which case
$\nabla^g = \nabla - d^\varphi f$.

Equivalently but more explicitly, every class $[\nabla] \in \Op_{\lg}(\CP)^\varphi$ has a \emph{quasi-canonical} representative of the form 
\begin{equation}
\nabla = d + \Bigg( p_{-1} - \frac{\varphi}{h^\vee} \rho + \sum_{j \in E} v_j p_j  \Bigg) dz,\label{qcf}
\end{equation}
where $v_j$ is a meromorphic function on $\CP$ for each positive exponent $j \in E$. The gauge transformations in $\lhN_+(\M)$ preserving quasi-canonical form are precisely those of the form $\exp\big( \sum_{j \in E_{\geq 2}} f_j p_j \big)$ with $f_j$ meromorphic functions on $\CP$. The effect of such gauge transformations on the functions $v_j$ is to send 
\begin{equation}
v_j \longmapsto v_j - f'_j+ \frac{j \varphi}{h^\vee} f_j\label{vuptof}
\end{equation}
for all $j\in E_{\geq 2}$, and to leave $v_1$ invariant.
\end{theorem}
\begin{proof}
Let $[\nabla] \in \Op_{\lg}(\CP)^\varphi$. Since ${h^\vee}^{-1} \rho$ is a normalised derivation element of $\lh$, see \S\ref{sec: principal grad}, it follows using Lemma \ref{lem: Lambda indep} that there is a representative of $[\nabla]$ of the form $\nabla = d + p_{-1} dz - {h^\vee}^{-1} \rho \, \varphi dz + \sum_{n \geq 0} u_ndz \in \op_{\lg}(\CP)^\varphi$ for some functions $u_n \in \lgp_n(\M)$.

Let $g \in \lhN_+(\M)$ be of the form $g = \exp(m)$ with $m = \sum_{n>0} m_n$ where $m_n \in \c_n(\M)$ for each $n > 0$. Using \eqref{gauge transf} we determine the gauge transformation of $\nabla$ by $g$ to be
\begin{equation*}
\nabla^g = d + p_{-1} dz - {h^\vee}^{-1} \rho \, \varphi dz + \sum_{n \geq 0} a_ndz
\end{equation*}
where $a_n \in\lgp_n(\M)$ for each $n \geq 0$ are of the form
\begin{align} \label{recursion can form}
a_ndz &= u_ndz + [m_{n+1}, p_{-1}] dz + F_n\big( \{ u_k, dm_k, m_k \}_{k<n} \big)\\
&\qquad - d^\varphi m_n + [m_n, u_0] + \ha \big[ m_n, [m_1, p_{-1}] \big] dz + \ha (1 - \delta_{n,1}) \big[ m_1, [m_n, p_{-1}] \big] dz. \notag
\end{align}
The last term on the first line of the right hand side contains all the terms involving only $m_k$ and $u_k$ with $k < n$, and the second line contains those terms involving $m_n$. Let $w_ndz$, $w_n \in \lgp_n(\M)$, denote the sum of all these terms, \emph{i.e.} we rewrite \eqref{recursion can form} as
\begin{equation} \label{recursion can form bis}
a_n = u_n + [m_{n+1}, p_{-1}] + w_n.
\end{equation}
We can now use \eqref{recursion can form bis} to determine $m_n \in \c_n(\M)$ recursively for all $n > 0$ by requiring that $a_n \in \a_n(\M)$ for each $n \geq 0$. Indeed, suppose $m_k$ has been determined for each $k \leq n$. Then $w_n$ is known (in fact $w_0 = 0$ for the base case) and so decomposing $u_n + w_n$ relative to the direct sum \eqref{lg n decompA} (or rather $\lgp_0 = \CC \delta \oplus \c_0$ in the case $n=0$) we can use the injectivity of $\ad_{p_{-1}} : \c_{n+1} \to \c_n$ to fix $m_{n+1}$ uniquely so as to cancel the component of $u_n + w_n$ in $\c_n$, thereby ensuring that $a_n \in \a_n(\M)$ for all $n > 0$ or $a_0 \in (\CC \delta)(\M)$. This proves $\nabla^g \in d + p_{-1} dz - {h^\vee}^{-1} \rho \, \varphi dz + (\CC \delta \oplus \hat\a_{\geq 1})(\M)dz$.

Let us write $\nabla^g = d + p_{-1} dz - {h^\vee}^{-1} \rho \, \varphi dz +  \delta \chi dz + a'dz$ with $\chi \in \M$ and $a' \in \hat\a_{\geq 1}(\M)$. In order to remove the term in $\delta$, we can apply a further gauge transformation by $h = \exp(- \chi p_1)$, which yields
\begin{equation*}
\nabla^{hg} = d + p_{-1} dz - {h^\vee}^{-1} \rho \, \varphi dz + a'dz + d^\varphi (\chi p_1).
\end{equation*}
The last two terms belong to $\hat\a_{\geq 1}(\M)dz$, which completes the proof of the first statement.

Finally, suppose $\nabla = d + p_{-1} dz - {h^\vee}^{-1} \rho \, \varphi dz + udz$ where $u \in \hat\a_{\geq 1}(\M)$ and let $g = \exp(m)$ for some $m \in \lhn_+(\M)$. Then $\nabla^g = d + p_{-1} dz - {h^\vee}^{-1} \rho \, \varphi dz + vdz$ where $v = \sum_{n \geq 0} a_n \in \lhb'_+(\M)$ is given by \eqref{recursion can form}. We want to recursively determine the components $m_n \in \lg_n (\M)$ of $m$ so that $v \in \hat\a_{\geq 1}(\M)$. Considering first the case $n=0$ we have $u_0 = a_0 = 0$ so that \eqref{recursion can form} reduces to $[m_1, p_{-1}] = 0$, and therefore $m_1 = 0$ since $ \ad_{p_{-1}} : \lg_1 \to \lg_0$ is injective. In particular, for every $n \geq 0$ the last two terms on the right hand side of \eqref{recursion can form} are now absent. Suppose that having $a_k \in \a_k(\M)$ for all $k < n$ requires that $m_k \in \a_k(\M)$ for each $k \leq n$. It just remains to show that the condition $a_n \in \a_n(\M)$ also implies $m_{n+1} \in \a_{n+1}(\M)$. For this we note that all the terms contained in $F_n\big( \{ u_k, dm_k, m_k \}_{k<n} \big)$ are commutators, which vanish using the fact that $u \in \hat\a_{\geq 1}(\M)$, $m_1=0$, $m_k \in \a_k(\M)$ for $1 < k \leq n$ and $\a_{\geq 1}$ is abelian. So \eqref{recursion can form} now simply reads
\begin{equation*}
a_ndz - u_ndz + d^\varphi m_n = [m_{n+1}, p_{-1}] dz.
\end{equation*}
The left hand side clearly belongs to $\a_n(\M)dz$, using the fact that $\ad_\rho m_n = n m_n$. On the other hand, the right hand side belongs instead to $\c_n(\M)dz$ since $\c_n = (\textup{im} \ad_{p_{-1}})_n$ for every $n > 0$, cf. Remark \ref{rem: pirem}(\ref{rem: Bn vs im p-1}). Hence both sides vanish so that, in particular, $m_{n+1} \in \a_{n+1}(\M)$. The vanishing of the left hand side is the final statement about the form of $\nabla^g - \nabla$.
\end{proof}

Although the quasi-canonical form of an $\lg$-oper $[\nabla] \in \Op_{\lg}(\CP)^\varphi$ is not unique, the coefficient $v_1$ of $p_1$ in any quasi-canonical form is the same. To emphasise the origin of this distinction between $v_1$  and all the remaining coefficients $v_j$, $j\in E_{\geq 2}$, the following is helpful.
\begin{proposition}
Every class $[\nabla] \in \Op_{\lg}(\CP)^\varphi$ has a representative of the form 
\begin{equation}
\nabla = d + \Bigg( p_{-1} - \frac{\varphi}{h^\vee} \rho + v_0 \delta + \sum_{j \in E} v_j p_j  \Bigg) dz,
\end{equation}
where $v_j$ is a meromorphic function on $\CP$ for each  $j \in \{0\} \cup E$. The gauge transformations in $\lhN_+(\M)$ preserving this form are precisely those of the form $\exp\big( \sum_{j \in E} f_j p_j \big)$ with $f_j$ meromorphic functions on $\CP$. The effect of such gauge transformations on the functions $v_j$ is as in \eqref{vuptof} for  all $j\in E_{\geq 2}$, and now also
\begin{align}
v_0 & \longmapsto v_0 + f_1 \nn\\
v_1 &\longmapsto v_1 - f'_1+ \frac{\varphi}{h^\vee} f_1.\nn
\end{align}
\end{proposition}
\begin{proof}The proof is very similar to that of Theorem \ref{thm: quasi-canonical form}.\end{proof}
Consequently, if one works not with $\lg$ but with the quotient by the centre $\lg/\CC\delta$ then the distinction between $v_1$ and the rest disappears, as follows. (We return to this point in \S\ref{sec: quad Ham} below.)
\begin{corollary}\label{cor: v1}
For an $(\lg/\CC\delta)$--oper $[\nabla] \in \Op_{\lg/\CC\delta}(\CP)^\varphi$, there is always a quasi-canonical representative of the form \eqref{qcf}. The gauge transformations in $\lhN_+(\M)$ preserving this form are precisely those of the form $\exp\big( \sum_{j \in E} f_j p_j \big)$ with $f_j$ meromorphic functions on $\CP$. The effect of such gauge transformations on the functions $v_j$ is as in \eqref{vuptof} but now for all $j\in E$ (including $1$).\qed
\end{corollary}

Returning to $\lg$-opers, we have the following explicit expression for the coefficient $v_1$ in any quasi-canonical form. 

\begin{proposition} \label{prop: can form u1}
The coefficient of $p_1\in \a_1$ of any quasi-canonical form of an $\lg$-oper $[\nabla] \in \Op_{\lg}(\CP)^\varphi$ is
\begin{equation*}
v_1 = {h^\vee}^{-1} \big( \ha (u_0 | u_0) + (\rho | u'_0) - {h^\vee}^{-1} \varphi (\rho | u_0) + (p_{-1} | u_1) \big),
\end{equation*}
where 
\be \nabla = d + p_{-1} dz - {h^\vee}^{-1} \rho \, \varphi dz + \sum_{n \geq 0} u_n dz \in \op_{\lg}(\CP)^\varphi, \nn\ee 
with $u_n \in \lgp_n$, is any representative of $[\nabla]$.
\begin{proof}
In the present case, the recursion relation \eqref{recursion can form} for $n=0$ gives $u_0 = - [m_1, p_{-1}]$. Note that here we are including in $m_1$ the term $-\chi p_1$ coming from the subsequent gauge transformation performed in the second step of the proof of Theorem \ref{thm: quasi-canonical form}. Using this, the relation \eqref{recursion can form} for $n=1$ then reads
\begin{align*}
a_1 dz &= u_1 dz + [m_2, p_{-1}] dz + \ha \big[ m_1, [m_1, p_{-1}] \big] dz + [m_1, u_0] dz - d^\varphi m_1\\
&= u_1 dz + [m_2, p_{-1}] dz - \ha \big[ m_1, [m_1, p_{-1}] \big] dz - d^\varphi m_1.
\end{align*}
By applying the linear map $(p_{-1}| \cdot)$ to both sides we find
\begin{equation*}
(p_{-1} | a_1)dz = (p_{-1} | u_1) dz + \ha (u_0| u_0) dz - (p_{-1} | d^\varphi m_1),
\end{equation*}
where to obtain the second term on the right hand side we have used again the fact that $u_0 = - [m_1, p_{-1}]$. To evaluate further the last term above, we note that
\begin{equation*}
(p_{-1} | d^\varphi m_1) = ([p_{-1}, \rho] | d^\varphi m_1) = (\rho | [d^\varphi m_1, p_{-1}]) = - (\rho | d u_0) + {h^\vee}^{-1} \varphi (\rho | u_0) dz,
\end{equation*}
where in the last step we used the definition \eqref{twisted de Rham def} of the twisted de Rham differential.
Since $(p_{-1} | p_1) = h^\vee$, we arrive at the desired expression for $v_1 = {h^\vee}^{-1} (p_{-1} | a_1)$.
\end{proof}
\end{proposition}

\begin{remark} \label{rem: can form u1}
Let $\nabla \in \op_{\lg}(\CP)^\varphi$ be as in the statement of Proposition \ref{prop: can form u1} and introduce $\wt u_0 \coloneqq - {h^\vee}^{-1} \rho \,\varphi + u_0 \in \lg_0(\M)$ and $\wt u_n \coloneqq u_n \in \lg_n(\M)$ for every $n > 0$. Then we have 
\be \nabla = d + p_{-1} dz + \sum_{n \geq 0} \wt u_n dz,\nn\ee 
and, using the fact that $(\rho | \rho) = 0$, cf. \S\ref{sec: principal grad}, the expression for the coefficient $v_1$ in any quasi-canonical form of $[\nabla]$ can be rewritten as
\begin{equation*}
v_1 = {h^\vee}^{-1} \big( \ha (\wt u_0 | \wt u_0) + (\rho | \wt u'_0) + (p_{-1} | \wt u_1) \big) dz. \qedhere
\end{equation*}
\end{remark}

\subsection{Twisted homology and functions on the space of affine opers} \label{sec: twisted homology coord}
Our goal is to describe functions $\Op_{\lg}(\CP)^\varphi\to \CC$ on the space of meromorphic $\lg$-opers on $\CP$. 

Theorem \ref{thm: quasi-canonical form} shows that one well-defined map $\Op_{\lg}(\CP)^\varphi\to \M$ is given by extracting the coefficient $v_1$ of $p_1$ in any quasi-canonical form (and Proposition \ref{prop: can form u1} gives the explicit formula). Obviously we can then ``pair'' this function with any point $p\in \CP$ where $v_1$ doesn't have a pole, by simply evaluating it there, $v_1 \mapsto v_1(p)$. 

Yet Theorem \ref{thm: quasi-canonical form} also shows that the remaining data in the oper comes in the form of functions $v_i$, $i\in E_{\geq 2}$, defined only up to certain ``twisted'' derivatives. So they are in some sense cohomology elements.
In \S\ref{sec: coord} we shall make that idea precise by showing that each of the functions $v_i$, $i\in E_{\geq 2}$, represents a cocycle in the cohomology of the de Rham complex with coefficients in a certain local system. A generalization of the usual de Rham theorem states that there is a pairing (given by integrating) between such cocycles and the cycles of the singular homology with coefficients in the dual local system. For the moment though, we are not quite in a position to invoke such results: a local system is a vector bundle with a flat connection and we cannot yet identify the correct bundle, since we have no handle on its transition functions.

Nonetheless, it is already possible to define the integrals one should take to obtain functions  $\Op_{\lg}(\CP)^\varphi\to \CC$, as follows. 

First, let us now and for the remainder of this article restrict attention to the case when the twist function $\varphi$ has only simple poles, \emph{i.e.} we shall take it to be of the form
\begin{equation} \label{twist function}
\varphi(z) \coloneqq \sum_{i=1}^N \frac{k_i}{z - z_i},
\end{equation}
for some $k_i \in \CC^\times$, $i=1,\ldots, N$. It has simple poles in the subset $\{ z_i \}_{i=1}^N \subset \CP$. 

\begin{remark} 
Based on the situation in finite types, \cite{FFT,VY3}, and the shift of argument affine Gaudin model introduced in \cite{FFsolitons}, our expectation is that introducing a pole of order $p\geq 2$ at $z_i$ in the twist function $\varphi$ (and more generally in the Miura $\lg$-opers of \S\ref{sec: class of Miura opers} below) will correspond to a Gaudin model in which one assigns to the marked point $z_i$ a representation of a \emph{Takiff algebra} $\g[t]/t^p\g[t]$ over the affine Kac-Moody algebra $\g$. 
\end{remark} 

We denote the complement of the set of marked points $\{ z_i \}_{i=1}^N$ as
\begin{equation} \label{set X def}
X \coloneqq \CC \setminus \{ z_i \}_{i=1}^N.
\end{equation}
Consider the multivalued holomorphic function $\P$ on $X$ defined by
\begin{equation}\label{def: P}
\P(z) \coloneqq \prod_{i=1}^N (z - z_i)^{k_i},
\end{equation}
which is related to the twist function as $\varphi(z) = \partial_z \log \P(z)$.
Observe that the ambiguity in the function $v_j$, namely \eqref{vuptof}, can be expressed as
\begin{equation*}
\P(z)^{-j/h^\vee} v_j(z) \longmapsto \P(z)^{-j/h^\vee} v_j(z) - \partial_z \big( \P(z)^{-j/h^\vee} f_j(z) \big).
\end{equation*}
We therefore obtain the following corollary of Theorem \ref{thm: quasi-canonical form}.

\begin{corollary}\label{cor: opint}
Suppose 
\be \nabla = d + p_{-1} dz - {h^\vee}^{-1} \rho \, \varphi dz + \sum_{j \in E} v_j p_j dz\nn\ee 
is a quasi-canonical form of an oper $[\nabla]\in \Op_{\lg}(\CP)^\varphi$. 

Let $r\in E_{\geq 2}$ be a positive exponent greater than or equal to 2. 
Let $\gamma$ be any contour in $X = \CC \setminus \{ z_i \}_{i=1}^N$ such that
\begin{enumerate}
\item $\gamma$ is closed;
\item there exists a single-valued branch of the function $\P^{-r/h^\vee}$ along $\gamma$;
\item $v_r$ has no poles (and is therefore holomorphic) along $\gamma$.
\end{enumerate}
 
Then the following integral is gauge-invariant, \emph{i.e.} it depends only on the oper $[\nabla]$ and is independent of the choice of quasi-canonical form:
\be I_r^\gamma([\nabla]) \coloneqq \int_\gamma \P(z)^{-r/h^\vee} v_r(z)dz.\nn\ee
This function is invariant under smooth deformations of the contour $\gamma$ which do not cross any pole of $v_r$ or any of the marked points $\{z_i\}_{n=1}^\8$. \qed
\end{corollary}

\section{Miura $\lg$-opers and the Bethe equations} \label{sec: Miura opers}

\subsection{A class of Miura $\lg$-opers} \label{sec: class of Miura opers}

Following \cite{Fopersontheprojectiveline}, we define a \emph{Miura $\lg$-oper} as a connection of the form
\begin{equation} \label{g-oper with twist}
\nabla \coloneqq d + p_{-1} dz + u \, dz \in \op_{\lg}(\CP)
\end{equation}
where $u \in \null^L \h(\M) = \h^\ast(\M)$, using the natural identification $\null^L \h = \h^\ast$.
Let
%\begin{equation*}
$\MOp_{\lg}(\CP)$ %\coloneqq d + p_{-1} dz + \lh(\M) dz.
%\end{equation*}
denote the affine space of all Miura $\lg$-opers. Given a Miura $\lg$-oper $\nabla \in \MOp_{\lg}(\CP)$ we refer to its class $[\nabla] \in \Op_{\lg}(\CP)$ as the underlying $\lg$-oper.

Recall the twist function $\varphi \in \M$ defined in \eqref{twist function}. Given any choice of normalised derivation element $\Lambda$ of $\lh$, cf. \S\ref{sec: def oper}, we introduce the affine subspace
\begin{equation} \label{MOp with twist}
\MOp_{\lg}(\CP)^\varphi \coloneqq d + p_{-1} dz - \Lambda \varphi dz + \lh'(\M) dz
\end{equation}
of $\MOp_{\lg}(\CP)$ where $\lh'$ is the span of the simple roots $\{ \alpha_i \}_{i=0}^\ell$. It follows from the first part of the proof of Lemma \ref{lem: Lambda indep} that $\MOp_{\lg}(\CP)^\varphi$ is independent of the choice of normalised derivation $\Lambda$.

In this paper we shall be interested in Miura $\lg$-opers \eqref{g-oper with twist} where the meromorphic $\h^\ast$-valued function $u \in \h^\ast(\M)$ has at most simple poles. Fix a collection of weights $\lambda_1, \ldots, \lambda_N \in \h^\ast$. We shall, more specifically, be interested in the case when $u$ has a simple pole at each marked point $z_i$, $i =1, \ldots, N$, with residue $-\lambda_i \in \lh$. We will furthermore allow the function $u$ to have simple poles at some additional $m \in \ZZ_{\geq 0}$ marked points $w_j$, $j = 1, \ldots, m$, with residues there given by simple roots $\alpha_{c(j)}$, for some function $c : \{ 1, \ldots, m\} \to I= \{ 0,\ldots, \ell \}$. In other words, we shall consider Miura $\lg$-opers of the form \cite{Fopersontheprojectiveline, FFsolitons}
\begin{equation} \label{u Miura op def}
\nabla = d + p_{-1} dz - \sum_{i=1}^N \frac{\lambda_i}{z - z_i} dz + \sum_{j=1}^m \frac{\alpha_{c(j)}}{z - w_j} dz.
\end{equation}
The residue of $\nabla$ at infinity is the weight $\lambda_\infty \coloneqq \sum_{i=1}^N \lambda_i - \sum_{j=1}^m \alpha_{c(j)} \in \lh$.

Decomposing each weight $\lambda_i \in \h^\ast$ with respect to the basis $\{ \alpha_i \}_{i=1}^\ell \cup \{ \rho, \delta \}$, we may write it as
\begin{equation} \label{hw lambda i}
\lambda_i = \dot{\lambda}_i + \frac{k_i}{h^\vee} \rho - \Delta_i \delta
\end{equation}
for some $\dot{\lambda}_i \in \dot\h^\ast \coloneqq \textup{span}_{\CC} \{ \alpha_j \}_{j=1}^\ell$, $k_i \coloneqq \langle \lambda_i, \cent \rangle \in \CC$ and $\Delta_i \coloneqq -\langle \lambda_i, \cocent \rangle \in \CC$. Since $\dot\lambda_i$, $\delta$ and the simple roots $\alpha_{c(j)}$ all lie in $\lh'$, it follows that $\nabla$ belongs to the space $\MOp_{\lg}(\CP)^\varphi$ with the twist function $\varphi$ defined as in \eqref{twist function} in terms of the $k_i$, $i=1,\ldots, N$.

\subsection{Regular points} \label{sec: regular points}

Let $\M^{\rm reg}_x$ be the $\CC$-algebra of meromorphic functions on $\CP$ which are holomorphic at $x$.
We shall say that an $\lg$-connection $\nabla = d + p_{-1} dz + b \, dz$ in $\op_{\lg}(\CP)$ is \emph{regular} at a point $x \in \CC$ if in fact $b \in \lhb_+(\M^{\rm reg}_x)$, \emph{i.e.} $b$ has no pole at $x$. Let $\op^{\rm reg}_{\lg}(\CP)_x$ denote the set of all such $\lg$-connections. It is stabilised by the action of the subgroup $\lhN_+(\M^{\rm reg}_x) \subset \lhN_+(\M)$ on $\op_{\lg}(\CP)$ by gauge transformations. In particular, we can define the quotient space
\begin{equation*}
\Op^{\rm reg}_{\lg}(\CP)_x \coloneqq \op^{\rm reg}_{\lg}(\CP)_x \big/ \lhN_+(\M^{\rm reg}_x).
\end{equation*}
If $x$ is not a pole of the twist function $\varphi$ we may similarly define the space $\op^{\rm reg}_{\lg}(\CP)^\varphi_x$ of $\lg$-connections of the form $\nabla = d + p_{-1} dz - \Lambda \varphi dz + b' dz$ where $\Lambda$ is a normalised derivation element of $\lh$ and $b' \in \lhb'_+(\M^{\rm reg}_x)$. We then also define
\begin{equation*}
\Op^{\rm reg}_{\lg}(\CP)^\varphi_x \coloneqq \op^{\rm reg}_{\lg}(\CP)^\varphi_x \big/ \lhN_+(\M^{\rm reg}_x).
\end{equation*}

\begin{lemma}
For each $x\in \CC$ there is a canonical injection
\begin{equation} \label{Op reg to Op}
\Op^{\rm reg}_{\lg}(\CP)_x \longhookrightarrow \Op_{\lg}(\CP).
\end{equation}
When $x$ is not a pole of $\varphi$ there is a canonical injection $\Op^{\rm reg}_{\lg}(\CP)^{\varphi}_x \into \Op_{\lg}(\CP)^{\varphi}$.
\begin{proof}
Since $\lhN_+(\M^{\rm reg}_x) \subset \lhN_+(\M)$ we certainly have a well-defined canonical map $\Op^{\rm reg}_{\lg}(\CP)_x \to \Op_{\lg}(\CP)$.
Suppose that two $\lg$-connections $\nabla, \nabla' \in \op^{\rm reg}_{\lg}(\CP)_x$, regular at $x$, define the same class $[\nabla] = [\nabla']$ in $\Op_{\lg}(\CP)$. We must show that they also define the same class in $\Op^{\rm reg}_{\lg}(\CP)_x$.

Applying the procedure in the first half of the proof of Theorem \ref{thm: quasi-canonical form} to both of the $\lg$-connections $\nabla, \nabla' \in \op^{\rm reg}_{\lg}(\CP)_x$, with $\M$ there replaced by $\M^{\rm reg}_x$, we find that they can each be brought to a quasi-canonical form which is regular at $x$ using a gauge transformation in $\lhN_+(\M^{\rm reg}_x)$. On the other hand, by the argument in the second half of the proof of Theorem \ref{thm: quasi-canonical form} with $\M$ there replaced by $\M^{\rm reg}_x$, we also deduce that these two quasi-canonical forms are related by a gauge transformation in $\exp(\hat \a_{\geq 2}(\M^{\rm reg}_x))$. It now follows that $\nabla$ and $\nabla'$ define the same class in $\Op^{\rm reg}_{\lg}(\CP)_x$.
\end{proof}
\end{lemma}

We will identify $\Op^{\rm reg}_{\lg}(\CP)_x$ with its image in $\Op_{\lg}(\CP)$ under the injection \eqref{Op reg to Op}. We then say that an $\lg$-oper $[\nabla] \in \Op_{\lg}(\CP)$ is \emph{regular} at $x \in \CP$ if it lies in $\Op^{\rm reg}_{\lg}(\CP)_x$. More concretely, this means that there exists a representative of the class $[\nabla]$ in $\op_{\lg}^{\rm reg}(\CP)_x$, \emph{i.e.} which has no pole at $x$. 

Recall the set $X = \CC \setminus \{ z_i \}_{i=1}^N$ introduced in \S\ref{sec: twisted homology coord}.
We define the space of \emph{$\lg$-opers regular on $X$} as
\begin{equation*}
\Op^{\rm reg}_{\lg}(\CP)_X \coloneqq \bigcap_{x \in X} \Op^{\rm reg}_{\lg}(\CP)_x \subset \Op_{\lg}(\CP).
\end{equation*}
Since the twist function has no poles in $X$, we may also define the space of \emph{$\lg$-opers with twist function $\varphi$ regular on $X$} as
\begin{equation*}
\Op^{\rm reg}_{\lg}(\CP)^\varphi_X \coloneqq \bigcap_{x \in X} \Op^{\rm reg}_{\lg}(\CP)^\varphi_x \subset \Op_{\lg}(\CP)^\varphi.
\end{equation*}

\subsection{Bethe equations}

%Let $\Aut \lh$ denote the group of automorphisms of the vector space $\h^\ast = \lh$. For each $i \in I$ we introduce the simple reflection $r_i \in \Aut \lh$ as
%\begin{equation*} r_i \lambda \coloneqq \lambda - \langle \lambda, \check\alpha_i \rangle \alpha_i, \end{equation*}
%for all $\lambda \in \lh$. The Weyl group $W$ of $\lg$ is defined as the subgroup of $\Aut \lh$ generated by the set of all simple reflections $r_i$, $i \in I$. We also define the shifted Weyl action of $r_i$ on $\lambda \in \lh$ as $r_i \cdot \lambda \coloneqq r_i(\lambda + \rho) - \rho$, with $\rho$ defined as in \S\ref{sec: principal grad}.

\begin{proposition} \label{prop: Miura oper BAE}
Let $x \in X$ and suppose $\nabla \in \MOp_{\lg}(\CP)^\varphi$ has the form
\begin{equation*}
\nabla = d + \Big( p_{-1} - {h^\vee}^{-1} \varphi \, \rho + \frac{\alpha_i}{z - x} + r \Big) dz
\end{equation*}
for some simple root $\alpha_i$, $i \in I$, where $r \in \lh'(\M)$ is regular at $x$. Then $[\nabla]$ is regular at $x$, \emph{i.e.} there is a representative of $[\nabla]$ which is regular at $x$, if and only if
\begin{equation} \label{affine BAE}
h^\vee \langle r(x), \check\alpha_i \rangle = \varphi(x).
\end{equation}
\begin{proof}
Suppose first that $\wt\nabla$ is a representative of $[\nabla]$ which is regular at $x$. Then the gauge transformation parameter $g \in \lhN_+(\M)$ determined by following the recursive procedure of Theorem \ref{thm: quasi-canonical form} is of the form $g = \exp(-\chi \delta) \exp(\sum_{n>0} m_n)$ where $\chi \in \M^{\rm reg}_x$ and $m_n \in \c_n(\M^{\rm reg}_x)$, $n > 0$ are all regular at $x$. Therefore, the quasi-canonical form $\wt\nabla^g$ of $[\nabla]$ is regular at $x$. Then, in particular, its component in $\a_1$ must be regular. Yet by Proposition \ref{prop: can form u1} the latter is proportional to (note that in the notation of Proposition \ref{prop: can form u1} we have $u_0 = \frac{\alpha_i}{z - x} + r$ and $u_1 = 0$ in the present case)
\begin{equation*}
\frac{(- \alpha_i | - \alpha_i + 2 \rho)}{2(z-x)^2} dz + \frac{(\alpha_i | r(x)) - {h^\vee}^{-1} \varphi(x) (\alpha_i | \rho)}{z - x} dz + \ldots
\end{equation*}
where the dots represent terms regular at $z = x$. Recalling that $\langle \rho, \check \alpha_i \rangle = 1$ for all $i \in I$, and in view of \eqref{bilinear form def}, we see that the double pole term here vanishes and the simple pole term vanishes only if the equation \eqref{affine BAE} holds. 

Conversely, suppose \eqref{affine BAE} holds. Let $g = \exp \big( \! -\frac{1}{z-x} \check e_i \big)$. For all $u \in \lh(\M)$ we have
\begin{align*}
(dg) g^{-1} &= \check e_i \frac{dz}{(z-x)^2}, \qquad
g u g^{-1} = u - \frac{1}{z-x} [\check e_i, u] = u + \frac{\langle u, \check \alpha_i \rangle}{z-x} \check e_i,\\
g p_{-1} g^{-1} &= p_{-1} - \frac{1}{z-x} [\check e_i, p_{-1}] + \frac{1}{2(z-x)^2} \big[ \check e_i, [\check e_i, p_{-1}] \big] = p_{-1} - \frac{\alpha_i}{z-x} - \frac{\check e_i}{(z-x)^2}.
\end{align*}
Therefore, with $u = - {h^\vee}^{-1} \varphi\, \rho + \frac{\alpha_i}{z-x} + r$ we find
\begin{align*}
\nabla^g &= d - (dg) g^{-1} + g p_{-1} g^{-1} dz + g u g^{-1} dz\\
&= d + \Big( p_{-1} - {h^\vee}^{-1} \varphi\, \rho + r(z) + \frac{\langle r(z), \check \alpha_i \rangle - {h^\vee}^{-1} \varphi(z)}{z-x} \check e_i \Big) dz.
\end{align*}
(The coefficient of the $(z-x)^{-2}$ term is $-1-1+\langle \alpha_i,\check\alpha_i\rangle=0$.)
This is regular at $x$ by virtue of \eqref{affine BAE}, so the $\lg$-oper $[\nabla] = [\nabla^g]$ is regular at $x$.
\end{proof}
\end{proposition}

\begin{remark} \label{rem: Miura oper BAE}
In the statement of Proposition \ref{prop: Miura oper BAE}, if we write $\nabla \in \op_{\lg}(\CP)^\varphi$ as
\begin{equation*}
\nabla = d + \Big( p_{-1} + \frac{\alpha_i}{z - x} + \wt r \Big) dz
\end{equation*}
where $\wt r \in \lh(\M)$ is regular at $x$, noting that $\varphi$ is regular at $x \in X$, then the Bethe equation \eqref{affine BAE} for the regularity of $[\nabla]$ at $x$ simply reads 
\begin{equation*} \langle \wt r(x), \check\alpha_i \rangle = 0.\qedhere\end{equation*}
\end{remark}

Recall the subset $X = \CC \setminus \{ z_i \}_{i=1}^N$ of $\CP$ introduced in \S\ref{sec: twisted homology coord}.

\begin{corollary} \label{cor: Bethe equations}
Let $\nabla \in \MOp_{\lg}(\CP)$ be of the form \eqref{u Miura op def}. We have $[\nabla] \in \Op^{\rm reg}_{\lg}(\CP)_X$ if and only if
\begin{equation} \label{Bethe equations}
- \sum_{i=1}^N \frac{(\lambda_i|\alpha_{c(j)})}{w_j - z_i} + \sum_{\substack{i =1\\ i \neq j}}^m \frac{(\alpha_{c(i)}|\alpha_{c(j)})}{w_j - w_i} = 0, \qquad j= 1, \ldots, m.
\end{equation}
We refer to these as the \emph{Bethe equations}.

In particular, when the Bethe equations hold then there exists a quasi-canonical representative of $[\nabla]$ in which the coefficient functions $v_i(z)$, $i\in E$, have no singularities at the Bethe roots $w_j$, $j=1,\dots,m$.
\begin{remark}
The Bethe equations \eqref{Bethe equations} coincide with those obtained in \cite{FFsolitons} and more generally in \cite{Fopersontheprojectiveline} for a general Kac-Moody algebra.
\end{remark}
\begin{proof}
The $\lg$-oper $[\nabla]$ is certainly regular away from the points $z_i$, $i =1, \ldots, N$ and $w_j$, $j=1, \ldots, m$, since the defining representative $\nabla$ in  \eqref{u Miura op def} is regular there.

And by Proposition \ref{prop: Miura oper BAE}, see also Remark \ref{rem: Miura oper BAE}, the $\lg$-oper $[\nabla]$ is also regular at each of the $w_j$ if and only if the $j^{\rm th}$ Bethe equation \eqref{Bethe equations} holds.
\end{proof}
\end{corollary}

Define the \emph{master function} to be
\begin{align} \label{Master function}
\Phi &\coloneqq \sum_{\substack{i, j=1\\ i < j}}^N (\lambda_i|\lambda_j) \log(z_i - z_j) - \sum_{i=1}^N \sum_{j=1}^m (\lambda_i|\alpha_{c(j)}) \log (z_i - w_j) \notag\\
&\qquad\qquad\qquad\qquad\quad + \sum_{\substack{i, j=1\\ i < j}}^m (\alpha_{c(i)}|\alpha_{c(j)}) \log (w_i - w_j).
\end{align}
It is a multivalued function on $\CC\setminus\{z_1,\dots,z_N,w_1,\dots,w_m\}$. One sees that the Bethe equations \eqref{Bethe equations} are given by
\begin{equation*}
\frac{\partial \Phi}{\partial w_j} = 0, \qquad j = 1, \ldots, m. 
\end{equation*}
Moreover it is known -- see Appendix \ref{sec: hyp arr} for a brief review -- that the eigenvalues of the quadratic Hamiltonians \eqref{quad Ham intro} are given in terms of the partial derivates $\partial \Phi / \partial z_i$.

The following result shows that the partial derivatives of the master function can be read off from the $\lg$-oper underlying the Miura $\lg$-oper $\nabla$ of \eqref{u Miura op def}.

\begin{theorem} \label{thm: Miura oper p1 coeff}
Let $\nabla \in \MOp_{\lg}(\CP)$ be a Miura oper of the form \eqref{u Miura op def}. The coefficient of $p_1$ in any quasi-canonical form of the underlying $\lg$-oper $[\nabla] \in \Op_{\lg}(\CP)$ is 
\begin{equation*}
\frac{1}{h^\vee} \Bigg( \sum_{i=1}^N \frac{\ha (\lambda_i | \lambda_i + 2 \rho)}{(z- z_i)^2} + \sum_{i=1}^N \frac{\partial \Phi / \partial z_i}{z - z_i} + \sum_{j=1}^m \frac{\partial \Phi / \partial w_j}{z - w_j} \Bigg) dz.
\end{equation*}
\begin{proof}

Let us write the Miura $\lg$-oper in \eqref{u Miura op def} as $\nabla = d + p_{-1} dz + u(z) dz$ with
\begin{equation*}
u(z) = \frac{\alpha_{c(j)}}{z - w_j} + \wt r(z), \qquad
\wt r(z) \coloneqq - \sum_{i=1}^N \frac{\lambda_i}{z - z_i} + \sum_{\substack{i=1\\ i \neq j}}^m \frac{\alpha_{c(i)}}{z - w_i}.
\end{equation*}
The result follows from a direct computation, using the expression given in Remark \ref{rem: can form u1} for the coefficient of $p_1$ in any quasi-canonical form of the $\lg$-oper $[\nabla]$, with $\wt u_0 = u$ given above and $\wt u_1 = 0$. Explicitly, we find on the one hand
\begin{equation*}
\ha (u(z)|u(z)) = \frac{1}{2} \sum_{i=1}^N \frac{(\lambda_i|\lambda_i)}{(z - z_i)^2} + \frac{1}{2} \sum_{j=1}^m \frac{(\alpha_{c(j)}| \alpha_{c(j)})}{(z - w_j)^2} + \sum_{i=1}^N \frac{\partial \Phi / \partial z_i}{z - z_i} + \sum_{j=1}^m \frac{\partial \Phi / \partial w_j}{z - w_j},
\end{equation*}
where the derivatives of the master function \eqref{Master function} with respect to the variables $z_i$, $i = 1, \ldots, N$ and $w_j$, $j = 1, \ldots, m$ read
\begin{equation*}
\frac{\partial \Phi}{\partial z_i} = \sum_{\substack{j=1\\ j \neq i}}^N \frac{(\lambda_i|\lambda_j)}{z_i - z_j} - \sum_{j=1}^m \frac{(\lambda_i|\alpha_{c(j)})}{z_i - w_j}, \qquad
\frac{\partial \Phi}{\partial w_j} = - \sum_{i=1}^N \frac{(\lambda_i|\alpha_{c(j)})}{w_j - z_i} + \sum_{\substack{i =1\\ i \neq j}}^m \frac{(\alpha_{c(i)}|\alpha_{c(j)})}{w_j - w_i}.
\end{equation*}
On the other hand, we also have
\begin{equation*}
(\rho|u'(z)) = \sum_{i=1}^N \frac{(\rho|\lambda_i)}{(z - z_i)^2} - \sum_{j=1}^m \frac{(\rho| \alpha_{c(j)})}{(z - w_j)^2}.
\end{equation*}
Adding the above and using the fact that $2 (\alpha_i | \rho) = (\alpha_i | \alpha_i)$ for any simple root $\alpha_i$ (since $\langle \rho, \check\alpha_i \rangle = 1$) we obtain the result.
\end{proof}
\end{theorem}

\section{Conjectures on affine Gaudin Hamiltonians} \label{sec: main conj}

Before turning to the affine case, let us recall some features of the situation in finite types. 
When $\g$ is a Kac-Moody algebra of finite type, the \emph{quantum Gaudin algebra} is a commutative subalgebra of $U(\g^{\oplus N})$ generated by the coefficients in the partial fraction decompositions of a finite collection of $U(\g^{\oplus N})$-valued meromorphic functions $S_k(z)$, indexed by the exponents $k\in\bar E$. The $S_k(z)$ have poles at the marked points $z_1,\dots,z_N$. They commute amongst themselves and with the diagonal action of $\g$. In particular, $1\in \bar E$, and the explicit form of $S_1(z)$ is 
\begin{equation} \label{S1 def finite}
S_1(z) = \sum_{i=1}^N \frac{\C^{(i)}}{(z - z_i)^2} + \sum_{i=1}^N \frac{\H_i}{z - z_i},
\end{equation}
where the $\H_i$ are the quadratic Hamiltonians in \eqref{quad Ham intro} and where $\C^{(i)}$ is the copy of the Casimir element $\C\in U(\g)^\g$ in the $i^{\rm th}$ tensor factor of $U(\g^{\oplus N})$. More generally, the pole terms of highest order in each $S_k(z)$ are $\sum_{i=1}^N \C_{k+1}^{(i)}\big/(z - z_i)^{k+1}$, where $\C_{k+1}\in U(\g)^\g$ is a central element -- as indeed it must be for $S_k(z)$ to commute with the diagonal action of $\g$. Each $S_k(z)$ has degree $k+1$ as an element of $U(\g^{\oplus N})$.

As we sketched in the introduction, to each Miura $\lg$-oper of the form \eqref{mop}, with the Bethe roots $w_i$ obeying the Bethe equations, there corresponds a joint eigenvector $\psi$ of the functions $S_k(z)$, and the joint eigenvalues are given by the coefficients $\bar v_k(z)$ of the $\bar p_k$ in the canonical form of the underlying $\lg$-oper.

Now, not all these features can be precisely preserved in the affine case. Indeed, in finite types the centre $U(\g)^\g$ is isomorphic (via the Harish-Chandra isomorphism) to a graded polynomial algebra $\CC[\{\C_{k+1}\}_{k\in \bar E}]$ in $\rank\g$ generators of the correct degrees. But in affine types the centre is much smaller. Namely, the centre of the (completed, as in \S\ref{sec: ao} below) envelope of an affine Kac-Moody algebra is isomorphic to the graded polynomial algebra in only two generators, $\cent$ and $\C$ (of degrees $0$ and $2$; the definition of $\C$ in the affine case is in \eqref{Omega Kac} below) \cite{CIcentre}. Thus, one should not expect to find meromorphic functions $S_k(z)$, indexed by the positive exponents $k\in E$, such that they commute with the diagonal action of $\g$ for each $z\in X= \CC\setminus\{z_1,\dots,z_N\}$ \emph{and} have degrees $k+1$.\footnote{That would be impossible in any type with an even exponent, since any polynomial in $\cent$ and $\C$ has degree $k+1$ with $k$ odd; in other types these considerations merely make it seem unnatural.} 

This is consistent with the results in \S\ref{sec: opers} -- \S\ref{sec: Miura opers} above: we saw in Theorem \ref{thm: quasi-canonical form} that the coefficients $v_k(z)$ of the quasi-canonical form of an $\lg$-oper are defined, for $k\in E_{\geq 2}$, only up to the addition of twisted derivatives. So they themselves are not good candidates for the eigenvalues of such would-be generating functions. But we also saw that there are well-defined  functions on the space of opers given by integrals, as in Corollary \ref{cor: opint}. It is natural to think that these functions are the eigenvalues of higher Gaudin Hamiltonians. That in turn suggests  that such Hamiltonians are \emph{themselves} given by such integrals. This is the content of Conjecture \ref{conj: higher Ham} below. 
To state it, we must define an appropriate completion of $U(\g^{\oplus N})$ when $\g$ is of untwisted affine type.

%\subsection{The algebra of observables of the quantum Gaudin model}
\subsection{Completion of $U(\g^{\oplus N})$}
\label{sec: ao}
Let $\g = \g(A)$ be an untwisted affine Kac-Moody algebra as in \S\ref{sec: Cartan data}. Let $\dot A = (A_{ij})_{i,j=1}^\ell$ denote the Cartan matrix of finite type obtained from the Cartan matrix $A$ of affine type by removing the $0^{\rm th}$ row and column, and $\dot\g\coloneqq \g(\dot A)$ the corresponding finite-dimensional simple Lie algebra.
The Lie algebra $\g$ can be realised as the semi-direct product $\hat{\mathcal L}\dot \g \rtimes \CC t \del_t$ of the central extension $\hat{\mathcal L}\dot \g \cong_\CC\mathcal L \dot \g \oplus \CC \cent$ 
of the loop algebra $\mathcal L \dot \g \coloneqq \dot\g[t, t^{-1}]$ with derivation element $\cocent$ acting as the derivative $t \del_t$ in the formal loop variable $t$. In what follows we shall identify $\g$ with $\hat{\mathcal L}\dot \g \rtimes \CC t \del_t$.

Let $\g^{\oplus N}$ denote the $N$-fold direct sum of $\g$. We denote by $X^{(i)}$ the copy of any $X \in \g$ in the $i^{\rm th}$ summand, for $i = 1, \ldots, N$.
Consider the left ideals $\mathcal I_n \coloneqq U(\g^{\oplus N}) (t^n\dot\g[t])^{\oplus N}$, for $n \in \ZZ_{\geq 0}$, of the universal enveloping algebra $U(\g^{\oplus N})$. They define a descending $\ZZ_{\geq 0}$-filtration on $U(\g^{\oplus N})$, that is to say we have $\mathcal I_0 \supset \mathcal I_1 \supset \mathcal I_2 \supset \ldots$ with $\cap_{n \geq 0} \mathcal I_n = \{ 0 \}$. Define the corresponding completion of $U(\g^{\oplus N})$ as the inverse limit
\begin{equation*}
\hat U(\g^{\oplus N}) \coloneqq \varprojlim_n U(\g^{\oplus N}) / \mathcal I_n.
\end{equation*}
By definition, an element of $\hat U(\g^{\oplus N})$ is a possibly infinite sum 
\be x = \sum_{m\geq 0} x_m\label{cex}\ee 
of elements in $U(\g^{\oplus N})$, with $x_m \in \mathcal I_m$ for all $m > 0$ so that only finitely many terms contribute when one works modulo any $\mathcal I_n$.
Since the $\mathcal I_n$, $n \geq 0$ are only left ideals, the quotients $U(\g^{\oplus N}) / \mathcal I_n$ are not associative algebras. However, the multiplication in $U(\g^{\oplus N})$ is continuous with respect to the linear topology whose basis of open neighbourhoods for $0$ is $\{ \mathcal I_n \}_{n \geq 0}$. So the completion $\hat U(\g^{\oplus N})$ is an associative algebra. 
%Let us call it the \emph{algebra of observables} of the Gaudin model associated with $\g$.

The tensor product $\bigotimes_{j=1}^N L_{\lambda_j}$ of irreducible $\g$-modules is \emph{smooth} as a module over  $U(\g^{\oplus N})$, meaning that for every $v\in \bigotimes_{j=1}^N L_{\lambda_j}$ there exists $n\in \ZZ_{\geq 0}$ such that $\mathcal I_n \on v = 0$. Therefore $\bigotimes_{j=1}^N L_{\lambda_j}$ is a module over the completion $\hat U(\g^{\oplus N})$.
Let $\hat U_{\bm k}(\g^{\oplus N})$, with $\bm k \coloneqq (k_i)_{i=1}^\ell$, denote the quotient of the algebra $\hat U(\g^{\oplus N})$ by the ideal $J_{\bm k}$ generated by the elements $\cent^{(i)} - k_i$, namely
\begin{equation*}
\hat U_{\bm k}(\g^{\oplus N}) \coloneqq \hat U(\g^{\oplus N}) / J_{\bm k}.
\end{equation*}
The action of  $\hat U(\g^{\oplus N})$ on $\bigotimes_{j=1}^N L_{\lambda_j}$ factors through the quotient $\hat U_{\bm k}(\g^{\oplus N})$. 
In particular, if we define
\begin{equation} \label{k of z}
\cent(z) \coloneqq \sum_{i=1}^N \frac{\cent^{(i)}}{z - z_i}
\end{equation}
then the image of $\cent(z)$ in $\hat U_{\bm k}(\g^{\oplus N})$ is the twist function $\varphi(z)$ as in \eqref{twist function}, cf. \eqref{hw lambda i}.

We have the usual ascending filtration $\CC 1 =\F_0 \subset \F_1 \subset \F_2\subset \dots$ of the universal enveloping algebra $U(\g^{\oplus N})$ of the Lie algebra $\g^{\oplus N}$. Every $x\in U(\g^{\oplus N})$ belongs to some filtered subspace; the \emph{degree} of $x$ is by definition the smallest $k\in \ZZ_{\geq 0}$ such that $x\in \F_k$. 
Let us say that an element $x\in \hat U(\g^{\oplus N})$ has \emph{(finite) degree} $k\in \ZZ_{\geq 0}$ if, when $x$ is written as a sum as in \eqref{cex}, the degrees of the $x_m$ are bounded above and $k$ is their maximum.

\subsection{Conjectures}\label{sec: conjectures}

\begin{conjecture} \label{conj: higher Ham}
There exist nonzero $\hat U(\g'^{\oplus N})$-valued meromorphic functions $\mathcal S_i(z)$, $i \in E$, on $\CP$ with the following properties:
\begin{enumerate}[(i)]

\item For each $i\in E$, $\mathcal S_i(z)$ has degree $i+1$.
  
\item For any $i, j \in E$ we have
\begin{align*}
[\mathcal S_i(z), \mathcal S_j(w)] &= \left(h^\vee \partial_z  - i \cent(z)\right) \mathcal A_{ij}(z, w) + \left(h^\vee\partial_w  - j \cent(w) \right) \mathcal B_{ij}(z, w),
\end{align*}
for some $\hat U(\g'^{\oplus N})$-valued meromorphic functions $\mathcal A_{ij}(z,w), \mathcal B_{ij}(z,w)$ on $\CP \times \CP$. 

\item For each $i \in E$ and each $j = 1, \ldots, N$ we have
\begin{equation*}
\left[\mathcal H_j, \mathcal S_i(z)\right] = \left(h^\vee \partial_z - i \cent(z) \right)\mathcal D^j_i(z),
\end{equation*}
for some $\hat U(\g'^{\oplus N})$-valued meromorphic function $\mathcal D^j_i(z)$ on $\CP$.

\item For each $i \in E$ and any $x \in \g$ we have, writing $\Delta x\coloneqq\sum_{j=1}^N x^{(j)}$,
\begin{equation*}
\left[\Delta x, \mathcal S_i(z)\right] = \left(h^\vee \partial_z - i \cent(z) \right)\mathcal C^x_i(z),
\end{equation*}
for some $\hat U(\g'^{\oplus N})$-valued meromorphic function $\mathcal C^x_i(z)$ on $\CP$. \qed
\end{enumerate}
\end{conjecture}

Suppose such functions $\mathcal S_i(z)$ do exist. 
For any contour $\gamma$ as in Corollary \ref{cor: opint} and any $i \in E$, denote by $\hat Q^\gamma_i$ the image of
\begin{equation} \label{integrated operators} \int_\gamma \P(z)^{-i / h^\vee} \mathcal S_i(z) dz
\end{equation}
in the quotient $\hat U_{\bm k}(\g^{\oplus N})$. Then these $\hat Q^\gamma_i$ are commuting Hamiltonians, as follows.

\begin{corollary}
Given Conjecture \ref{conj: higher Ham}, one has 
\be [\hat Q^\gamma_i, \hat Q^{\eta}_j] = 0\nn\ee 
for any $i, j \in E$ and any pair of contours $\gamma, \eta$. 
Moreover, each $\hat Q^\gamma_i$ 
commutes with the diagonal action of $\g$ and with the quadratic Hamiltonians $\mathcal H_j$, $j =1, \ldots, N$. \qed
\end{corollary}

\begin{conjecture} \label{conj: e-val op}
Let $\psi \in \big( \! \bigotimes_{j=1}^N L_{\lambda_j} \big)_{\lambda_\infty}$ be the Schechtman-Varchenko vector associated with the Miura $\lg$-oper $\nabla \in \MOp_{\lg}(\CP)$ in \eqref{u Miura op def}. For every $j \in E$, let $v_j(z)$ be the coefficient of $p_j$ in any quasi-canonical form of the $\lg$-oper $[\nabla]$.

If the Bethe roots $w_j$, $j = 1, \ldots, m$, satisfy the Bethe equations \eqref{Bethe equations} then
\begin{equation*}
\hat Q^\gamma_i \psi = \int_\gamma \P(z)^{-i / h^\vee} v_i(z) dz\; \psi
\end{equation*}
for every $i \in E$ and any choice of contour $\gamma$ as in Corollary \ref{cor: opint}. \qed
\end{conjecture}

The remainder of this section is devoted to showing that these conjectures are consistent with what is known about the quadratic Hamiltonians for affine Gaudin models. 

In a separate paper \cite{LVY2} we explicitly construct $\mathcal S_2(z)$ in the case $\g' = \widehat{\mathfrak{sl}}_M$, $M\geq 3$ and show that the statements of Conjecture \ref{conj: higher Ham} hold for $i=1,2$. In these cases we also verify Conjecture \ref{conj: e-val op} for $m=0, 1$. In particular, we will show in \cite{LVY2} that in the case of two sites ($N=2$) the Hamiltonians $\hat{Q}^\gamma_1$ and $\hat{Q}^\gamma_2$ are proportional to the quadratic and cubic integrals of motion of quantum Boussinesq theory. This generalises the observation made in \cite{FFsolitons} that the local quadratic Hamiltonian of the two site affine Gaudin model associated with $\widehat{\mathfrak{sl}}_2$ is proportional to the Hamiltonian $L_0$ of the coset Virasoro algebra. Therefore, in the particular case $\g' = \widehat{\mathfrak{sl}}_2$ with $N=2$, our Conjecture \ref{conj: e-val op} complements the conjecture made in \cite[\S 6.4]{FFsolitons} about the joint eigenvectors of the local affine Gaudin Hamiltonians.

\subsection{Quadratic affine Gaudin Hamiltonians} \label{sec: quad Ham}
Recall $\g \cong \hat{\mathcal L}\dot \g \rtimes \CC t \del_t$. Fix a basis $I^a$, for $a = 1, \ldots, \dim\dot\g$, of $\dot\g$. Recall the non-degenerate bilinear form $(\cdot|\cdot): \g \times \g \to \CC$ on $\g$ from \S\ref{sec: Cartan data}. It restricts to a non-degenerate bilinear form $\dot \g \times \dot\g \to \CC$ on $\dot\g$. Let $I_a$ be the dual basis of $\dot\g$ with respect to this restriction. 
A basis of $\g$ is then given by $I^a_n \coloneqq I^a \otimes t^n$, for $a = 1, \ldots, \dim \dot\g$ and $n \in \ZZ$, together with $\cent$ and $\cocent$. The corresponding dual basis of $\g$, with respect to $(\cdot|\cdot)$, is given by $I_{a, -n} \coloneqq I_a  \otimes t^{-n}$, for $a = 1, \ldots, \dim \dot\g$ and $n \in \ZZ$, together with $\cocent$ and $\cent$.
In terms of these bases, the quadratic Gaudin Hamiltonians \eqref{quad Ham intro} in this untwisted affine case then take the form
\begin{equation} \label{Gaudin Ham def}
\mathcal H_i = \sum_{\substack{j=1\\ j \neq i}}^N \frac{\cent^{(i)} \cocent^{(j)} + \cocent^{(i)} \cent^{(j)} + \sum_{n \in \ZZ} I^{(i)}_{a, -n} I^{a(j)}_n}{z_i - z_j}\in \hat U(\g^{\oplus N}), \qquad i = 1, \ldots, N.
\end{equation}
(Here and below we employ summation convention: $I_a I^a := \sum_{a=1}^{\dim \dot \g} I_a I^a$.)

For every $i = 1, \ldots, N$, the completed enveloping algebra $\hat U(\g^{\oplus N})$ also contains the $i^{\rm th}$ copy of the \emph{quadratic Casimir} of $\g$, namely the element of $\hat U(\g^{\oplus N})$ defined as
\begin{equation} \label{Omega def}
\C^{(i)} \coloneqq (\cent^{(i)} + h^\vee) \cocent^{(i)} + \ha I^{(i)}_{a, 0} I^{a(i)}_0 + \sum_{n > 0} I^{(i)}_{a, -n} I^{a(i)}_n.
\end{equation}

\begin{proposition} \label{prop: Casimir hom grad}
Each $\C^{(i)}$, $i=1, \ldots, N$ is central in $\hat U(\g^{\oplus N})$.

Its action on the tensor product of irreducible highest weight $\g$-module $\bigotimes_{j=1}^N L_{\lambda_j}$, for any $\lambda_1, \ldots, \lambda_N \in \h^\ast$, is given by multiplication by $\ha (\lambda_i | \lambda_i + 2 \rho)$.
\begin{proof}
It suffices to consider the case $N=1$, for which we can drop all superscripts labelling the copy of $\g$ in the direct sum $\g^{\oplus N}$. For the first statement we will simply show that $\C$ as defined in \eqref{Omega def} coincides with the quadratic Casimir for a general Kac-Moody algebra \cite[\S 2.5]{KacBook} in the affine case. The second part of the statement will then follow from \cite[Corollary 2.6]{KacBook}.

We have the Cartan subalgebra $\dot\h= \h \cap \dot\g$ of $\dot\g$ and the root space decomposition
\begin{equation*}
\dot\g = \dot\h \oplus \bigoplus_{\alpha \in \dot\Delta} \dot\g_\alpha,
\end{equation*}
where $\dot\Delta$ denotes the root system of $\dot\g$. Let $\dot\Delta_+ \subset \dot\Delta$ denote the subset of positive roots. The corresponding root space decomposition of the untwisted affine Kac-Moody algebra $\g$ reads
\begin{equation*}
\g = \bigoplus_{\wt\alpha \in \Delta} \g_{\wt\alpha},
\end{equation*}
where $\Delta \coloneqq \{ \alpha + n \delta \, |\, \alpha \in \dot\Delta, n \in \ZZ \}$ is the root system of $\g$. We denote the subset of positive roots by $\Delta_+ \coloneqq \{ \alpha + n \delta \, |\, \alpha \in \dot\Delta, n > 0 \} \cup \dot\Delta_+$. Explicitly, $\g_{\pm \alpha+n\delta} = \dot\g_{\pm \alpha} \otimes t^n$ for every $\alpha \in \dot\Delta_+$ and $n \in \ZZ$, while $\g_{n\delta} = \dot\h \otimes t^n$ for all $n \in \ZZ \setminus \{ 0 \}$ and $\g_0 = \h$.
We fix a basis $e^s_{\wt\alpha}$, $s = 1, \ldots, \dim \g_{\wt\alpha}$ of the root space $\g_{\wt\alpha}$ for each $\wt\alpha \in \Delta_+$ and denote by $e^s_{-\wt\alpha}$, $s = 1, \ldots, \dim \g_{\wt\alpha}$ its dual basis in $\g_{-\wt\alpha}$. Also fix a basis $\{ u_i \}_{i=1}^{\dim \h} = \{ h_i \}_{i=1}^\ell \cup \{ \cent, \cocent \}$ of $\h$, where $\{ h_i \}_{i=1}^\ell$ is a basis of $\dot\h$, and let $\{ u^i \}_{i=1}^{\dim \h} = \{ h^i \}_{i=1}^\ell \cup \{ \cocent, \cent \}$ be its dual basis, where $\{ h^i \}_{i=1}^\ell$ is the basis of $\dot\h$ dual to $\{ h_i \}_{i=1}^\ell$.

We may now rewrite the expression \eqref{Omega def} for the quadratic Casimir using the above dual bases of $\g$. For the second term on the right hand side of \eqref{Omega def} we have
\begin{equation*}
I_{a, 0} I^a_0 = \sum_{i=1}^\ell h_i h^i + \sum_{\alpha \in \dot\Delta_+} \big( e_\alpha e_{-\alpha} + e_{-\alpha} e_\alpha \big) = 2 \nu^{-1}(\dot\rho) + \sum_{i=1}^\ell h_i h^i + 2 \sum_{\alpha \in \dot\Delta_+} e_{-\alpha} e_\alpha,
\end{equation*}
where we dropped the superscript `$s$' on the basis elements $e^s_{\pm \alpha}$ for $\alpha \in \dot\Delta_+$ since in this case $\dim \g_{\pm \alpha} = 1$. In the second equality we used the relation $[e_\alpha, e_{-\alpha}] = \nu^{-1}(\alpha)$ and set $\dot\rho \coloneqq \ha \sum_{\alpha \in \dot\Delta_+} \alpha$.
On the other hand, the infinite sum over $n>0$ in \eqref{Omega def} can be written as
\begin{equation*}
\sum_{n > 0} I_{a, -n} I^a_n = \sum_{\wt\alpha \in \Delta_+ \setminus \dot\Delta_+} \sum_{s = 1}^{\dim \g_{\wt\alpha}} e^s_{-\wt\alpha} e^s_{\wt\alpha}.
\end{equation*}

Recall the set of fundamental coweights $\{ \check\Lambda_i \}_{i=0}^\ell$ of $\g$ defined by \eqref{def: co Lambda}. The set of fundamental coweights $\{ \check\omega_i \}_{i=1}^\ell$ of $\dot\g$ can be identified with $\check\omega_i = \check\Lambda_i - a_i \check\Lambda_0$ for each $i = 1, \ldots,\ell$. If we set $\epsilon_i \coloneqq a_i \check a_i^{-1}$ for $i = 0, \ldots, \ell$, then
\begin{equation*}
\nu^{-1}(\dot\rho) = \sum_{i=1}^\ell \epsilon_i^{-1} \check\omega_i = \sum_{i=0}^\ell \epsilon_i^{-1} (\check\Lambda_i - a_i \check\Lambda_0) = \nu^{-1}(\rho) - h^\vee \cocent
\end{equation*}
where in the second step we used the assumption that $a_0 = 1$, cf. Remark \ref{rem: not A2l}.
Therefore, combining all the above we can rewrite the quadratic Casimir \eqref{Omega def} as
\begin{equation} \label{Omega Kac}
\C = \nu^{-1}(\rho) + \ha \sum_{i=1}^{\dim \h} u_i u^i + \sum_{\wt\alpha \in \Delta_+} \sum_{s = 1}^{\dim \g_{\wt\alpha}} e^s_{-\wt\alpha} e^s_{\wt\alpha},
\end{equation}
which coincides with its expression given in \cite[\S 2.5]{KacBook}, as required.
\end{proof}
\end{proposition}

By direct analogy with the finite-dimensional case, cf. \eqref{S1 def finite}, it is natural to introduce the $\hat U(\g^{\oplus N})$-valued meromorphic function
\begin{equation} \label{S1 def affine}
S_1(z) \coloneqq \sum_{i=1}^N \frac{\C^{(i)}}{(z - z_i)^2} + \sum_{i=1}^N \frac{\H_i}{z - z_i}.
\end{equation}

We then have the following direct generalisation of the finite-dimensional case.

\begin{theorem}
Let $\nabla \in \MOp_{\lg}(\CP)$ be of the form \eqref{u Miura op def}. If the set of Bethe roots $w_j$, $j = 1, \ldots, m$ satisfy the Bethe equations \eqref{Bethe equations}, then the eigenvalue of $S_1(z)$ on the subspace $\big( \! \bigotimes_{j=1}^N L_{\lambda_j} \big)_{\lambda_\infty}$ of weight $\lambda_\infty = \sum_{i=1}^N \lambda_i - \sum_{j=1}^m \alpha_{c(j)} \in \h^\ast$, is given by $h^\vee$ times the coefficient of $p_1$ in any quasi-canonical form of the underlying $\lg$-oper $[\nabla]$.
\begin{proof}
This follows form Theorem \ref{thm: Miura oper p1 coeff} together with Proposition \ref{prop: Casimir hom grad} and \eqref{evalue eq}.
\end{proof}
\end{theorem}

The expression \eqref{S1 def affine} can alternatively be described as follows. 
For any $x \in \g$ we define the $\g^{\oplus N}$-valued meromorphic function, cf. \eqref{k of z},
\begin{equation*}
x(z) \coloneqq \sum_{i=1}^N \frac{x^{(i)}}{z - z_i}.
\end{equation*}
We then introduce the \emph{formal Lax matrix} of the Gaudin model associated with $\g$ as the element
\begin{equation} \label{formal Lax}
\mathcal L(z) \coloneqq \cent \otimes \cocent(z) + \cocent \otimes \cent(z) + \sum_{n \in \ZZ} I_{a, -n} \otimes I^a_n(z)
\end{equation}
of the completed tensor product $\g \,\hat\otimes\, \g^{\oplus N}$. Then the generating function \eqref{S1 def affine} for the quadratic affine Gaudin Hamiltonians can be rewritten as
\begin{equation} \label{S1 alernative}
S_1(z) = \ha \nord{\!\big( \mathcal L(z) \big| \mathcal L(z) \big)\!} - \; h^\vee \cocent'(z),
\end{equation}
where $\nord{\cdot}$ denotes normal ordering by mode numbers, \emph{i.e.} $\nord{I^a_m I^b_n}$ is $I_n^b I^a_m$ if $m\geq 0$ and $I^a_mI^b_n$ otherwise.

\begin{remark} This expression can be regarded as a quantisation of the generating function for the quadratic Hamiltonians of a \emph{classical} affine Gaudin model \cite{V17}. Indeed, the latter is a meromorphic function valued in a completion of the symmetric algebra $S(\g^{\oplus N})$, given explicitly by $\ha (\mathcal L(z) | \mathcal L(z))$. In fact, the above expression for $S_1(z)$ can be heuristically obtained from $\ha (\mathcal L(z) | \mathcal L(z))$ by using the commutation relations of $\g$ to rewrite each term so that all raising operators, \emph{i.e.} $I^a_n$ with $n > 0$, appear on the right. This procedure results in a meaningless infinite sum, but the term $- h^\vee \cocent'(z)$ can be thought of as its regularisation; see \cite[\S 2.11]{KacBook} for a similar motivation of the linear term $\nu^{-1}(\rho)$ in the expression \eqref{Omega Kac} for the quadratic Casimir of $\g$ in the proof of Proposition \ref{prop: Casimir hom grad}.
\end{remark}

Now we explain how the generating function $S_1(z)$ of the  quadratic Hamiltonians fits into our conjecture on the form of the higher affine Gaudin Hamiltonians. We first reinterpret it in light of Corollary \ref{cor: v1}. Define the \emph{local Lax matrix} as the part of the formal Lax matrix \eqref{formal Lax} involving only the loop generators of $\g$, namely
\begin{equation} \label{local Lax matrix}
L(z) \coloneqq \sum_{n \in \ZZ} I_{a, -n} \otimes I^a_n(z).
\end{equation}
The expression \eqref{S1 alernative} for $S_1(z)$ can now be rewritten as follows
\begin{equation} \label{S1 with twisted der d}
S_1(z) = \mathcal S_1(z) - \, \big(h^\vee \partial_z - \cent(z) \big) \cocent(z),
\end{equation}
where we defined
\begin{equation*}
\mathcal S_1(z) \coloneqq \ha \nord{\!\big( L(z) \big| L(z) \big)\!}.
\end{equation*}

Let $\psi \in \big( \! \bigotimes_{j=1}^N L_{\lambda_j} \big)_{\lambda_\infty}$ denote the Schechtman-Varchenko vector corresponding to the Miura $\lg$-oper $\nabla \in \MOp_{\lg}(\CP)$ in \eqref{u Miura op def}. Recall the expression \eqref{hw lambda i} for the weights $\lambda_i \in \h^\ast$, $i = 1, \ldots, N$. On $\bigotimes_{j=1}^N L_{\lambda_j}$, $\cent(z)$ acts as $\varphi(z)$ as we noted above, and the action of $\cocent(z)$ on $\psi$ is given by
\begin{align*}
\cocent(z) \psi &= \sum_{i=1}^N \frac{\Delta_i - m_0}{z - z_i} \psi \eqqcolon \Delta(z) \psi,
\end{align*}
where $m_0$ is the number of Bethe roots associated to the affine simple root $\alpha_0$.
In other words, the action of \eqref{S1 with twisted der d} on the Schechtman-Varchenko vector reads
\begin{equation} \label{S1 on psi}
S_1(z) \psi = \mathcal S_1(z) \psi - h^\vee \bigg( \Delta'(z) - \frac{\varphi(z)}{h^\vee} \Delta(z) \bigg) \psi.
\end{equation}
Observe that the final term is a twisted derivative of degree 1.

Now recall from Corollary \ref{cor: v1} that if instead of working with $\lg$-opers we were to consider $\lg/\CC \delta$-opers, then the coefficients of all the $p_i$'s, $i \in E$ in a quasi-canonical form would be on an equal footing since the coefficient $v_1(z)$ of $p_1$ would also only be defined up to a twisted derivative
\begin{equation*}
v_1 \longmapsto v_1 - f'_1 + \frac{\varphi}{h^\vee} f_1,
\end{equation*}
with $f_1 \in \M$. In particular, only its integral
\begin{equation*}
\int_\gamma \P(z)^{- 1 /h^\vee} v_1(z) dz
\end{equation*}
over a cycle $\gamma$ as in Corollary \ref{cor: opint} would provide a well-defined function on the space of $\lg/\CC \delta$-opers.

In exactly the same way as we conjecture the spectrum of an affine Gaudin model associated with $\g$ to be described by $\lg$-opers, the space of $\lg/\CC \delta$-opers should describe the spectrum of an affine Gaudin model associated with the derived algebra $\g'$. Indeed, since the weights appearing as the residues in a Miura $\lg/\CC \delta$-oper are classes in $\lh / \CC\delta = \h^\ast/\CC \delta$, \emph{i.e.} weights in $\h^\ast$ defined up to an arbitrary multiple of $\delta$, we should not include the generator $\cocent$ on the Gaudin model side since it pairs non-trivially with the weight $\delta$. And one way to disregard the generator $\cocent$ from the expression \eqref{S1 with twisted der d} for $S_1(z)$ is to consider its integral 
\begin{equation*}
\int_\gamma \P(z)^{- 1/ h^\vee} S_1(z) dz
\end{equation*}
over a contour $\gamma$ as in Corollary \ref{cor: opint}. Indeed, it follows from \eqref{S1 on psi} that the action of this operator on the Schechtman-Varchenko vector $\psi$ coincides with the action of the operator
\begin{equation*}
\int_\gamma \P(z)^{- 1/ h^\vee} \mathcal S_1(z) dz.
\end{equation*}

The following Lemma shows that Conjecture \ref{conj: higher Ham} holds at least for $i=1$. 
\begin{lemma}
For any distinct $z, w \in X$ we have
\begin{equation*}
\big[ \mathcal S_1(z), \mathcal S_1(w) \big] = \big( h^\vee \partial_z - \cent(z) \big) \mathcal A(z, w) - \big( h^\vee\partial_w - \cent(w) \big) \mathcal A(z, w),
\end{equation*}
where $\mathcal A(z, w)$ is the $\hat U(\g'^{\oplus N})$-valued meromorphic function on $\CP \times \CP$ given by
\begin{equation*}
\mathcal A(z, w) \coloneqq \frac{1}{z - w} \sum_{\substack{i,j=1\\ i \neq j}}^N \frac{\sum_{n \in \ZZ} n I^{(i)}_{a, -n} I^{a (j)}_n}{(z - z_i)(w - z_j)}.
\end{equation*}
Also, for each $j=1,\dots,N$, 
\begin{equation*}
\big[ \mc H_j , \mc S_1(w) \big] = \big( \hc \del_w - \cent(w)\big) \sum_{\substack{i=1\\ i \neq j}}^N \frac{\sum_{n \in \ZZ} n I^{(j)}_{a, -n} I^{a (i)}_n}{(w - z_i)(w-z_j)}.
\end{equation*}
Moreover, for any $x \in \g$ we have
$[ \Delta x, \mathcal S_1(z) ] = \big( h^\vee \partial_z - \cent(z) \big) [x, \cocent](z)$.
\begin{proof}
Since the quadratic Gaudin Hamiltonians $\H_i$, $i = 1, \ldots, N$ mutually commute and the $\C^{(i)}$, $i = 1, \ldots, N$ are central in $\hat U(\g^{\oplus N})$ by Proposition \ref{prop: Casimir hom grad}, it follows that $[S_1(z), S_1(w)] = 0$. 
Noting that
$[ \cocent(z), \mathcal S_1(w) ] = \mathcal A(z, w)$
one has \be \big[\hc \cocent'(z) - \cent(z) \cocent(z), \mc S_1(w)\big] = \big( \hc \del_z - \cent(z)\big) \mc A(z,w) \nn\ee
Using the relation \eqref{S1 with twisted der d} we get the first result. It also follows that 
$[S_1(z), \mc S_1(w)] = - (\hc \del_w - \cent(w)) \mc A(z,w)$, from which, taking the residue at $z=z_j$, we obtain the commutators with $\mc H_j$.
The last part follows similarly from the relation $[\Delta x, S_1(z)] = 0$ which is a consequence of the fact that both $\C^{(i)}$ and $\mathcal H_i$, for each $i =1, \ldots, N$, commute with the diagonal action of $\g$.
\end{proof}
\end{lemma}

\section{Coordinate invariance and meromorphic $\lg$-opers on curves}\label{sec: coord}

Throughout \S\ref{sec: opers}\---\ref{sec: Miura opers} we fixed a global coordinate $z$ on $\CC \subset \CP$ and studied meromorphic $\lg$-opers in that coordinate.
Let us now consider meromorphic $\lg$-opers in local charts, and discuss their behaviour under changes in coordinate. In this section only, we shall work over an arbitrary Riemann surface $\Sigma$.

When $\lg$ is of finite type, an $\lg$-oper on $\Sigma$ is a triple $(\mathcal F, \mathcal B, \nabla)$ where $\mathcal F$ is a principal $\lG$-bundle, $\mathcal B$ is an $\lB$-reduction and $\nabla$ is a connection on $\mathcal F$ with certain properties; see \emph{e.g.} \cite{BD91, Fre07}. Concretely, such a triple can be constructed by gluing together trivial $\lG$-bundles over coordinate patches, each equipped with a connection given by an $\lg$-oper in canonical form, using the $\lB$-valued transition functions relating canonical forms in different coordinates (see equation \eqref{transition can to can} below) \cite{Fopersontheprojectiveline, Fre07}.

The abstract definition of an $\lg$-oper as a triple can be generalised to the case when $\lg$ is of affine type \cite{Fopersontheprojectiveline}. However, since the quasi-canonical form is not unique in this case by Theorem \ref{thm: quasi-canonical form}, and there is no naturally preferred quasi-canonical form, it is less clear how to construct such a triple explicitly. We therefore proceed differently: we first define the space of $\lg$-opers over any coordinate patch as in \S\ref{sec: opers} and then glue these together to form a sheaf, the sheaf of $\lg$-opers over $\Sigma$.

\subsection{The sheaf of $\lg$-opers $\Op_{\lg}$}

For any open subset $U \subset \Sigma$ we let $\K(U)$ denote the field of meromorphic functions on $U$.
We denote by $\K$ the sheaf $U \mapsto \K(U)$ of meromorphic functions on $\Sigma$. When $\Sigma = \CP$, the field $\M$ of meromorphic functions on $\CP$, introduced in \S\ref{sec: inverse limits}, is the field of global sections of $\K$.
For any open subset $U \subset \Sigma$, we define the Lie algebra $\lhn_+(\K(U))$ and the group $\lhN_+(\K(U))$ as in \S\ref{sec: def oper}. We also set $\lhb_+(\K(U)) \coloneqq \lh(\K(U)) \oplus \lhn_+(\K(U))$.

To begin with, let us suppose that $U\subset \Sigma$ is an open subset equipped with a holomorphic coordinate $t:U \to \CC$. Define $\op_{\lg}(U)$ to be the affine space of connections of the form 
\begin{equation} \label{nf}
\nabla\coloneqq d + p_{-1}dt  + b  dt, \quad b\in \lhb_+(\K(U)).
\end{equation}
As in \S\ref{sec: def oper}, it admits an action of the group $\lhN_+(\K(U))$ by gauge transformations, and we define the space of meromorphic $\lg$-opers on $U$ to be the quotient
\begin{equation} \label{def: Uop}
\Op_{\lg}(U) \coloneqq \op_{\lg}(U) \big/ \lhN_+(\K(U)).
\end{equation}
The proof of the following is as for Theorem \ref{thm: quasi-canonical form}.
\begin{theorem}\label{thm: qcU}
Let $U \subset \Sigma$ be open with a holomorphic coordinate $t : U \to \CC$. Every class $[\nabla] \in \Op_{\lg}(U)$ has a representative of the form 
\begin{equation*}
\nabla = d + \Bigg( p_{-1} - \frac{\varphi}{h^\vee} \rho + \sum_{i\in E} v_i p_i \Bigg) dt
\end{equation*}
where $\varphi \in \K(U)$ and $v_i \in \K(U)$ for each $i\in E$. We call such a form \emph{quasi-canonical}. It is unique up to residual gauge transformations as in Theorem \ref{thm: quasi-canonical form}.
\qed
\end{theorem}
We would like to understand the behaviour of such quasi-canonical representatives under changes in local coordinate. We will come back to this in \S\ref{sec:  qc form curve} below. The first problem is to formulate the definition of $\Op_{\lg}(U)$ itself in a coordinate-independent fashion. 
Indeed, suppose $s:U\to \CC$ is another holomorphic coordinate on the same open set $U \subset \Sigma$, with $t= \mu(s)$. The connection in \eqref{nf} becomes
\be \nabla = d + p_{-1} \mu'(s) ds + b \mu'(s) ds.\label{nfp}\ee
This is no longer of the form \eqref{nf} in the new coordinate $s$, and in this sense the definition of $\op_{\lg}(U)$ is coordinate dependent. However, it is possible to re-express $\Op_{\lg}(U)$ as the quotient of a suitably larger affine space of connections $\widetilde\op_{\lg}(U)\supset \op_{\lg}(U)$, which itself is coordinate \emph{independent}, by some larger group of gauge transformations $\lhB_+(\K(U)) \supset \lhN_+(\K(U))$ to be defined below.

Indeed, let $\widetilde\op_{\lg}(U)$ be the affine space consisting of all connections of the form
\begin{equation} \label{tnf}
\widetilde \nabla = d + \left(\sum_{i=0}^\ell \psi_i\check f_i + b\right) dt
\end{equation}
with $\psi_i$ a nonzero element of $\K(U)$ for each $i\in I$, and $b\in \lhb_+(\K(U))$.
Observe that the definition of $\widetilde{\op}_{\lg}(U)$ is independent of the choice of coordinate. (The derivative $\mu'$ in \eqref{nfp} belongs to $\K(U)$ since it is holomorphic and non-vanishing on $U$.) 

Now we define the group $\lhB_+(\K(U))$ and its action on $\widetilde{\op}_{\lg}(U)$ by gauge transformations.
First, let $P\coloneqq \bigoplus_{i=0}^\ell \ZZ \Lambda_i\subset \lh$ denote the lattice of integral coweights of $\lg$, where  $\{\Lambda_i\}_{i=0}^\ell$ are the fundamental coweights of $\lg$ defined in \eqref{def: Lambda}. 
Let $\lH(\K(U))$ denote the abelian group generated by elements of the form $\phi^\lambda$, $\phi \in \K(U) \setminus \{ 0 \}$, $\lambda\in P$, subject to the relations $\phi^{\lambda} \psi^{\lambda} = (\phi\psi)^\lambda$, $\phi^{\lambda+\mu} = \phi^\lambda \phi^\mu$ for all $\phi,\psi\in \K(U)\setminus \{ 0 \}$ and $\lambda,\mu\in P$. (Note that this definition makes sense for \emph{any} open subset $U \subset \Sigma$, but to describe the action of the group $\lH(\K(U))$ on $\widetilde{\op}_{\lg}(U)$ we shall only need the case when $U$ is a coordinate chart.)

For each $\check\alpha\in \check Q$ in the root lattice of $\lg$, we have the (adjoint) action of the group $\lH(\K(U))$ on the space $\lg_{\check\alpha}(\K(U))$ of meromorphic functions on $U$ valued in the root space $\lg_{\check\alpha}$, given by
\begin{equation} \label{aa}
\phi^\lambda n \phi^{-\lambda} \coloneqq  \phi^{\langle \lambda, \check\alpha\rangle} n,
\end{equation}
for all $n\in \lg_{\check\alpha}(\K(U))$, $\phi \in \K(U) \setminus \{ 0 \}$ and $\lambda \in P$.
Here $\langle \lambda, \check\alpha\rangle\in \ZZ$, by definition of $P$, so that $\phi^{\langle \lambda, \check \alpha \rangle} \in \K(U)$. Hence we get an action on the Lie algebra $\lhn_+(\K(U))$. Then $\lH(\K(U))$ acts also on the group $\lhN_+(\K(U))$, with $\phi^\lambda\exp(n)\phi^{-\lambda} \coloneqq \exp(\phi^\lambda n\phi^{-\lambda})$.
We may now define the desired group to be the semi-direct product
\begin{equation*}
\lhB_+(\K(U)) \coloneqq \lhN_+(\K(U)) \rtimes \lH(\K(U)).
\end{equation*}
That is, $\lhB_+(\K(U))$ is the group generated by elements of the form $\exp(n) \phi^\lambda$ with $n\in \lhn_+(\K(U))$, $\phi\in \K(U) \setminus \{ 0 \}$ and $\lambda\in P$, with the group product given by
\begin{equation*}
(\exp(n) \phi^{\lambda})( \exp(m) \psi^{\mu}) \coloneqq \Big( \! \exp(n) \exp\big( \phi^{\lambda} m \phi^{-\lambda}\big) \Big) \big(\phi^\lambda \psi^\mu \big),
\end{equation*} 
for any $m, n\in \lhn_+(\K(U))$, $\phi, \psi \in \K(U) \setminus \{ 0 \}$ and $\lambda, \mu \in P$.

Finally, we define the gauge action of $\lH(\K(U))$ on connections in $\widetilde{\op}_{\lg}(U)$, of the form \eqref{tnf}, by
\begin{equation} \label{def: Haction}
\phi^{\lambda} \left(d + \sum_{i=0}^\ell \psi_i\check f_idt  + bdt\right) \phi^{-\lambda} \coloneqq d + \sum_{i=0}^\ell \phi^{-\langle\lambda,\check\alpha_i\rangle} \psi_i\check f_i dt - \lambda \phi^{-1} d\phi + \phi^{\lambda} b\phi^{-\lambda} dt,
\end{equation}
where, again, $ \phi^{\lambda} b\phi^{-\lambda}$ is defined by extending \eqref{aa} to $\lhb_+(\K(U))$ by linearity.
\begin{lemma} Equation \eqref{def: Haction} defines an action of the group $\lH(\K(U))$ on the space of connections $\widetilde\op_{\lg}(U)$. Combining it with the action of $\lhN_+(\K(U))$ defined as in \S\ref{sec: def oper}, we obtain a well-defined action of $\lhB_+(\K(U))$ on $\widetilde\op_{\lg}(U)$. \qed
\end{lemma}
\begin{lemma}\label{lem: Uop}
The space of meromorphic $\lg$-opers on a coordinate chart $U$ is equal to the quotient of $\widetilde \op_{\lg}(U)$ by this gauge action of $\lhB_+(\K(U))$: 
\begin{equation*}
\Op_{\lg}(U) = \widetilde\op_{\lg}(U) \big/ \lhB_+(\K(U)).
\end{equation*}
\end{lemma}
\begin{proof}
Let $\widetilde \nabla \in \widetilde \op_{\lg}(U)$ be as in \eqref{tnf}. On inspecting \eqref{def: Haction}, we see that its $\lH(\K(U))$-orbit has a unique representative in $\op_{\lg}(U)$, namely
$( \prod_{i=0}^\ell \psi_i^{\Lambda_i} ) \widetilde \nabla ( \prod_{i=0}^\ell \psi_i^{-\Lambda_i} )$.
\end{proof}

\begin{remark}
If we were to replace $P$ by $P\oplus \CC\delta$ in the definition of ${}^L\tilde H(\K(U))$ then the quotient $\widetilde\op_{\lg}(U) \big/ \lhB_+(\K(U))$ would be smaller than $\Op_{\lg}(U)$; in fact it would be isomorphic to $\Op_{\lg/\CC\delta}(U)$.
\end{remark}

Now suppose $U \subset \Sigma$ is \emph{any} open subset, not necessarily a coordinate chart. Let $\{ U_\alpha \}_{\alpha \in A}$ be an open cover of $U$ by coordinate charts, \emph{i.e.} open subsets $U_\alpha \subset \Sigma$ for each $\alpha$ in some indexing set $A$ with holomorphic coordinates $t_\alpha : U_\alpha \to \CC$ such that $U = \cup_{\alpha \in A} U_\alpha$. We define $\Op_{\lg}(U)$ to be the set of collections $\{ [\nabla_\alpha] \in \Op_{\lg}(U_\alpha) \}_{\alpha \in A}$ with the following property:  for any pair of overlapping charts $U_\alpha \cap U_\beta \neq \emptyset$ and any choice of representatives $\nabla_\alpha \in \widetilde{\op}_{\lg}(U_\alpha)$ and $\nabla_\beta \in \widetilde{\op}_{\lg}(U_\beta)$, their restrictions $\nabla_\alpha|_{U_\alpha \cap U_\beta}$ and $\nabla_\beta|_{U_\alpha \cap U_\beta}$ define the same class in $\Op_{\lg}(U_\alpha \cap U_\beta)$. That is, the pair of representatives $\nabla_\alpha \in \widetilde{\op}_{\lg}(U_\alpha)$ and $\nabla_\beta \in \widetilde{\op}_{\lg}(U_\beta)$ considered on the overlap $U_\alpha \cap U_\beta$ are related by a gauge transformation in $\lhB_+(\K(U_\alpha \cap U_\beta))$. Since $\lhB_+(\K(U_\alpha))$ and $\lhB_+(\K(U_\beta))$ are naturally subgroups of $\lhB_+(\K(U_\alpha \cap U_\beta))$, the above property does not depend on the choice of representatives of the $\lg$-opers $[\nabla_\alpha] \in \Op_{\lg}(U_\alpha)$ for each $\alpha \in A$. This defines the \emph{sheaf of $\lg$-opers} $\Op_{\lg}$.

\subsection{Quasi-canonical form} \label{sec:  qc form curve}

Let $U \subset \Sigma$ be open and $[\nabla] \coloneqq \{ [\nabla_\alpha] \in \Op_{\lg}(U_\alpha) \}_{\alpha \in A}$ be an element of $\Op_{\lg}(U)$. Call $\widetilde{\nabla} \coloneqq \{ \widetilde{\nabla}_\alpha \in \op_{\lg}(U_\alpha) \}_{\alpha \in A}$ a representative of $[\nabla]$ if $[\widetilde{\nabla}_\alpha] = [\nabla_\alpha]$ for each $\alpha \in A$. We shall say that this representative is in \emph{quasi-canonical form} if for each $\alpha \in A$, $\widetilde{\nabla}_\alpha$ is a quasi-canonical form as in Theorem \ref{thm: qcU} with respect to the local coordinate $t_\alpha : U_\alpha \to \CC$, \emph{i.e.} for each $\alpha \in A$ we have
\begin{equation} \label{oper rep Ualpha}
\widetilde{\nabla}_\alpha = d + \Bigg( p_{-1} - \frac{\varphi_\alpha(t_\alpha)}{h^\vee} \rho + \sum_{i\in E} v_{\alpha, i}(t_\alpha) p_i \Bigg) dt_\alpha
\end{equation}
for some $\varphi_\alpha \in \K(U_\alpha)$ and $v_{\alpha, i} \in \K(U_\alpha)$, $i \in E$. In this section we identify which sheaves the collections of functions $\{ \varphi_\alpha \}_{\alpha \in A}$ and $\{ v_{\alpha, i} \}_{\alpha \in A}$, $i \in E$ define sections of.

It suffices to consider an open subset $U \subset \Sigma$ equipped with a pair of holomorphic coordinates $t : U \to \CC$ and $s : U \to \CC$, and to determine the gauge transformation parameter in $\lhB_+(\K(U))$ relating quasi-canonical forms in each coordinate. In the above notation, $U$ corresponds to the overlap $U_\alpha \cap U_\beta$ of the open sets $U_\alpha$ and $U_\beta$ with coordinates $t = t_\alpha : U_\alpha \to \CC$ and $s = t_\beta : U_\beta \to \CC$, respectively. So suppose that we start with a representative of an $\lg$-oper $[\nabla]\in \Op_{\lg}(U)$ which is in quasi-canonical form in the $t$ coordinate, as in Theorem \ref{thm: qcU}: 
\begin{equation*}
\nabla = d + p_{-1} dt - \frac{\varphi(t)}{h^\vee} \rho dt + \sum_{i\in E} v_i(t) p_i dt.
\end{equation*}
In terms of the other coordinate $s$ with $t=\mu(s)$ we have
\begin{equation*}
\nabla = d + p_{-1} \mu'(s) ds - \frac{\varphi(\mu(s))}{h^\vee} \rho \mu'(s)ds + \sum_{i\in E} v_i(\mu(s)) p_i \mu'(s) ds.
\end{equation*}
This can be brought into quasi-canonical form in the $s$ coordinate by performing a gauge transformation by  $\mu'(s)^\rho\in \lH(\K(U))$. Indeed, one finds that 
\begin{subequations}\label{ccp}
\begin{equation*}
\mu'(s)^\rho \,\nabla\, \mu'(s)^{-\rho} =  
d + p_{-1} ds - \frac{\tilde\varphi(s)}{h^\vee} \rho ds  + \sum_{i\in E} \tilde v_i(s) p_i ds
\end{equation*}
making use of the second relation in \eqref{Lie alg a com rel}, and where we defined
\begin{align} \label{phipt}
- \frac{1}{h^\vee} \tilde\varphi(s) &\coloneqq - \frac{1}{h^\vee} \varphi(\mu(s))\mu'(s) - \frac{\mu''(s)}{\mu'(s)},\\
\tilde v_i(s)  &\coloneqq v_i(\mu(s)) \mu'(s)^{i+1}, \qquad i\in E.\label{vpt}
\end{align}
\end{subequations}
Here we will interpret the transformation property \eqref{phipt} as well as \eqref{vpt} in the case $i = 1$. We will come back to \eqref{vpt} for $i \in E_{\geq 2}$ in \S\ref{sec: twisted coh curve} below.

We shall need the following notation.
For any $k \in \ZZ$, let us denote by $\Omega^k \coloneqq \Omega^{\otimes k}$ the $k^{\rm th}$ tensor power\footnote{Recall that over $\CP$, $\Omega^k$ is defined for all $k\in \ZZ/2$. In particular $\Omega^{1/2}$ is the tautological line bundle. For our purposes with $\lg$ affine we shall need only integer powers.}  of the canonical line bundle 
(\emph{i.e.} the cotangent bundle) 
$\Omega$ over $\Sigma$.
We denote by $U \mapsto \Gamma(U,\Omega^k)$ the sheaf of meromorphic sections of $\Omega^k$.
Also let $U \mapsto \Conn(U, \Omega)$ denote the sheaf of meromorphic connections on $\Omega$.

\begin{theorem}\label{thm: ct}
Let $U \subset \Sigma$ be open and $\nabla$ be any quasi-canonical form of an $\lg$-oper $[\nabla] \in \Op_{\lg}(U)$.
The coefficient of $\rho$ is the component of a connection in $\Conn(U, \Omega)$.
\end{theorem}
\begin{proof}
The coefficient of $\rho$ in $\nabla$, or to be more precise the collection of coefficients of $\rho$ for every $\nabla_\alpha \in \op_{\lg}(U_\alpha)$ where $\nabla = \{ \nabla_\alpha \}_{\alpha \in A}$ relative to a cover $\{ U_\alpha \}_{\alpha \in A}$ of $U$, is independent of the choice of quasi-canonical form $\nabla$ by Lemma \ref{lem: Lambda indep}.

In the local trivialization defined by the coordinate $t$, a meromorphic section of $\Omega^{k}$ is given by a meromorphic function $f(t)$, and a connection $\Gamma(U,\Omega^{k}) \to \Gamma(U, \Omega \ox \Omega^{k})$ is a differential operator $f(t)\mapsto f'(t) + A(t) f(t)$, specified by a meromorphic function $A(t)$, the component of the connection in this local trivialization. Here $f'(t) + A(t) f(t)$ must transform as a section of $\Omega \ox \Omega^{k}$, which is to say that 
\be \tilde f'(s) + \tilde A(s) \tilde f(s) = \left( f'(t) + A(t)f(t)\right) \mu'(s)^{k+1}.\nn\ee 
Now in fact 
\begin{align*}
\tilde f'(s) + \tilde A(s) \tilde f(s) 
&= \del_s\big( f(t) \mu'(s)^{k}\big) + \tilde A(s) f(t) \mu'(s)^k\\
&= f'(t) \mu'(s)^{k+1} + k f(t) \mu'(s)^{k-1} \mu''(s)+ \tilde A(s) f(t) \mu'(s)^k
\end{align*}
and we see that indeed $A$ must transform as $\tilde A(s) = A(t) \mu'(s) - k \mu''(s)/\mu'(s)$. On comparing with \eqref{phipt} the 
%first 
result follows.
\end{proof}

\subsection{Interlude: Comparison with finite type opers}\label{fto}
Before proceeding, it is interesting to compare the transformation properties \eqref{ccp} with those of opers of finite type. To that end, we now briefly recall the situation for finite type opers. Suppose, for this subsection only, that $\lg$ is of finite type. Let $U \subset \Sigma$ be an open subset with two coordinates $t : U \to \CC$ and $s : U \to \CC$. Define $\op_{\lg}(U)$ to be the affine space of connections of the form 
\begin{equation*}
\nabla \coloneqq d + \bar p_{-1}dt  + b  dt, \quad b\in \lb_+(\K(U)).
\end{equation*}
with $\bar p_{-1} \coloneqq \sum_{i=1}^\ell \check f_i$, and define $\widetilde\op_{\lg}(U)$ to be the affine space  of connections of the form
\begin{equation*}
\widetilde \nabla \coloneqq d + \left(\sum_{i=1}^\ell \psi_i\check f_i + b\right) dt
\end{equation*}
with $\psi_i$ a nonzero element of $\K(U)$ for each $i\in I\setminus\{0\}$, and with $b\in \lb_+(\K(U))$. (In finite type $\lhb_+=\lb_+$ since $\ln_m=\{0\}$ for all $m \geq h^\vee$.) Then the definition of the space $\Op_{\lg}(U)$ of $\lg$-opers on $U$ in (\ref{def: Uop}) and Lemma \ref{lem: Uop} remains correct as written. Let $\bar\rho \coloneqq \sum_{i=1}^\ell \bar \Lambda_i$ be the sum of the fundamental coweights $\bar\Lambda_i$ of $\lg$.
There is a unique element $\bar p_1\in \ln_+$ such that $\{\bar p_{-1}, 2 \bar \rho, \bar p_1\}$ form an $\mathfrak{sl}_2$-triple:
\begin{equation} \label{sl2 triple finite}
[\bar p_{-1}, \bar p_1] = -2\bar\rho, \qquad [2\bar \rho, \bar p_{\pm 1}] = \pm 2 \bar p_{\pm 1}.
\end{equation}
The analogue of Theorem \ref{thm: qcU} in finite type is the following statement:
Each gauge equivalence class $[\nabla] \in \Op_{\lg}(U)$ contains a unique representative of the form
\begin{equation} \label{fcf}
\nabla = d + \bar p_{-1} dt + \sum_{i\in \bar E} \bar v_i(t) \bar p_i dt,
\end{equation}
where the (multi)set $\bar E$ of exponents is now finite and where, for each $i\in \bar E$ we have $\bar v_i \in \K(U)$ and $\bar p_i\in \ln_+$ are elements such that  
\begin{equation} \label{finite pi properties}
[\bar p_{1}, \bar p_i] = 0, \qquad [\bar \rho, \bar p_i] = i \bar p_i.
\end{equation}
In terms of the new coordinate $s$ with $t=\mu(s)$ we have
\begin{equation*}
\nabla = d + \bar p_{-1} \mu'(s) ds + \sum_{i\in \bar E} \bar v_i(\mu(s)) \bar p_i \mu'(s) ds.
\end{equation*}
Similarly to the affine case above, we may first perform a gauge transformation by $\mu'(s)^{\bar \rho}\in \lH(\K(U))$ to bring the $\bar p_{-1}$ term into the canonical form $\bar p_{-1} ds$, namely
\begin{equation*}
\mu'(s)^{\bar \rho} \,\nabla\, \mu'(s)^{-\bar \rho} =  
d + \bar p_{-1} ds - \bar \rho \,\frac{\mu''(s)}{\mu'(s)} ds  + \sum_{i\in \bar E} \bar v_i(\mu(s)) \bar p_i \mu'(s)^{i+1}  ds.
\end{equation*}
However, in contrast to the affine case, we are not yet done, because it is necessary -- in order to reach the canonical form -- to remove the $\bar \rho$ term by performing a further gauge transformation by $\exp\big(\bar p_1 \frac{\mu''(s)}{2 \mu'(s)}\big)$. One finds that
\begin{equation} \label{transition can to can}
\exp\left(\bar p_1 \frac{\mu''(s)}{2 \mu'(s)}\right) \mu'(s)^{\bar\rho} \,\nabla\, \mu'(s)^{-\bar\rho} \exp\left( \! -\bar p_1 \frac{\mu''(s)}{2 \mu'(s)}\right)\\
= d + \bar p_{-1} ds + \sum_{i\in \bar E} \tilde{\bar v}_i(s) \bar p_i ds
\end{equation}
using both the relations \eqref{finite pi properties} and \eqref{sl2 triple finite}, and where we defined
\begin{subequations} \label{ccpfinite}
\begin{align}
\tilde{\bar v}_1(s) &\coloneqq \bar v_1(\mu(s)) \mu'(s)^2 - \ha (S\mu)(s),\\
\tilde{\bar v}_i(s) &\coloneqq \bar v_i(\mu(s)) \mu'(s)^{i+1}, \qquad i\in \bar E, i>1.
\end{align}
\end{subequations}
Here $S\mu$ is the \emph{Schwarzian derivative} of $\mu$,
\begin{equation*}
S\mu \coloneqq \frac{\mu'''}{\mu'} - \frac{3}{2} \left( \frac{\mu''}{\mu'} \right)^2.
\end{equation*}

Now, what \eqref{ccpfinite} shows is that each of the $\bar v_i$, $i > 1$ transforms as a section of the power $\Omega^{i+1}$ of the canonical bundle, but that $\bar v_1$ transforms as a \emph{projective connection}; see \cite{Fopersontheprojectiveline, Fre07}. Since \eqref{fcf} was the unique canonical form of the $\lg$-oper in this chart, that means there is an isomorphism, in finite type $\lg$, 
\be \Op_{\lg}(U) \simeq \Proj(U) \times \prod_{\substack{i\in \bar E\\ i>1}} \Gamma(U,\Omega^{i+1}),\label{ft op}\ee 
where $\Proj(U)$ denotes the space of projective connections on $U$.

\begin{remark}
Let us emphasise why, in the affine case, there is no analogue of the second gauge transformation by $\exp\big(\bar p_1 \frac{\mu''(s)}{2 \mu'(s)}\big)$ performed above. When $\lg$ is of affine type it is the central element $\delta$ (and not $\rho$) which appears in the bracket $[p_1,p_{-1}] = \delta$. The derivation element $\rho$ is not in the derived subalgebra. Hence, as we saw in Lemma \ref{lem: Lambda indep}, there is no way to remove the term $- \varphi/h^\vee \rho \, dt$ from a quasi-canonical form via $\lhN_+(\K(U))$-valued gauge transformations. Rather, the twist function $\varphi$ forms part of the data defining the underlying $\lg$-oper (and Theorem \ref{thm: ct} gives its properties under coordinate transformations).
\end{remark}

\subsection{Twisted cohomologies} \label{sec: twisted coh curve}
Now we return to the case in which $\lg$ is of affine type. We would like to give a coordinate-independent description of the space of affine opers  analogous to \eqref{ft op} in finite types.

Let $[\nabla] \in \Op_{\lg}(U)$ be a meromorphic $\lg$-oper on an open subset $U \subset \Sigma$ and let $\nabla$ be a representative in quasi-canonical form.
According to Theorem \ref{thm: ct}, the coefficient of $\rho$ in $\nabla$ defines a (trivially flat, since we are working on a curve) connection on $\Omega$ over $U$, which we denote by
\begin{equation*}
\nabla|_\rho : \Gamma(U, \Omega) \longrightarrow \Gamma(U, \Omega \otimes \Omega),
\end{equation*}
If $t : U \to \CC$ is a coordinate on $U$ then it can be written as $f \mapsto \nabla|_\rho f = df - {h^\vee}^{-1} \varphi f dt$ for some $\varphi \in \K(U)$.
We therefore obtain a surjective map
\begin{equation} \label{Op to Conn}
\Op_{\lg}(U) \longtwoheadrightarrow \Conn(U, \Omega), \qquad [\nabla] \longmapsto \nabla|_\rho
\end{equation}
into the space of meromorphic connections on $\Omega$ over $U$. Given any $\overline{\nabla} \in \Conn(U, \Omega)$, we denote its preimage in $\Op_{\lg}(U)$ under the map \eqref{Op to Conn} by $\Op_{\lg}(U)^{\overline{\nabla}}$. This can be seen as a coordinate-independent version, over an open subset of an arbitrary curve $U \subset \Sigma$, of the space $\Op_{\lg}(\CP)^\varphi$ introduced in \S\ref{sec: twist function}.

For each $j \in \ZZ$, $\nabla|_\rho$ induces a connection on the line bundle $\Omega^j$,
\begin{equation*}
\nabla|_\rho : \Gamma(U, \Omega^j) \longrightarrow \Gamma(U, \Omega \otimes \Omega^j).
\end{equation*}
In a local coordinate $t : U \to \CC$ it takes the form $f \mapsto \nabla|_\rho f = df - j {h^\vee}^{-1} \varphi f dt$.
The transformation property \eqref{vpt} suggests that the coefficient of $p_j$ in $\nabla$, for each $j \in E$, defines a section of $\Omega^{j+1}$. This is indeed the case for $j=1$ since the coefficient of $p_1$ in $\nabla$, \emph{i.e.} the collection of coefficients of $p_1$ for every $\nabla_\alpha \in \op_{\lg}(U_\alpha)$, is independent of the choice of quasi-canonical form $\nabla$ by Proposition \ref{prop: can form u1}. However, it is not quite true for $j\in E_{\geq 2}$ since, unlike the coefficient $\varphi_\alpha$ of $\rho$ and $v_{\alpha, 1}$ of $p_1$ in \eqref{oper rep Ualpha}, the coefficient $v_{\alpha, j}$ of $p_j$, $j \in E_{\geq 2}$ in each chart $U_\alpha$ depends on the choice of quasi-canonical representative $\nabla_\alpha$ by Theorem \ref{thm: qcU}. More precisely, as formulated in Theorem \ref{thm: quasi-canonical form}, the coefficient of $p_j$ in $\nabla$ is defined up to the addition of an element $\nabla|_\rho f$.

Recall that a \emph{local system} is a vector bundle equipped with a flat connection. In our case, for each $j \in E$ we have the local system $(\Omega^j, \nabla|_\rho)$ consisting of the line bundle $\Omega^j$ equipped with the connection $\nabla|_\rho$.
Given any local system one has the associated de Rham complex with coefficients in that local system. In our case it is
\begin{equation*}
0 \longrightarrow \Gamma(U, \Omega^j) \xrightarrow{\nabla|_\rho} \Gamma(U, \Omega \otimes \Omega^j) 
\longrightarrow 0.
\end{equation*}
The \emph{de Rham cohomology of $U$ with coefficients in $(\Omega^j, \nabla|_\rho)$}, or simply the \emph{twisted de Rham cohomology}, is then by definition the quotient
\begin{equation*}
H^1(U,\Omega^j, \nabla|_\rho) \coloneqq \Gamma(U,\Omega \otimes \Omega^j) 
                     \big/\nabla|_\rho \Gamma(U, \Omega^j).
\end{equation*}
We now see that, for each $j \in E_{\geq 2}$, the coefficient of $p_j$ in $\nabla$ defines an element of this twisted de Rham cohomology which is independent of the choice of quasi-canonical form $\nabla$. In other words, we have the following analogue of Theorem \ref{thm: ct}.

\begin{proposition} \label{prop: coeff pj}
Let $U \subset \Sigma$ be open and $\nabla$ be any quasi-canonical form of an $\lg$-oper $[\nabla] \in \Op_{\lg}(U)$. For each $j\in E_{\geq 2}$, the coefficient of $p_j$ belongs to $H^1(U, \Omega^j, \nabla|_\rho)$. The coefficient of $p_1$ belongs to $\Gamma(U, \Omega^2)$. \qed
\end{proposition}

Combining this with Theorem \ref{thm: ct} we arrive at the following.

\begin{theorem}\label{thm: space of opers}
For any open subset $U \subset \Sigma$, the space $\Op_{\lg}(U)$ fibres over $\Conn(U, \Omega)$ and we have the isomorphism
\begin{equation*}
\Op_{\lg}(U)^{\overline{\nabla}} \simeq \Gamma(U,\Omega^2) \times
\prod_{j\in E_{\geq 2}} H^1(U,\Omega^j, \overline{\nabla})
\end{equation*}
for the fibre over any connection $\overline{\nabla} \in \Conn(U, \Omega)$. \qed
\end{theorem}

\begin{remark}\label{rem: sqo} 
Recall Corollary \ref{cor: v1}. One can also define the sheaf $\Op_{\lg/\CC\delta}$ of $(\lg/\CC\delta)$-opers over $\Sigma$, and one has the analogue of the above theorem, with \be \Op_{\lg/\CC\delta}(U)^{\overline{\nabla}} \simeq 
\prod_{j\in E} H^1(U,\Omega^j, \overline{\nabla})\nn\ee 
for every $\overline{\nabla} \in \Conn(U, \Omega)$. 
\end{remark}

In the present paper, our main interest lies in $\lg$-opers over $\CP$. What we need is the analogous statement to Theorem \ref{thm: space of opers} for the space $\Op_{\lg}^\reg(\CP)_X^\varphi$ of global meromorphic $\lg$-opers which are holomorphic on the complement $X = \CC \setminus \{ z_i\}_{i=1}^N$ of the set of marked points, as defined in \S\ref{sec: regular points}.

For every $j \in E$, let us denote by $\Gamma_X(\CP,\Omega^{j+1})$ the space of global meromorphic sections of $\Omega^{j+1}$ that are holomorphic on $X$. Let $H^1_X(\CP,\Omega^j, \nabla|_\rho)$ be the corresponding twisted cohomologies. Also let $\Conn_X(\CP, \Omega)$ denote the space of global meromorphic connections on $\Omega$ which are holomorphic on $X$.
 
\begin{theorem}\label{thm: space of opers on CP1}
$\Op^\reg_{\lg}(\CP)_X$ fibres over $\Conn_X(\CP, \Omega)$ and we have the isomorphism
\begin{equation*}
\Op^\reg_{\lg}(\CP)_X^{\overline{\nabla}} \simeq \Gamma_X(\CP,\Omega^2) \times \prod_{j\in E_{\geq 2}} H^1_X(\CP,\Omega^j, \overline{\nabla})
\end{equation*}
for the fibre over any $\overline{\nabla} \in \Conn_X(\CP, \Omega)$. \qed
\end{theorem}

\section{Twisted homology and the integral pairing}\label{sec: twisted homology}
Recall the integrals from Corollary \ref{cor: opint}.
With Theorem \ref{thm: space of opers on CP1} in hand we can give a coordinate-independent description of functions on $\Op_{\lg}^\reg(\CP)_X^{\overline{\nabla}}$.

\subsection{Twisted homology}
Suppose $U\subset X$ is an open subset with a holomorphic coordinate $t:U\to \CC$. In the trivialization defined by the coordinate $t$, a horizontal section of the local system $(\Omega^j, \nabla|_\rho)$, for any $j \in \ZZ$, is given by a holomorphic function $f$ on $U$ such that $df - j {h^\vee}^{-1} \varphi f dt = 0$. Recall the multivalued holomorphic function $\P$ on the complement $X = \CC \setminus \{ z_i \}_{i=1}^N$ from \eqref{def: P}: concretely, $f$ is a constant multiple of a univalued branch of $\P^{j/h^\vee}$ over $U$. 

Recall that a \emph{singular $p$-simplex in $X$} is a continuous map $\sigma: \Delta_p\to X$ from the standard $p$-simplex $\Delta_p$ to $X$. (We shall need only $p\in \{0,1,2\}$.)
Define a \emph{twisted $p$-simplex in $X$ of degree $j$} to be a pair $\sigma \otimes f$ consisting of a singular $p$-simplex $\sigma$ in $X$ together with a horizontal section $f$ of $(\Omega^{j}, \nabla|_\rho)$ over an open neighbourhood of $\sigma(\Delta_p)$ in $X$.
A \emph{twisted $p$-chain of degree $j$} is then a finite formal sum  
\be \gamma = \sum_k \sigma_k \otimes g_k\nn\ee
of twisted $p$-simplices in $X$ of degree $j$.
Let $C_p(X,\Omega^j, \nabla|_\rho)$ be the (infinite-dimensional) complex vector space of twisted $p$-chains in $X$ of degree $j$, where scalar multiplication of chains is by scalar multiplication of the horizontal sections.
Recall that the usual boundary operator sends a singular $p$-simplex $\sigma$ to $\sum_{k=0}^p (-1)^k s_k$ where $s_k$ is the restriction of $\sigma$ to the $k^{\rm th}$ face of $\Delta_p$ (which is canonically identified with $\Delta_{p-1}$). In our twisted setting, the boundary operator $\partial$ is the linear map 
\begin{equation} \label{twisted boundary}
\partial : C_{p}(X, \Omega^j, \nabla|_\rho) \longrightarrow C_{p-1}(X,\Omega^j, \nabla|_\rho),
\end{equation}
defined by $\sigma \otimes f \mapsto \sum_{k=0}^p (-1)^k s_k \otimes f_k$
where $f_k$ is the restriction of $f$ to an open neighbourhood of $s_k(\Delta_{p-1})$ in $X$. 
The property $\partial^2 = 0$ follows from the same property in the usual setting. 
The kernel of the map \eqref{twisted boundary}, which we denote by $Z_{p}(X, \Omega^j, \nabla|_\rho)$, is the space of \emph{closed} twisted $p$-chains of degree $j$. The \emph{twisted homology of $X$} is then the quotient of vector spaces
\begin{equation*}
H_1(X, \Omega^j, \nabla|_\rho) \coloneqq Z_1(X, \Omega^j, \nabla|_\rho) / \partial C_2(X, \Omega^j, \nabla|_\rho).
\end{equation*}
and its elements are \emph{twisted cycles of degree $j$}.

\subsection{Twisted de Rham theorem}

Let $\omega\in \Gamma_X(\CP,\Omega\ox\Omega^j)$ and $(\sigma,f)$ be a twisted $1$-simplex in $X$ of degree $-j$, for any $j \in E$. On an open neighbourhood $U$ of $\sigma(\Delta_1)$ we have the holomorphic $1$-form $f\omega\in \Gamma_X(U,\Omega)$. Define the integral of $\omega$ over $\sigma \otimes f$ to be the usual integral of $f\omega$ over the singular $1$-simplex $\sigma$:
\begin{equation*}
\int_{\sigma \otimes f} \omega \coloneqq \int_\sigma f \omega.
\end{equation*}
Extending by linearity we have the integral $\int_\gamma \omega$ over any $\gamma\in C_1(X,\Omega^{-j}, \nabla|_\rho)$.

In the same way, one defines the integral of a $0$-form $\omega\in \Gamma_X(\CP,\Omega^j)$ over a twisted $0$-chain $\gamma\in C_0(X,\Omega^{-j}, \nabla|_\rho)$; and also in principle of a $2$-form over a twisted $2$-chain, although of course since we are on a curve the only meromorphic $2$-form we have to integrate is the zero $2$-form. 

\begin{proposition}[Twisted Stokes's theorem] 
Let $p \in \{0, 1\}$. For any $p$-form $\omega \in \Gamma_X(\CP,\Omega^{\wx p} \otimes \Omega^j)$ and any twisted $(p+1)$-chain $\gamma\in C_{p+1}(X,\Omega^{-j}, \nabla|_\rho)$,
\begin{equation*}
\int_\gamma \nabla|_\rho \omega = \int_{\partial\gamma} \omega.
\end{equation*}
\end{proposition}
\begin{proof}
By linearity it suffices to consider $\gamma = \sigma \otimes f$ with $\sigma$ a singular $(p+1)$-simplex on $X$ and $f$ a horizontal section of $\Omega^{-j}$ over an open neighbourhood of $\sigma(\Delta_{p+1})$. We then have
\begin{equation*}
\int_\gamma \nabla|_\rho \omega = \int_\sigma f \nabla|_\rho \omega = \int_\sigma d (f \omega) = \int_{\partial \sigma} f\omega 
= \int_{\partial (\sigma \otimes f) } \omega = \int_{\partial \gamma} \omega,
\end{equation*}
where in the second equality we used $d (f \omega) = (\nabla|_\rho f)\omega + f \nabla|_\rho \omega = f \nabla|_\rho \omega$ (since $f$ is horizontal) and in the third equality we used the usual Stokes's theorem.
\end{proof}
\begin{proposition}[Complex de Rham theorem]
There is a bilinear pairing between twisted homologies and cohomologies, given by integrating twisted forms over twisted chains:
\begin{equation*}
H_i(X, \Omega^{-j}, \nabla|_\rho) \times H^i_X(\CP,\Omega^j, \nabla|_\rho) \longrightarrow \CC, \qquad
(\gamma, \omega) \longmapsto \int_\gamma \omega.
\end{equation*}
for $i\in\{0,1\}$ and any $j \in E$.  
\end{proposition}
\begin{proof} The fact that the integral pairing between forms and chains descends to a well-defined pairing between homologies and cohomologies follows from the twisted version of Stokes's theorem above.
\end{proof}

\section{Discussion} \label{sec: discussion}

\subsection{Smooth opers, Drinfel'd-Sokolov and (m)KdV} \label{sec: smooth DS}

The present paper concerned the role of affine opers in describing the spectrum of a (conjectured) family of higher Hamiltonians for affine Gaudin models. Affine Miura opers also play a conceptually quite distinct role in mathematical physics: namely they serve as the phase space of classical generalized mKdV theories.
In the latter context, the procedure of Theorem \ref{thm: quasi-canonical form} for putting an affine oper into quasi-canonical form essentially appears in the paper of Drinfel'd and Sokolov \cite[\S6]{DS}. Specifically, if one replaces meromorphic functions with smooth functions on the circle, and -- crucially -- if one sets the twist function $\varphi$ to zero, then our procedure in \S\ref{sec: quasi-can form} coincides with the procedure of \cite{DS} to construct the densities of Hamiltonians of the classical generalised $\lg$-(m)KdV hierarchy. In what follows we elaborate on this last statement, and contrast the two settings.

In our earlier definition of $\lg$-opers from \S\ref{sec: opers} and of Miura $\lg$-opers from \S\ref{sec: class of Miura opers}, one can replace the algebra $\M$ of meromorphic functions on $\CP$ by the algebra $C^\8(S^1, \CC)$ of smooth functions on the circle. On doing so, we obtain the spaces of smooth $\lg$-opers $\Op_{\lg}(S^1)$ and Miura $\lg$-opers $\MOp_{\lg}(S^1)$ on $S^1$. Furthermore, one may also consider the spaces of smooth $\lg/\CC \delta$-opers $\Op_{\lg/\CC \delta}(S^1)$ and Miura $\lg/\CC \delta$-opers $\MOp_{\lg/\CC \delta}(S^1)$, cf. Corollary \ref{cor: v1} and \S\ref{sec: quad Ham}.
The phase space of $\lg$-mKdV can then be identified with the set $\MOp_{\lg/\CC \delta}(S^1)^0$ of Miura $\lg/\CC \delta$-opers with zero twist function $\varphi = 0$. Indeed, let $\sigma: S^1 \to (0,2\pi)$ denote the natural coordinate on the circle  $S^1 = \RR/ 2\pi\ZZ$. Then a connection $\nabla \in \MOp_{\lg/\CC \delta}(S^1)^0$ takes the form
\begin{equation} \label{mkdv}
\nabla = d + p_{-1}d\sigma + \sum_{i=1}^\ell u_i(\sigma) \alpha_i d\sigma
\end{equation}
where $u_i(\sigma) \in C^\infty(S^1, \CC)$ are smooth functions on the circle, the \emph{classical $\lg$-mKdV fields}. Recalling that $\lh'$ is the span of the simple roots $\{ \alpha_i \}_{i=0}^\ell$, here we implicitly identify the quotient $\lh'/\CC \delta$ with the span of the subset $\{ \alpha_i \}_{i=1}^\ell$.

To go from mKdV to KdV we first need some definitions. Let $\overline{\ln}_+$ be the finite-dimensional nilpotent Lie subalgebra of $\ln_+$ generated by $\check e_i$, $i=1, \ldots, \ell$. We may form the infinite-dimensional nilpotent subalgebra $\overline{\ln}_+(C^\infty(S^1, \CC))$ of the Lie algebra $\lhn_+(C^\infty(S^1, \CC))$ defined as a completion of $\ln_+ \otimes C^\infty(S^1, \CC)$ as in \S\ref{sec: inverse limits}. The Baker-Campbell-Hausdorff formula \eqref{BCH formula} then endows the vector space $\overline{\ln}_+(C^\infty(S^1, \CC))$ with the structure of a group which we denote by $\overline{\lN}_+(C^\infty(S^1, \CC))$. This is a subgroup of $\lhN_+(C^\infty(S^1, \CC))$ defined just as in \S\ref{sec: group lhN} but with $\M$ replaced by $C^\infty(S^1, \CC)$.

The canonical map $\MOp_{\lg/\CC\delta}(S^1)^0 \rightarrow \Op_{\lg/\CC\delta}(S^1)^0, \nabla \mapsto [\nabla]$ factors through
\begin{equation} \label{mKdV to KdV}
\MOp_{\lg/\CC\delta}(S^1)^0 \longhookrightarrow \op_{\lg/\CC\delta}(S^1)^0 \longtwoheadrightarrow \mathscr M \coloneqq \op_{\lg/\CC\delta}(S^1)^0 \big/ \overline{\lN}_+(C^\infty(S^1, \CC)).
\end{equation}
The phase space of $\lg$-KdV is by definition the quotient space $\mathscr M$. Consider now the connection $\nabla \in \MOp_{\lg/\CC\delta}(S^1)^0$ given in \eqref{mkdv} which we write as
\begin{equation*}
\nabla = d + \bar p_{-1}d\sigma + \check f_0 d\sigma + \sum_{i=1}^\ell u_i(\sigma) \alpha_i d\sigma \in \op_{\lg}(S^1)^0
\end{equation*}
with $\bar p_{-1} = \sum_{i=1}^\ell \check f_i$.
Since $[\check e_i, \check f_0] = 0$ for $i = 1, \ldots, \ell$ it follows that $\check f_0$ is invariant under the adjoint action of $\overline{\lN}_+(C^\infty(S^1, \CC))$ on $\lhg(C^\infty(S^1, \CC))$. Therefore, exactly as in the finite-dimensional setting recalled in \S\ref{fto}, the image of $\nabla$ under the map \eqref{mKdV to KdV} has a unique representative of the form 
\begin{equation} \label{KdV connection}
d + p_{-1} d\sigma + \sum_{i\in \bar E} v_i(\sigma) \bar p_i d\sigma.
\end{equation}
Here $\bar p_1$ denotes the unique element in $\overline{\ln}_+$ such that $\{ \bar p_{-1}, 2 \rho - 2h^\vee \Lambda_0, \bar p_1 \} \subset \lg/\CC \delta$ forms an $\sl_2$-triple, and the $\bar p_i$, $i \in \bar E$ span the kernel of $\ad \bar p_1 : \overline{\ln}_+ \to \overline{\ln}_+$.
The smooth functions $v_i(\sigma)\in C^\8(S^1,\CC)$ are the \emph{classical $\lg$-KdV fields}. 

Now Theorem \ref{thm: quasi-canonical form}, and in particular also Corollary \ref{cor: v1}, generalises to the smooth setting. Therefore, starting from the smooth Miura $\lg/\CC\delta$-oper $\nabla$ in \eqref{mkdv} we obtain a quasi-canonical form
\begin{equation} \label{mKdV can form}
d + p_{-1} d\sigma + \sum_{i\in E} h_i(\sigma) p_i d\sigma
\end{equation}
of the underlying smooth $\lg/\CC \delta$-oper $[\nabla] \in \Op_{\lg/\CC \delta}(S^1)^0$.
Let $g \in \lhN_+(C^\infty(S^1, \CC))$ denote the gauge transformation parameter sending the Miura $\lg/\CC\delta$-oper \eqref{mkdv} to the quasi-canonical form \eqref{mKdV can form}. For any $i \in E$, the \emph{$i^{\rm th}$ mKdV flow} is given by
\begin{equation} \label{mKdV flow}
\frac{\partial}{\partial t_i} \nabla_{\partial_\sigma} = \big[ (g^{-1} p_{-i} g)_+, \nabla_{\partial_\sigma} \big] 
\end{equation}
where $\lg \mapsto \ln_+$, $X \mapsto X_+$ is the canonical projection onto $\ln_+$ relative to the Cartan decomposition of $\lg$. Note that since $p_{-i}$ commutes with $\a_{\geq 2}$ in the quotient $\lg/\CC\delta$, the expression $g^{-1} p_{-i} g$ does not depend on the ambiguity in the quasi-canonical form described in Theorem \ref{thm: quasi-canonical form}. Now the right hand side of \eqref{mKdV flow} takes values in $\lh'/\CC\delta$ by \cite[Lemma 6.7]{DS} so that equation \eqref{mKdV flow} indeed defines a flow, for each $i \in E$, on the phase space $\MOp_{\lg/\CC\delta}(S^1)^0$ of $\lg$-mKdV. Furthermore, these flows are mutually commuting \cite[Proposition 6.5]{DS} and define the \emph{$\lg$-mKdV hierarchy}. Similarly, one can also define commuting flows on the phase space $\mathscr M$ of $\lg$-KdV \cite[\S 6.2]{DS}, giving rise to the \emph{$\lg$-KdV hierarchy}.

The smooth functions $h_i(\sigma)\in C^\8(S^1,\CC)$ appearing in the quasi-canonical form \eqref{mKdV can form} are the \emph{densities of the $\lg$-(m)KdV Hamiltonians}.
Indeed, as in Theorem \ref{thm: quasi-canonical form}, the $h_i$ are defined up to the addition of exact derivatives (now not twisted, since $\varphi = 0$). To get gauge-invariant functions we should integrate them over a cycle, and on the circle that leaves only one possibility. Thus, we obtain the following functions on the phase space $\MOp_{\lg/\CC\delta}(S^1)^0$ of $\lg$-mKdV, or alternatively on the phase space $\mathscr M$ of $\lg$-KdV:
\begin{equation*}
\mathscr H_i \coloneqq \int_0^{2 \pi} h_i(\sigma) d\sigma, \qquad i\in E.
\end{equation*}
These are the \emph{$\lg$-(m)KdV Hamiltonians}. They are conserved quantities under the flows of both the $\lg$-mKdV hierarchy \eqref{mKdV flow} and the $\lg$-KdV hierarchy \cite[Proposition 6.6]{DS}. Moreover, they generate these flows with respect to certain Poisson brackets on $\MOp_{\lg/\CC\delta}(S^1)^0$ and $\mathscr M$, respectively \cite[Propositions 6.11 and 6.10]{DS}.

It is interesting to note that the equation \eqref{mKdV flow} also defines a flow on $\MOp_{\lg/\CC\delta}(S^1)^\varphi$ for any choice of smooth function $\varphi \in C^\infty(S^1, \CC)$ since $\rho \not\in \lgp$. In fact, one could define $\lg$-mKdV flows in the analytic setting of the present paper with non-zero twist function $\varphi$. In the case  $\varphi = 0$, the $\lg$-(m)KdV flows have previously been discussed in the analytic setting in \cite{MR3239138} when $\lg = \widehat{\mathfrak{sl}}_N$ and in \cite{MR3297115} for $\lg$ of type $\null^2 A_2$.

\subsection{$\lg$-(m)KdV on polygonal domains}

As recalled in \S\ref{sec: smooth DS}, the affine space of Miura $\lg$-opers $\MOp_{\lg}(\CP)^\varphi$ considered in the present article does not quite correspond to the phase space of $\lg$-mKdV. Indeed, for us the twist function $\varphi$ plays a central role in characterising the affine Gaudin model, just as in the classical case \cite{V17}.

In the present section we show that the simple class of rational Miura $\lg$-opers in $\MOp_{\lg}(\CP)^\varphi$ introduced in \S\ref{sec: class of Miura opers} can, nevertheless, alternatively be described in terms of Miura $\lg$-opers of the mKdV-type as in \eqref{mkdv}, \emph{i.e.} with zero twist function. The price to pay, however, is that one then needs to work with multivalued Miura $\lg$-opers defined over a polygonal region of the complex plane (whose shape now encodes the same information as the twist function did). For this reason we prefer to keep the twist function explicit and work with the space $\MOp_{\lg}(\CP)^\varphi$.

Let us remark, in passing, that the situation described below is very reminiscent of the relation between the modified sinh-Gordon equation and the sinh-Gordon equation in the context of the massive ODE/IM correspondence \cite{Lukyanov:2010rn}.

Recall the class of Miura $\lg$-opers of the form \eqref{u Miura op def} introduced in \S\ref{sec: class of Miura opers}. We consider the analogous class of Miura $\lg/\CC\delta$-oper of the form
\begin{subequations} \label{Miura rho explicit}
\begin{equation} \label{Miura rho explicit a}
\nabla = d + p_{-1} dz + \widetilde u(z) dz - \frac{\varphi(z)}{h^\vee} \rho\, dz,
\end{equation}
where $\widetilde u \in (\lh'/\CC\delta)(\M)$ is the rational function valued in $\lh'/\CC\delta$ defined as
\begin{equation} \label{Miura rho explicit b}
\widetilde u(z) \coloneqq - \sum_{i=1}^N \frac{\dot{\lambda}_i}{z - z_i} dz + \sum_{i = 1}^\ell \Bigg( \sum_{j=1}^{m_i} \frac{1}{z - w^i_j} - \sum_{j=1}^{m_0} \frac{1}{z - w^0_j} \Bigg) \alpha_i dz.
\end{equation}
\end{subequations}
Here, by comparison with the expression \eqref{u Miura op def}, we have split the sum over the simple poles at the Bethe roots into separate sums over the collection of Bethe roots $w^i_j$, $j = 1, \ldots, m_i$ of the same colour $i \in I = \{ 0, \ldots, \ell \}$. We also implicitly identify the subspace $\lh'/\CC\delta$ of $ \lgp/\CC\delta$ with the span of the simple roots $\{ \alpha_i \}_{i=1}^\ell$ and the subspaces $\lg_n/\CC\delta$ for $n \neq 0$ with $\lg_n$, as we did in \S\ref{sec: smooth DS}.

Recall from \S\ref{sec: twisted homology coord} that $\varphi(z) = \partial_z \log \P(z)$. Fix a collection of cuts $C \subset \CP$ between the branch points $z_i$ of the multivalued function $\P^{1/h^\vee}$ on the Riemann sphere $\CP$ and let $\overline{\M}$ denote the field of meromorphic functions on $\CP \setminus C$. From now on we fix a branch of $\P^{1/h^\vee}$ which by abuse of notation we also denote by $\P^{1/h^\vee} \in \overline{\M}$.

By treating \eqref{Miura rho explicit} as a Miura $\lg/ \CC\delta$-oper on $\CP \setminus C$ and working over the larger field $\overline{\M} \supset \M$ one can then remove the $\rho$ term by performing a gauge transformation by $\P(z)^{- \rho/h^\vee} \in \lH(\overline{\M})$. Using the second relation in \eqref{Lie alg a com rel} with $n = -1$ we find
\begin{equation} \label{Miura oper no rho}
\widetilde \nabla \coloneqq \P(z)^{-\rho/h^\vee} \nabla \P(z)^{\rho/h^\vee} = d + p_{-1} \P(z)^{1/h^\vee} dz + \widetilde u(z) dz.
\end{equation}
In order to bring $\widetilde \nabla$ back to the form of a Miura $\lgp/\CC\delta$-oper, consider the new variable $x$ defined as the indefinite integral
\begin{equation} \label{SC mapping}
x \coloneqq \int \P(z)^{1/h^\vee} dz = \int \prod_{i=1}^N (z - z_i)^{k_i / h^\vee} dz,
\end{equation}
where in the second equality we used the explicit form \eqref{def: P} of the function $\P$.

Suppose, for simplicity, that all the $z_i$, $i = 1, \ldots, N$ are real. It is then convenient to relabel them so that they are ordered as $z_1 < z_2 < \ldots < z_N$. For generic $k_i \in \CC$, $i = 1, \ldots, N$, a possible choice of cuts $C \subset \CP$ is to take the following union of open intervals along the real axis in the $z$-plane
\begin{equation*}
C = (- \infty, z_1) \cup \bigcup_{i=2}^{N-1} (z_i, z_{i+1}) \cup (z_N, \infty).
\end{equation*}
Let $x_i \in \CP$, for each $i = 1, \ldots, N$, be the image of $z_i$ under the transformation \eqref{SC mapping}, and let $x_{N+1} \in \CP$ denote the image of $z = \infty$. Suppose the $x_i$, $i=1,\ldots,N+1$ are all distinct and the ordered set $(x_i)_{i=1}^N$ describes the adjacent vertices of a simple polygonal domain $P$ in the $x$-plane, where one of the $x_i$'s could be infinite. In this case, the transformation $z \mapsto x$ given in \eqref{SC mapping} defines a Schwartz-Christoffel mapping. It is a biholomorphic map $\mathbb H \to P$, \emph{i.e.} a bijective holomorphic map whose inverse is also holomorphic, from the upper-half $\mathbb H \coloneqq \{ z \in \CC \,|\, \Im z \geq 0 \}$ of the $z$-plane to the polygon $P$ in the $x$-plane. Each interval $(z_i, z_{i+1})$ for $i =1, \ldots, N-1$ is sent to the straight edge from $x_i$ to $x_{i+1}$, while the semi-infinite intervals $(-\infty, z_1)$ and $(z_N, \infty)$ are sent to the edges connecting $x_{N+1}$ to $x_1$ and $x_N$, respectively. The interior angles $\alpha_j$ of $P$ at each of its vertices $x_j$, for $j =1, \ldots, N+1$, are given by
\begin{equation*}
\alpha_i = \frac{k_i + h^\vee}{h^\vee} \pi, \quad \text{for} \;\; i = 1, \ldots, N \qquad \text{and}\quad \alpha_{N+1} = - \frac{\sum_{i=1}^N k_i + h^\vee}{h^\vee} \pi.
\end{equation*}

Furthermore, the transformation $z \mapsto x$ maps the lower-half $\overline{\mathbb H} = \{ z \in \CC \,|\, \Im z \leq 0 \}$ of the $z$-plane to the reflection $P'$ in the $x$-plane of the polygonal domain $P$ through its edge connecting the vertices $x_1$ and $x_2$. The map $z \mapsto x$ in \eqref{SC mapping} therefore sends the Riemann sphere $\CP$, equipped with the global coordinate $z$ on the dense open subset $\CC \subset \CP$, to another copy of the Riemann sphere $\CP$, equipped with the global coordinate $x$ on an open and dense subset identified with the interior of the domain $P \cup P'$ in the $x$-plane. The case $N=2$ is depicted in Figure \ref{fig: polygonal region}. A very similar change of coordinate to this particular example was considered in \cite{Lukyanov:2013wra} in the context of the massive ODE/IM correspondence for the Fateev model.
\begin{figure}[h]
\raisebox{11mm}{\begin{tikzpicture}[scale=.6]
\filldraw (0,0) node [below right=-.5mm]{\scriptsize $z_1$} circle (2pt);
\filldraw (4,0) node [below left=-.5mm]{\scriptsize $z_2$} circle (2pt);
\draw[blue, -stealth', postaction={decorate,decoration={markings,mark=at position .4 with {\arrow{stealth'}}}}] (2,0) .. controls (-2,-2) and (-2,1.25) .. (2,1.25) node[above]{$\gamma$};
\draw[blue, postaction={decorate,decoration={markings,mark=at position .8 with {\arrow{stealth'}}}}] (2,1.25) .. controls (6,1.25) and (6,-2) .. (2,0);
\draw[blue, postaction={decorate,decoration={markings,mark=at position .8 with {\arrow{stealth'}}}}] (2,0) .. controls (-2,2) and (-3,-1.5) .. (2,-1.5);
\draw[blue] (2,-1.5) .. controls (7,-1.5) and (6,2) .. (2,0);
\draw[gray!30!black, decorate, decoration={snake, segment length=2mm, amplitude=.5mm}] (0,0) -- (-3,0);
\draw[gray!30!black, decorate, decoration={snake, segment length=2mm, amplitude=.5mm}] (4,0) -- (7,0);
\end{tikzpicture}}
\qquad\quad
\begin{tikzpicture}
\filldraw[fill=gray!20,draw=gray!50!black] (0,0) coordinate (a) node[anchor=east]{\scriptsize $x_1$}
  -- (2.5,1) coordinate (b) node[anchor=south]{\scriptsize $x_3$}
  -- (3,-1) coordinate (c) node[anchor=west]{\scriptsize $x_2$}
  -- (1.4, -2.3) coordinate (b2) node[anchor=north]{\scriptsize $x'_3$}
  -- cycle;
\draw[gray!50!black] (a) -- (c);
\draw[blue,postaction={decorate,decoration={markings,mark=at position .5 with {\arrow{stealth'}}}}]
				($.3*(b2)-.3*(c)+(c)$) -- ($.6*(b)$);
\draw[blue,postaction={decorate,decoration={markings,mark=at position .5 with {\arrow{stealth'}}}}]
				($.6*(b2)$) -- node[above]{$\gamma$} ($.7*(b2)-.7*(c)+(c)$);
\draw[blue,postaction={decorate,decoration={markings,mark=at position .3 with {\arrow{stealth'}}}}]
				($.7*(b)-.7*(c)+(c)$) -- ($.2*(b2)$);
\draw[blue,postaction={decorate,decoration={markings,mark=at position .2 with {\arrow{stealth'}}}}]
				($.2*(b)$) -- ($.3*(b)-.3*(c)+(c)$);
\draw[blue,dotted] ($.3*(b)-.3*(c)+(c)$) -- ($.3*(b2)-.3*(c)+(c)$);
\draw[blue,dotted] ($.7*(b)-.7*(c)+(c)$) -- ($.7*(b2)-.7*(c)+(c)$);
\draw[blue,dotted] ($.2*(b)$) -- ($.2*(b2)$);
\draw[blue,dotted] ($.6*(b)$) -- ($.6*(b2)$);
\node[scale=.6, rotate=20] at ($(a)!0.5!(b)$) {|};
\node[scale=.6, rotate=110] at ($(a)!0.5!(b2)$) {|};
\node[scale=.6, rotate=100] at ($(c)!0.5!(b)$) {||};
\node[scale=.6, rotate=40] at ($(c)!0.5!(b2)$) {||};
\draw[-stealth'] (2.7,-1.7) node[right]{\small $P'$} to [bend right=10] (1.5,-1);
\draw[-stealth'] (3.1,.4) node[right]{\small $P$} to [bend right=10] (2.2,0);
\end{tikzpicture}
\caption{A pochhammer contour $\gamma$ in the case of two marked points $z_1$, $z_2$ in the $z$-plane (left) and its image in the polygonal region $P \cup P'$ in the $x$-plane (right). The edge $x_i x_3$ is identified with $x_i x'_3$ for $i = 1, 2$.}
\label{fig: polygonal region}
\end{figure}
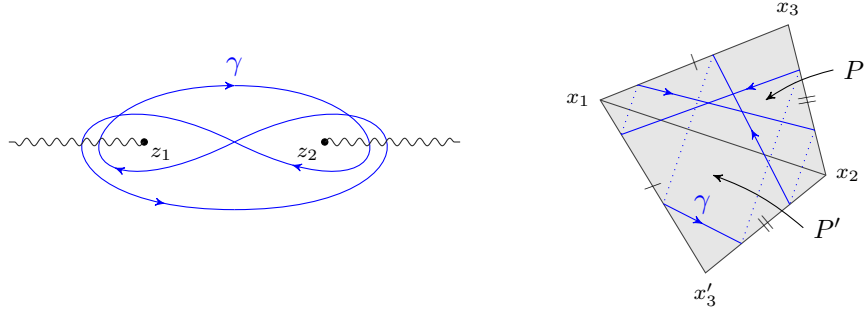

Coming back to the connection \eqref{Miura oper no rho}, the pullback of the meromorphic differential $\widetilde u(z) dz$ by the inverse transformation $\CP \rightarrow \CP$, $x \mapsto z$ gives a multivalued differential $\widehat u(x) dx$, where $\widehat u$ is the multivalued function on $\CP$ given in the interior of the domain $P \cup P'$ of the $x$-plane by $\widehat u(x) \coloneqq \P(z(x))^{-1/h^\vee} \widetilde u(z(x))$. Here we have used the fact that $dx/dz = \P(z)^{1/h^\vee}$.
Therefore \eqref{Miura oper no rho} can now be re-expressed as a multivalued Miura $\lgp/\CC\delta$-oper on $\CP$ which is meromorphic on $\CP \setminus \{ x_i \}_{i=1}^{N+1}$ and given in the interior of the polygonal domain $P \cup P'$ of the $x$-plane by
\begin{equation} \label{polygonal mKdV}
\widetilde \nabla = d + p_{-1} dx + \sum_{i=1}^\ell \widehat u_i(x) \alpha_i dx.
\end{equation}
Here we wrote $\widehat u(x) = \sum_{i=1}^\ell \widehat u_i(x) \alpha_i$ in the basis of simple roots $\{ \alpha_i \}_{i=1}^\ell$. Comparing the above expression \eqref{polygonal mKdV} with the connection \eqref{mkdv}, it is tempting to regard the $\widehat u_i(x)$ for $i = 1, \ldots, \ell$ as classical $\lg$-mKdV fields on the interior of the polygonal domain $P \cup P'$ in the $x$-plane.

Just as in \S\ref{sec: smooth DS}, one could also consider bringing the Miura $\lg$-oper \eqref{polygonal mKdV} to a form analogous to \eqref{KdV connection} in the smooth setting and define classical $\lg$-KdV fields $\widehat v_r(x)$, $r \in E$ on the interior of the polygonal domain $P \cup P'$ as the coefficients of the $\bar p_r$, $r \in E$, \emph{i.e.}
\begin{equation} \label{polygonal KdV}
d + p_{-1} dx + \sum_{r\in \bar E} \widehat v_r(x) \bar p_r dx.
\end{equation}
Suppose that the collection of Bethe roots $w^i_j$, $j = 1, \ldots, m_i$ for $i \in \{ 1, \ldots, \ell \}$ satisfy the Bethe equations \eqref{Bethe equations}, or more explicitly
\begin{equation*}
- \sum_{k=1}^N \frac{(\lambda_k|\alpha_{c(j)})}{w^i_j - z_k} + \sum_{\substack{k =1\\ k \neq j}}^m \frac{(\alpha_{c(k)}|\alpha_{c(j)})}{w^i_j - w^i_k} = 0, \qquad j= 1, \ldots, m_i,
\end{equation*}
for every $i \in \{ 1, \ldots, \ell \}$. Then it follows from Proposition \ref{prop: Miura oper BAE} and the explicit form \eqref{Miura rho explicit} of the connection $\nabla$ we started with in the $z$-plane, that the $\widehat v_r(x)$, $r \in \bar E$ are holomorphic at the images $x(w^i_j)$ of these Bethe roots under the Schwarz-Christoffel transformation \eqref{SC mapping}. However, each $\widehat v_r(x)$, $r \in \bar E$ will generically still be singular at the images $x(w^0_j)$, $j = 1, \ldots, m_0$, of the Bethe roots of ``colour'' 0. (By contrast, let us stress that there exists a quasi-canonical form of the affine $\lg$-oper in which \emph{all} Bethe roots are erased, as in Corollary \ref{cor: Bethe equations}.) 

Multivalued $\lg$-opers of the form \eqref{polygonal KdV} but with a certain irregular singularity were conjectured in \cite{FFsolitons} to describe the spectrum of quantum $\g$-KdV. Specifically, it was shown that in the case $\g = \widehat{\mathfrak{sl}}_2$ such $\lg$-opers are equivalent to the Schr\"odinger operators with `monster' potentials used to describe the spectrum of both local and non-local integrals of motion in quantum KdV theory via the ODE/IM correspondence \cite{MR1733841,BLZ4, BLZ5,DDT}. 

In the ODE/IM setting, it was recently shown in \cite{Frenkel:2016gxg} and \cite{MRV1, MRV2} that a certain $Q\widetilde{Q}$-system can be extracted from both sides of the `KdV-oper' correspondence proposed in \cite{FFsolitons}, providing strong evidence in support of the conjecture.
By contrast, the proposal of the present work is a direct approach to establishing a correspondence between the spectra of quantum Gaudin models of affine type and opers of Langlands-dual affine type. It relies on the idea that, in close parallel with the well-established correspondence in finite types,  the spectrum can be obtained from a (quasi-)canonical form of the (affine) oper.

\subsection{Classical limit of higher Gaudin Hamiltonians} \label{sec: classical lim}

One of the motivations for Conjectures \ref{conj: higher Ham} and \ref{conj: e-val op} comes from the structure of higher local integrals of motion in \emph{classical} affine Gaudin models \cite{V17}, constructed in \cite{LMV17} when the underlying finite-dimensional simple Lie algebra $\dot\g$, cf. \S\ref{sec: ao}, is of classical type.

Specifically, to every exponent $i \in E$ and every zero $x$ of the twist function, \emph{i.e.} an $x \in \CC$ such that $\varphi(x) = 0$, is assigned an element $Q^x_i$ of a certain completion $\hat S_{\bm k}(\g'^{\oplus N})$ of the quotient of the symmetric algebra of $\g'^{\oplus N}$ by the ideal generated by the elements $\cent^{(j)} - k_j$, $j = 1, \ldots, N$. These were obtained by generalising the approach of \cite{Evans:1999mj}, where certain $\dot\g$-invariant homogeneous
polynomials $P_i : \dot\g^{\times (i+1)} \to \CC$ of degree $i+1$ were constructed for each $i \in E$. Extending these to $\g'/\CC \cent \cong_\CC \mathcal L \dot\g$ as $P_i(x_m, \ldots, y_n) \coloneqq P_i(x, \ldots, y) \delta_{m + \ldots + n, 0}$ for any $x, \ldots, y \in \dot\g$ and $m, \ldots, n \in \ZZ$, they can be applied to the first tensor factor of the local Lax matrix \eqref{local Lax matrix}. The resulting $\hat S_{\bm k}(\g'^{\oplus N})$-valued meromorphic functions on $\CP$ are then evaluated at any zero $x$ of the twist function to produce the charges $Q^x_i$, $i \in E$. The collection of these charges was shown in \cite{LMV17} to form a Poisson commutative subalgebra of $\hat S_{\bm k}(\g'^{\oplus N})$, \emph{i.e.} $\{ Q^x_i, Q^{x'}_j \} = 0$ for every $i, j \in E$ and any pair of zeroes $x, x'$ of the twist function.

We also expect the operators $\mathcal S_j(z)$, $j \in E$ in Conjecture \ref{conj: higher Ham} to be built from the local Lax matrix $L(z)$ in a similar way. Furthermore, reintroducing Planck's constant $\hbar$ to take the classical limit, we expect the dependence of the integral \eqref{integrated operators} on $\hbar$ to come in the form
\begin{equation*}
\int_{\gamma} \P(z)^{-i / \hbar h^\vee} \mathcal S_i(z) dz.
\end{equation*}
In the classical limit $\hbar \to 0$ such an integral localises, by the steepest descent method, at the critical points of the function $\P(z)$. Yet these are nothing but the zeroes of the twist function $\varphi(z) = \partial_z \log \P(z)$. Moreover, for generic $z_i$, $i = 1, \ldots, N$ the number of zeroes of the twist function is $N-1$, which coincides with the number of linearly independent cycles from Corollary \ref{cor: opint}.
We thus expect that the higher affine Gaudin Hamiltonians $\hat Q^\gamma_i$ of Conjecture \ref{conj: higher Ham} provide a quantisation of the local integrals of motion $Q^x_i$ for classical affine Gaudin models in \cite{LMV17}.

Strictly speaking, the classical field theories correspond to classical affine Gaudin models with reality conditions and various other generalizations \cite{V17}. To understand the corresponding quantum field theories in the present framework, one would need to extend the constructions above to, in particular, the cyclotomic case and the case of irregular singularities. In finite types, such generalizations of quantum Gaudin models were studied in \cite{VY16,VY17a} and \cite{FFT, VY17} respectively.  

\subsection{Two-point case} One natural arena in which to test the conjectures in \S\ref{sec: conjectures} is the special case of $N=2$ marked points and $\g'=  \widehat{\mathfrak{sl}}_2$. As was noted in \cite[\S6.4]{FFsolitons}, in that case the GKO coset construction \cite{GKO} means that one already has candidates for the higher affine Gaudin Hamiltonians, namely the Integrals of Motion of quantum KdV acting in multiplicity spaces. With that in mind, it is interesting to note that Bethe equations of a two-point Gaudin model of type $\widehat{\mathfrak{sl}}_2$ appeared in \cite[\S7.3]{FJM} as limits of Bethe equations for quantum toroidal algebras.

\appendix
\section{Hyperplane arrangements and quadratic Hamiltonians} \label{sec: hyp arr}

\subsection{The Aomoto complex}
Fix coordinates $z_1,\dots,z_N,w_1,\dots,w_m$ on $\CC^N \times \CC^m$. 
Let $\cc_{N+m}$ denote the hyperplane arrangement in $\CC^N\times \CC^m$ consisting of the following affine hyperplanes:
\begin{alignat}{2}
H_{ij} &: w_i -  w_j = 0 ,\quad &&1\leq i<j \leq m, \nn\\
H_i^j  &: w_i -  z_j = 0, \quad &&1\leq i\leq m,\quad 1\leq j \leq N,\nn\\
H^{ij} &: z_i - z_j = 0  ,\quad &&1\leq i<j \leq N. \nn\end{alignat}

A \emph{weighting} of the hyperplane arrangement $\cc_{N+m}$ means a map $a: \cc_{N+m} \to \CC$, \emph{i.e.} an assignment to every hyperplane $H$ of a number $a(H)\in \CC$ called its \emph{weight}. Suppose we fix such a weighting. The corresponding \emph{master function} is by definition 
\be \Phi \coloneqq \sum_{H\in \cc_{N+m}} a(H) \log l_H ,\nn\ee
where $l_H=0$ is an affine equation for the hyperplane $H$. It is a multivalued function on the complement of the arrangement,
\be U(\cc_{N+m}) \coloneqq (\CC^N \times \CC^m) \setminus \cc_{N+m}. \nn\ee
Let $\mathscr L$ denote the trivial line bundle, over this complement $U(\cc_{N+m})$,  equipped with the flat connection given by
\be d_\Phi f = df + (d\Phi) f ,\nn\ee 
where $f$ is any holomorphic function on $U(\cc_{N+m})$  and $d$ is the de Rham differential. Let $\Omega^\bl(\mathscr L)$ denote the complex of $U(\cc_{N+m})$-sections of the holomorphic de Rham complex of $\mathscr L$. That means, $\Omega^\bl(\mathscr L)$ is isomorphic as a vector space to the usual de Rham complex of holomorphic forms on $U(\cc_{N+m})$, but with the differential 
\be d_\Phi x = dx + d\Phi \wx x .\nn\ee

The \emph{Orlik-Solomon algebra} \cite{OS} $\A^\bl= \bigoplus_{p=0}^m \A^p$ can be defined as the $\CC$-algebra of differential forms generated by $1$ and the one-forms 
\be d\log l_H= d l_H /l_H, \qquad H\in \cc_{N+m}.\nn\ee 
Notice that $dx = 0$ for all $x\in \A^\bl$.
We have $d_\Phi(\A^k) \subset  \A^{k+1}$. The complex $(\A^\bl,d_\Phi)$ is called the \emph{Aomoto complex} or the \emph{complex of hypergeometric differential forms of weight $a$}.
There is an inclusion of complexes
\be (\A^\bl, d_\Phi) \longhookrightarrow (\Omega^\bl(\mathscr L), d_\Phi) .\nn\ee

\subsection{Kac-Moody data}
Let $\g$ be any Kac-Moody Lie algebra with symmetrizable Cartan matrix of size $r$. Let $B$ be the (non-extended) symmetrized Cartan matrix of $\g$. A realization of this Kac-Moody Lie algebra involves the choice of the following data:
\begin{enumerate}[(1)]
\item A finite-dimensional complex vector space $\h$;
\item A non-degenerate symmetric bilinear form $(\cdot,\cdot): \h\times \h\to \CC$;
\item A collection $\alpha_1,\dots,\alpha_\ru\in \h^*$ of linearly independent elements (the \emph{simple roots}) such that $B= \left( (\alpha_i,\alpha_j)\right)_{i,j = 1}^r$.
\end{enumerate}
Let $\tilde\g$ denote ``the Kac-Moody algebra $\g$ but without Serre relations'' -- see \cite[Section 6]{SV} for the precise definition. One has $\tilde \g = \tilde\n_- \oplus \h \oplus \tilde\n_+$, where $\tilde\n_-$ and $\tilde\n_+$ are free Lie algebras in generators $F_i$ and $E_i$, $i=1,\dots,\ru$, respectively.
For any weight $\lambda\in \h^*$ let $M_\lambda$ denote the Verma module over $\tilde \g$ of highest weight $\lambda$ and let $M_{\lambda}^*$ denote its contragredient dual. The \emph{Shapovalov form} is a certain bilinear form defined on $\tilde \g$ and on each weight subspace of $M_{\lambda}$. The quotient $\tilde \g/ \ker S$ is the Kac-Moody algebra $\g$. The quotient $L_\lambda \cong M_\lambda/\ker S$ is the irreducible $\g$ module of highest weight $\lambda$. 
We can regard the Shapovalov form $S$ on $M_\lambda$ as a map $M_\lambda \to M_\lambda^*$. Then we can identify $L_\lambda$ with the image of $M_\lambda$ in $M_\lambda^*$, \emph{i.e.} 
\be L_\lambda = S(M_{\lambda}) \subset M_{\lambda}^*.\nn\ee   

Now fix weights $\lambda_1,\dots,\lambda_N \in \h^*$ and a tuple $(k_1,\dots,k_r) \in \mathbb Z_{\geq 0}^r$ and consider the subspace 
\be \left(\bigotimes_{i=1}^N M^*_{\lambda_i} \right)_{\lambda_\8} \nn\ee
of weight 
\be \lambda_\8 \coloneqq \sum_{i=1}^N \lambda_i - \sum_{j=1}^r k_j \alpha_j.\nn\ee
Set $m = k_1 + \dots + k_r$ and consider the arrangement $\cc_{N+m}$ as above. We assign ``colours'' $1,\dots,\ru$ to the coordinates $w_1,\dots,w_m$ in such a way that $k_i$ of them have colour $i$, for each $i$. Pick any $\kappa\in\Cx$. The inner products amongst the weights $\lambda_i$ and roots $\alpha_i$ defines a weighting of the arrangement $\cc_{N+m}$. Namely:
\begin{equation} \label{weighting}
a(H^{ij}) = (\lambda_i,\lambda_j)/\kappa, \quad
a(H^i_j) = -(\lambda_i,\alpha_{c(j)})/\kappa, \quad
a(H_{ij}) = (\alpha_{c(i)},\alpha_{c(j)})/\kappa,
\end{equation}
where $c(i)$ is the colour of the coordinate $w_i$ for each $i=1, \ldots, m$, cf. \eqref{Master function}.

\subsection{Eigenvectors of the quadratic Gaudin Hamiltonians}
In \cite[Section 7]{SV} Schechtman and Varchenko define a map 
\be \tilde\eta : \left(\bigotimes_{i=1}^N M^*_{\lambda_i} \right)_{\lambda_\8} \longrightarrow  \Omega^m( \mathscr L) \nn\ee
and prove that it obeys
\be d_\Phi \tilde\eta(x) = \frac 1 \kappa \sum_{i<j} \tilde\eta( \Omega_{ij} x) \wx d \log(z_i-z_j) .\label{ki}\ee
Here $\Omega_{ij}$ are certain endomorphisms of $M_{\lambda_i}^* \ox M_{\lambda_j}^*$ whose definition can be found in \cite[Section 7]{SV}. Their important property here is that on the subspace $L_{\lambda_i} \ox L_{\lambda_j}$, $\Omega_{ij}$ coincides with the action of the element $\Xi$ of the Kac-Moody algebra $\g$. (Recall, the element $\Xi$ is the element of the formal sum $\sum_{\alpha} \g_{\alpha} \ox \g_{-\alpha}$ over root spaces, defined by the standard bilinear form on $\g$ of \cite[Chapter 2]{KacBook}.)

The statement \eqref{ki} is Theorem 7.2.5$''$ in \cite{SV}. Note that the differential $d \tilde \eta$ in the statement of Theorem 7.2.5$''$ is the differential in $\Omega^\bl (\mathscr L)$, \emph{i.e.} $d_\Phi\eta$ in our notation. (This is stated explicitly in Theorem 7.2.5$'$.) 

Let 
\be \Theta \in \left(\bigotimes_{i=1}^N M_{\lambda_i} \right)_{\lambda_\8} \ox 
\left(\bigotimes_{i=1}^N M^*_{\lambda_i} \right)_{\lambda_\8}\nn\ee
be the canonical element. Consider the image of $\Theta$ under the map $S \ox \tilde \eta$. The result is an element
\be \Psi \coloneqq (S \ox \tilde \eta)\Theta \in \left(\bigotimes_{i=1}^N L_{\lambda_i} \right)_{\lambda_\8} 
\ox \Omega^m( \mathscr L) \nn\ee
which obeys
\be d_\Phi \Psi = \Bigg(\frac{1}{\kappa} \sum_{i<j} \Xi_{ij} \,d\log(z_i-z_j)\Bigg) \wx \Psi.\label{kii}\ee 
We have the trivial fibre bundle $\pi:\CC^N \times \CC^m\onto \CC^N$ given by projecting along $\CC^m$. 
Now $\Psi$ can be uniquely written in the form
\be \Psi = \psi dw_1\wx \dots \wx dw_m + \psi' ,\nn\ee
where $\psi$ is a $\big(\!\bigotimes_{i=1}^N L_{\lambda_i} \big)_{\lambda_\8}$-valued holomorphic function on the complement $U(\cc_{N+m})$ of the arrangement and where $\psi'$ is a $\big(\!\bigotimes_{i=1}^N L_{\lambda_i} \big)_{\lambda_\8}$-valued holomorphic $m$-form on $U(\cc_{N+m})$ whose pull-back to each fibre vanishes. (That is, informally, $\psi'$ has factors of $dz_i$.) 

The function $\psi$ is the \emph{weight function}, or \emph{Schechtman-Varchenko vector}. 
Now the image of $\tilde\eta$ is in fact in the subspace $\A^m\into  \Omega^m( \mathscr L)$ of hypergeometric forms and these are closed as we noted above. So $d \Psi  = 0$ and hence $d_\Phi \Psi = d\Phi \wx \Psi$. Therefore the identity \eqref{kii} becomes
\begin{equation}
\sum_{i=1}^m \frac{\del \Phi}{\del w_i}  dw_i \wx \psi' + 
\sum_{i=1}^N \frac{\del \Phi}{\del z_i}  dz_i \wx \Psi =  \frac{1}{\kappa}\sum_{i=1}^N \sum_{j\neq i} \frac{\Xi_{ij}}{z_i-z_j} dz_i \wx \Psi .\label{kiii}
\end{equation}
Assume that
\be \frac{\del \Phi}{\del w_i} = 0, \quad\text{for}\quad i=1,\dots,m,\label{critpt}\ee
\emph{i.e.} assume we are at a critical point of the pull-back of the master function to the fibre. On inspecting the weighting of the arrangement in \eqref{weighting}, one sees that equations \eqref{critpt} are the Bethe ansatz equations. 
Now take the inner derivative $\neg \del/ \del z_i$ of \eqref{kiii} and then pull back to the fibre. (That is, consider the coefficient of $dz_i \wx dw_1 \wx \dots \wx dw_m$.) The resulting equality of top-degree forms on the fibre gives
\be \kappa \frac{\del \Phi}{\del z_i}  \psi 
 =  \sum_{j\neq i} \frac{\Xi_{ij}}{z_i-z_j} \psi.\label{evalue eq}\ee
In other words, $\psi$ is an eigenvector of the quadratic Gaudin Hamiltonians, with the eigenvalue $\kappa \frac{\del \Phi}{\del z_i}$. 
It is known that if the critial point \eqref{critpt} is isolated then this eigenvector $\psi$ is nonzero. Indeed, this follows from \cite[Theorem 9.13]{V11}; see, for example, the proof of \cite[Theorem 9.17]{VYhyperplanes}.

\subsection{KZ solutions} 
A related but slightly less direct way to derive the Bethe ansatz for the quadratic Gaudin Hamilitonians starting with \eqref{kii} is to go via asymptotics of solutions of the KZ equations in terms of hypergeometric integrals, as in \cite{RV}. For comparision, and also because hypergeometric integrals appear in a quite different role in the present paper, we now briefly recall this method. 
Let us re-write \eqref{kiii} but now include the term $d \Psi$ (which as we noted above is actually zero). We obtain
\begin{equation*}
\sum_{i=1}^N \frac{\partial \psi}{\partial z_i} dz_i \wx d\bm w + d \psi' + \sum_{i=1}^m \frac{\del \Phi}{\del w_i}  dw_i \wx \psi' + 
    \sum_{i=1}^N \frac{\del \Phi}{\del z_i}  dz_i \wx \Psi
=  \frac{1}{\kappa}\sum_{i=1}^N \sum_{j\neq i} \frac{\Xi_{ij}}{z_i-z_j} dz_i \wx \Psi,
\end{equation*}
where $d\bm w \coloneqq dw_1 \wx \ldots \wx dw_m$.
Now upon taking the inner derivative $\neg \del/\del z_i$ of this equation and pulling back to the fibres of the trivial fibre bundle $\pi:\CC^N \times \CC^m\onto \CC^N$ we obtain the equation
\begin{equation*}
\frac{\partial \psi}{\partial z_i} d\bm w + \frac{\del \Phi}{\del z_i} \psi d\bm w -  \sum_{j=1}^m dw_j \wx \bigg( \frac{\del}{\del w_j} \chi_i +  \frac{\del \Phi}{\del w_j} \chi_i \bigg) =  \frac{1}{\kappa} \sum_{j\neq i} \frac{\Xi_{ij}}{z_i-z_j} \psi d\bm w,
\end{equation*}
where $\chi_i$ denotes the pullback to the fibre of the inner derivative $\neg \del/ \del z_i$ applied to $\psi'$. Next, we multiply this last equation through by the multivalued function $e^{\Phi}$ and integrate over an $m$-dimensional twisted cycle $\Gamma$ in the fibre, along which this function is single-valued. Since the bracketed sum on the left-hand side is the twisted derivative of the function $\chi_i$, it vanishes once we take the integral over $\Gamma$. We therefore end up with
\begin{equation*}
\kappa \frac{\partial}{\partial z_i} \int_\Gamma e^\Phi \psi d\bm w = \sum_{j\neq i} \frac{\Xi_{ij}}{z_i-z_j} \int_\Gamma e^\Phi \psi d\bm w,
\end{equation*}
which expresses the fact that the integral $\int_\Gamma e^\Phi \psi d\bm t$ is a solution of the KZ equation. The eigenvalue equation \eqref{evalue eq} can be recovered in the limit $\kappa \to 0$ by applying the method of steepest descent to this integral solution, which localises in this limit to the isolated zeroes of the Bethe equations \eqref{critpt}. This is to be contrasted with the discussion of the classical limit of the higher affine Gaudin Hamiltonians in \S\ref{sec: classical lim}.

\providecommand{\bysame}{\leavevmode\hbox to3em{\hrulefill}\thinspace}
\providecommand{\MR}{\relax\ifhmode\unskip\space\fi MR }
\providecommand{\MRhref}[2]{%
  \href{http://www.ams.org/mathscinet-getitem?mr=#1}{#2}
}
\providecommand{\href}[2]{#2}

\end{document}